\documentclass[10pt]{amsart}
\usepackage{amsmath,amssymb,amsfonts,amscd,epsfig}
\usepackage{axodraw4j}
\usepackage{pstricks}
\usepackage{color}

\theoremstyle{plain}
\newtheorem{thm}{Theorem}
\newtheorem{lem}[thm]{Lemma}
\newtheorem{con}[thm]{Conjecture}
\newtheorem{cor}[thm]{Corollary}
\newtheorem{prop}[thm]{Proposition}

\newtheorem{remark}[thm]{Remark}
\newtheorem{defn}[thm]{Definition}
\newtheorem{ex}[thm]{Example}

\numberwithin{thm}{section}
\numberwithin{equation}{section}

\newcommand{\x}{\mathsf{x}}
\newcommand{\zz}{{\overline{z}}}
\newcommand{\Li}{\mathrm{Li}}
\newcommand{\dd}{\mathrm{d}}
\newcommand{\sv}{{\mathrm{sv}}}
\font \rus= wncyr10
\newcommand{\sha}{\, \hbox{\rus x} \,}
\newcommand{\Lie}{\mathrm{Lie}}
\newcommand{\To}{\longrightarrow}
\newcommand{\sA}{{\mathcal A}}
\newcommand{\sB}{{\mathcal B}}
\newcommand{\sC}{{\mathcal C}}

\newcommand{\sH}{{\mathcal H}}
\newcommand{\sI}{{\mathcal I}}

\newcommand{\sM}{{\mathcal M}}

\newcommand{\sO}{{\mathcal O}}
\newcommand{\sP}{{\mathcal P}}

\newcommand{\sS}{{\mathcal S}}

\newcommand{\CC}{{\mathbb C}}

\newcommand{\NN}{{\mathbb N}}

\newcommand{\PP}{{\mathbb P}}
\newcommand{\QQ}{{\mathbb Q}}
\newcommand{\RR}{{\mathbb R}}

\newcommand{\ZZ}{{\mathbb Z}}

\title{Graphical functions and single-valued multiple polylogarithms}
\author{Oliver Schnetz}
\begin{document}
\begin{abstract}
Graphical functions are single-valued complex functions which arise from Feynman amplitudes. We study their properties and use their connection to
multiple polylogarithms to calculate Feynman periods. For the zig-zag and two more families of $\phi^4$ periods we give exact results modulo
products. These periods are proved to be expressible as integer linear combinations of single-valued multiple polylogarithms evaluated at one.
For the larger family of `constructible' graphs we give an algorithm that allows one to calculate their periods by computer algebra.
The theory of graphical functions is used in \cite{ZZ} to prove the zig-zag conjecture.
\end{abstract}
\maketitle
\section{Introduction}
\subsection{Feynman periods}
In four dimensional $\phi^4$ theory the period map assigns positive real numbers to 4-regular\footnote{A graph is 4-regular if every vertex has four edges.} internally 6-connected\footnote{
A 4-regular graph is internally 6-connected if the only way to split the graph with four edge cuts is by separating off a vertex.} graphs \cite{SchnetzCensus}. The periods determine the
contributions of primitive logarithmically divergent graphs to the beta function of the underlying quantum field theory.

Although periods are originally associated to sub-divergence free four-point\linebreak[4] graphs,
it is convenient to complete the graph by adding an extra vertex. This vertex, henceforth labeled $\infty$, is glued to the four external (half-)edges of the graph (see figure 1).
As a remnant of conformal symmetry, graphs with the same completion have identical period. By choice of $\infty$ one can consider the completed graph as an equivalence
class of four-point graphs with the same period.

\begin{center}
\fcolorbox{white}{white}{
  \begin{picture}(336,75) (24,-38)
    \SetColor{Black}
    \SetWidth{0.8}
    \Vertex(53,33){2.8}
    \Vertex(64,-13){2.8}
    \Vertex(40,-13){2.8}
    \Vertex(76,5){2.8}
    \Vertex(32,24){2.8}
    \Vertex(71,24){2.8}
    \Vertex(28,3){2.8}
    \Arc(52,9)(24.413,145,505)
    \Arc(86.871,38.557)(34.324,-170.683,-108.465)
    \Arc(85.63,2.475)(26.026,124.203,213.79)
    \Arc(18.357,-68.443)(71.964,51.658,83.102)
    \Arc(15.214,0.679)(28.734,-26.183,54.255)
    \Arc(10.5,43)(43.661,-66.371,-13.241)
    \Arc(51,61.5)(42.5,-118.072,-61.928)
    \Arc(73.318,-34.136)(39.172,87.539,145.591)
    \Text(44,-35)[lb]{\normalsize{\Black{$\overline{Z_5}$}}}

    \Vertex(133,33){2.8}
    \Vertex(132,-16){2.8}
    \Vertex(158,9){2.8}
    \Vertex(108,9){2.8}
    \Vertex(150,26){2.8}
    \Vertex(151,-8){2.8}
    \Vertex(115,-8){2.8}
    \Vertex(116,26){2.8}
    \Arc(133,9)(24.413,145,505)
    \Arc(162.86,39.265)(30.51,-168.151,-97.268)
    \Arc(173.5,8.5)(28.504,142.125,217.875)
    \Arc(168.6,-28.567)(39.034,105.757,159.662)
    \Arc(131.231,-54.462)(50.493,66.95,109.953)
    \Arc(97.667,10)(24.333,-41.112,41.112)
    \Arc[clock](106.1,-16.8)(25.962,83.587,3.976)
    \Arc(97.4,46.567)(39.034,-74.243,-20.338)
    \Arc(133.405,57.177)(36.204,-120.555,-62.718)
    \Text(129,-35)[lb]{\normalsize{\Black{$\overline{Z_6}$}}}

    \Vertex(191,-17){2.8}
    \Vertex(207,6){2.8}
    \Vertex(214,-17){2.8}
    \Vertex(231,6){2.8}
    \Vertex(236,-17){2.8}
    \Vertex(252,6){2.8}
    \Line(191,-17)(236,-17)
    \Line(191,-17)(207,6)
    \Line(207,6)(214,-17)
    \Line(214,-17)(231,6)
    \Line(231,6)(236,-17)
    \Line(236,-17)(252,6)
    \Line(252,6)(207,6)
    \Arc[clock](221.5,-5.5)(33.534,-159.944,-339.944)
    \Text(211,-35)[lb]{\normalsize{\Black{$Z_5$}}}

    \Vertex(284,-17){2.8}
    \Vertex(299,6){2.8}
    \Vertex(308,-17){2.8}
    \Vertex(321,6){2.8}
    \Vertex(332,-17){2.8}
    \Vertex(345,6){2.8}
    \Vertex(355,-17){2.8}
    \Line(284,-17)(355,-17)
    \Line(284,-17)(299,6)
    \Line(299,6)(345,6)
    \Line(299,6)(308,-17)
    \Line(308,-17)(321,6)
    \Line(321,6)(332,-17)
    \Line(332,-17)(345,6)
    \Line(345,6)(355,-17)
    \Arc[clock](319.375,-7.6)(36.602,-165.119,-373.267)
    \Text(315,-35)[lb]{\normalsize{\Black{$Z_6$}}}
  \end{picture}
}
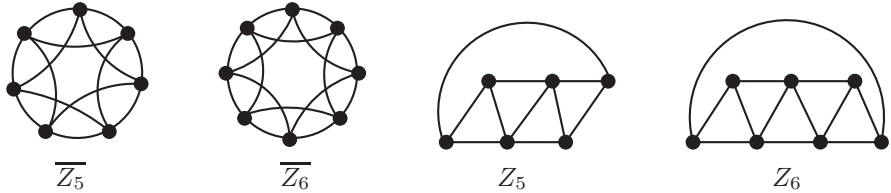
Figure 1: Completed ($\overline{Z_\bullet}$) and uncompleted ($Z_\bullet$) zig-zag graphs with five and six loops.
\end{center}
\vskip2ex

The structure of $\phi^4$ periods was first studied systematically by Broadhurst and Kreimer in 1995 \cite{BK} (see also \cite{BK1}). They found by exact numerical methods that up to
seven loops (the number of independent cycles in the four-point graph) many periods are multiple zeta values (MZVs), i.e.\ rational linear combinations of multiple zeta sums
\begin{equation}\label{MZV}
\zeta(n_1,n_2,\ldots,n_r)=\sum_{1\leq k_1<k_2<\ldots<k_r}\frac{1}{k_1^{n_1}k_2^{n_2}\cdots k_r^{n_r}},\quad\hbox{with}\quad n_r\geq2.
\end{equation}
Multiple zeta sums (\ref{MZV}) span a $\QQ$ vector space $\sH$ which is conjectured to be graded by the weight $n=n_1+n_2+\ldots+n_r$.

There are four equivalent ways to define the period of a graph: position space, momentum space, parametric space, and dual parametric space \cite{SchnetzCensus}.
In the context of algebraic geometry one often uses Feynman (or Schwinger) parameters \cite{BEK}, \cite{BrFeyn}. In this article, however, it is essential to use position space.

\begin{center}
\fcolorbox{white}{white}{
  \begin{picture}(336,75) (44,-26)
    \SetWidth{0.8}
    \SetColor{Black}
    \Line(48,35)(96,35)
    \Line(48,-6)(96,-6)
    \Vertex(48,35){2.8}
    \Vertex(48,-6){2.8}
    \Vertex(96,-6){2.8}
    \Vertex(96,35){2.8}
    \Text(44,20)[lb]{\normalsize{\Black{$x$}}}
    \Text(93,20)[lb]{\normalsize{\Black{$\infty$}}}
    \Text(118,31)[lb]{\normalsize{\Black{$Q_e=1$}}}
    \Text(44,-22)[lb]{\normalsize{\Black{$x$}}}
    \Text(94,-22)[lb]{\normalsize{\Black{$1$}}}
    \Text(118,-13)[lb]{\normalsize{\Black{$Q_e=||x-e_1||^2$}}}
    \Line(228,35)(276,35)
    \Line(228,-6)(276,-6)
    \Vertex(228,35){2.8}
    \Vertex(276,35){2.8}
    \Vertex(228,-6){2.8}
    \Vertex(276,-6){2.8}
    \Text(225,20)[lb]{\normalsize{\Black{$x$}}}
    \Text(274,20)[lb]{\normalsize{\Black{$0$}}}
    \Text(300,31)[lb]{\normalsize{\Black{$Q_e=||x||^2$}}}
    \Text(225,-22)[lb]{\normalsize{\Black{$x$}}}
    \Text(274,-23)[lb]{\normalsize{\Black{$y$}}}
    \Text(300,-13)[lb]{\normalsize{\Black{$Q_e=||x-y||^2$}}}
  \end{picture}
}
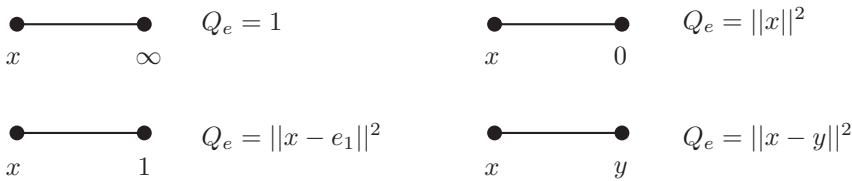
Figure 2: Position space Feynman rules: edges correspond to quadrics (or 1).
\end{center}
\vskip2ex

In a completed Feynman graph we label the $V\geq3$ vertices by $0,$ $1,$ $\infty,$ $x_1,$ $\ldots,$ $x_{V-3}$,
where 0 is the origin of $\RR^4$, 1 corresponds to a unit vector $e_1$, and $x_i\in\RR^4$ for $1\leq i\leq V-3$.
Edges of the Feynman graph correspond to propagators. The propagator is the constant function 1 if the edge has one vertex $\infty$ or if it connects 0 and 1.
Otherwise the propagator of an edge $e$ from $x=(x^1,x^2,x^3,x^4)$ to $y=(y^1,y^2,y^3,y^4)$ is the reciprocal of the quadric
$$
Q_e=||x-y||^2=(x^1-y^1)^2+(x^2-y^2)^2+(x^3-y^3)^2+(x^4-y^4)^2.
$$
The period of the completed Feynman graph $\Gamma$ is then given by the $4(V-3)$-dimensional integral
\begin{equation}\label{Pdef}
P(\Gamma)=\left(\prod_{v\notin\{0,1,\infty\}} \int_{\RR^4}\frac{\dd^4x_v}{\pi^2}\right)\frac{1}{\prod_eQ_e},
\end{equation}
where the products are over vertices $v$ and edges $e$ of $\Gamma$.

Because edges that connect to infinity do not contribute to the integral we can remove infinity before we apply the Feynman rules. It is often useful to generalize
the definition of a period to non-$\phi^4$ graphs. The Fourier identity \cite{BK}, for example, maps $\phi^4$ graphs to non-$\phi^4$ graphs without changing the period. This map will
be used to link the period of the zig-zag graphs to the sequential family of graphs defined in the next subsection. Although completion
is also possible for non-$\phi^4$ graphs if one introduces edges of negative weights (similar to \S\ref{completion}) we use non-$\phi^4$ graphs in an uncompleted form. In this form
graphs have no vertex $\infty$ and the definition of the period (\ref{Pdef}) remains unchanged. We use capital Greek letters for completed and Latin letters for uncompleted graphs.

Position space Feynman rules were used in \cite{BK}, \cite{BK1}, and in most calculations of Feynman periods. In fact, until 2012 only eight $\phi^4$ periods have been calculated:
The trivial period 1 ($=P(Z_1)$), the wheels with three \cite{C4} and with four \cite{C5} spokes which are $Z_3$ and $Z_4$, respectively.\footnote{Wheels with
any number of spokes can be calculated \cite{B2}. However, wheels with more than four spokes do not exist in $\phi^4$ theory.}
The period of the zig-zag graph (see figure 1) with 5 loops was calculated by Kazakov in 1983 \cite{Kazakov}.
The 6 loop zig-zag period was derived by Broadhurst in 1985 \cite{B1} and confirmed by Ussyukina in 1991 \cite{Ussy}.
Until 2012 the only calculated non zig-zag periods were $G(3,1,0)$ in \cite{BK} which is $P_{6,2}$ in \cite{SchnetzCensus}, $P_{6,3}$ by E. Panzer, and the
bipartite graph $K_{4,4}$ in \cite{S3} which is $P_{6,4}$ in \cite{SchnetzCensus}. By exact numerical methods 24 more periods were determined up to eight loops
\cite{SchnetzCensus}, \cite{B3} all of which are MZVs.

Recent methods, partially based on the theory of graphical functions, allows one to calculate several hundred distinct periods up to eleven loops \cite{coact}.

The zig-zag periods were conjectured to all orders in \cite{BK}:
\begin{con}\label{zzcon} (Zig-zag conjecture). The period of the graph $Z_n$ is given by
\begin{equation} \label{PZ}
P(Z_n) = 4\frac{(2n-2)!}{n!(n-1)!}\Big(1-\frac{1-(-1)^n}{2^{2n-3}}\Big)\zeta(2n-3).
\end{equation}
\end{con}
In spite of its seeming simplicity the zig-zag conjecture remained open for 17 years.

Until 2012 all known $\phi^4$ periods were MZVs. Recently $\phi^4$ periods with extensions of MZVs by second and sixth roots of unity were calculated \cite{coact}.
Assuming transcendentality conjectures these periods are not MZVs.
There exists strong mathematical evidence that $\phi^4$ periods in general are of an even more general type \cite{K3}, \cite{BD}.

In this article we develop a method that allows one to calculate the zig-zag periods and two more families of $\phi^4$ periods (and `sequential' non-$\phi^4$ periods).
Modulo products of MZVs an explicit formula for these periods is given. For a more general class of `constructible' periods we present a computer algorithm that works up to eleven loops.

Finally, the zig-zag conjecture is proved in \cite{ZZ} using corollary \ref{maincor}.

\subsection{Sequential graphs}
An important family of graphs can be encoded by words in the three letter alphabet $0,1,2$.

\begin{center}
\fcolorbox{white}{white}{
  \begin{picture}(336,120) (28,-19)
    \SetWidth{0.8}
    \SetColor{Black}
    \Vertex(194,91){2.8}
    \Vertex(194,-5){2.8}
    \Vertex(82,43){2.8}
    \Vertex(114,43){2.8}
    \Vertex(146,43){2.8}
    \Vertex(178,43){2.8}
    \Vertex(210,43){2.8}
    \Vertex(242,43){2.8}
    \Vertex(274,43){2.8}
    \Vertex(306,43){2.8}
    \Line(82,43)(194,91)
    \Line(82,43)(194,-5)
    \Line(194,91)(114,43)
    \Line(194,91)(146,43)
    \Line(194,91)(178,43)
    \Line(178,43)(194,-5)
    \Line(210,43)(194,-5)
    \Line(242,43)(194,-5)
    \Line(194,91)(274,43)
    \Line(194,-5)(306,43)
    \Line(194,91)(306,43)
    \Line(82,43)(306,43)
    \Text(208,91)[lb]{\normalsize{\Black{$1$}}}
    \Text(208,-15)[lb]{\normalsize{\Black{$0$}}}
  \end{picture}
}
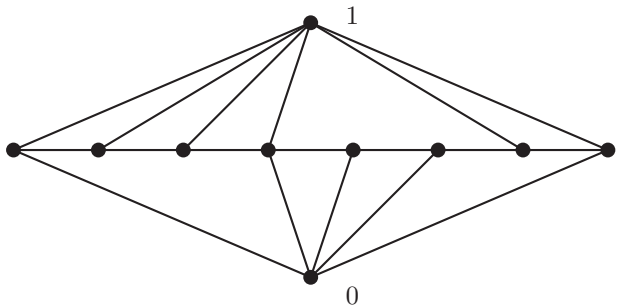
Figure 3: The sequential graph $G_{21120012}$.
\end{center}
\vskip2ex

\begin{defn}
Let $w$ be a word in 0,1,2. The sequential graph $G_w$ is the graph with two distinguished vertices 0 and 1 and a horizontal chain of vertices that connect either to 0, to 1, or to both 0 and 1
(see figure 3). Reading from left to right the connections are encoded in the letters 0, 1, or 2, respectively.
\end{defn}
Sequential graphs have well-defined periods $P(G_w)$ in four dimensions if the word $w$ begins and ends in 2 (corollary \ref{Pwelldef}).

Certain sequential graphs are related by duality to zig-zag graphs. With the notation
\begin{equation}
w^{\{n\}}=\underbrace{ww\ldots w}_{n}
\end{equation}
for the $n$ fold iteration of a word $w$ the graphical dual $Z_n^\star$ of the uncompleted zig-zag graph with $n$ loops is
\begin{eqnarray}
Z_{2m+3}^\star=G_{2(01)^{\{m\}}2}\;\cup\;\{e_{01}\}&\hbox{or}&G_{2(10)^{\{m\}}2}\;\cup\;\{e_{01}\},\nonumber\\
Z_{2m+4}^\star=G_{2(01)^{\{m\}}02}\;\cup\;\{e_{01}\}&\hbox{or}&G_{2(10)^{\{m\}}12}\;\cup\;\{e_{01}\},
\end{eqnarray}
where $e_{01}$ is the edge 01. The period of an uncompleted graph with a planar embedding equals the period of its planar dual \cite{BK}, \cite{SchnetzCensus}.
Hence we can express the zig-zag period in terms of the period of a sequential graph,
\begin{eqnarray}\label{Pzigzag}
P(Z_{2m+3})&=&P(G_{2(01)^{\{m\}}2})\;\;=\;\;P(G_{2(10)^{\{m\}}2}),\nonumber\\
P(Z_{2m+4})&=&P(G_{2(01)^{\{m\}}02})\;\;=\;\;P(G_{2(10)^{\{m\}}12}),
\end{eqnarray}
where we have dropped the edge $e_{01}$ because its propagator 1 does not contribute to the period.

Sequential graphs were independently analysed in \cite{Drummond}.

\subsection{Graphical functions}
A graphical function is the evaluation of a graph $G$ with three distinguished vertices $0,1,z$ with position space Feynman rules and no integration over $z$,
\begin{equation}\label{fdef4}
f_G(z)=\left(\prod_{v\notin\{0,1,z\}} \int_{\RR^4}\frac{\dd^4x_v}{\pi^2}\right)\frac{1}{\prod_eQ_e}.
\end{equation}
Here we assume four dimensions. Arbitrary dimensions greater than two are considered in \S \ref{Gf}.
A special case are graphical functions in two dimensions which we will define in \S \ref{2d}.

By symmetry $f_G$ depends only on two real parameters: the norm $||z||$ and the angle between $z$ and the unit vector $e_1$.
We can hence consider $f_G$ as a function on the complex plane $\CC$ where we identify $e_1$ with 1 and choose any of the two possible orientations.

Graphical functions are single-valued functions which are real analytic in $\CC\backslash\{0,1\}$ \cite{PropGF}.
Independence of the orientation in $\CC$ results in the reflection symmetry
$$f_G(z)=f_G(\zz).$$
Graphical functions also arise as conformal integrals in $N=4$ supersymmetric Yang-Mills theory \cite{SYM}.

Sequential graphs give rise to sequential (graphical) functions by appending a horizontal edge connected to $z$ (see figure 4).
The four dimensional graphical function
\begin{equation}\label{fI}
f_{\mathrm{I}}(z)=\frac{1}{z\zz(z-1)(\zz-1)}
\end{equation}
serves as an initial case for constructing sequential functions.

\begin{center}
\fcolorbox{white}{white}{
  \begin{picture}(336,120) (28,-27)
    \SetWidth{0.8}
    \SetColor{Black}
    \Vertex(40,83){2.8}
    \Vertex(40,35){2.8}
    \Vertex(40,-13){2.8}
    \Line(40,83)(40,-13)
    \Text(50,82)[lb]{\normalsize{\Black{$1$}}}
    \Text(50,33)[lb]{\normalsize{\Black{$z$}}}
    \Text(50,-17)[lb]{\normalsize{\Black{$0$}}}

    \Vertex(194,83){2.8}
    \Vertex(194,-13){2.8}
    \Vertex(82,35){2.8}
    \Vertex(114,35){2.8}
    \Vertex(146,35){2.8}
    \Vertex(178,35){2.8}
    \Vertex(210,35){2.8}
    \Vertex(242,35){2.8}
    \Vertex(274,35){2.8}
    \Vertex(306,35){2.8}
    \Vertex(338,35){2.8}
    \Line(82,35)(194,83)
    \Line(82,35)(194,-13)
    \Line(194,83)(114,35)
    \Line(194,83)(146,35)
    \Line(194,83)(178,35)
    \Line(178,35)(194,-13)
    \Line(210,35)(194,-13)
    \Line(242,35)(194,-13)
    \Line(194,83)(274,35)
    \Line(194,-13)(306,35)
    \Line(194,83)(306,35)
    \Line(82,35)(338,35)
    \Text(210,83)[lb]{\normalsize{\Black{$1$}}}
    \Text(210,-23)[lb]{\normalsize{\Black{$0$}}}
    \Text(337,19)[lb]{\normalsize{\Black{$z$}}}
  \end{picture}
}
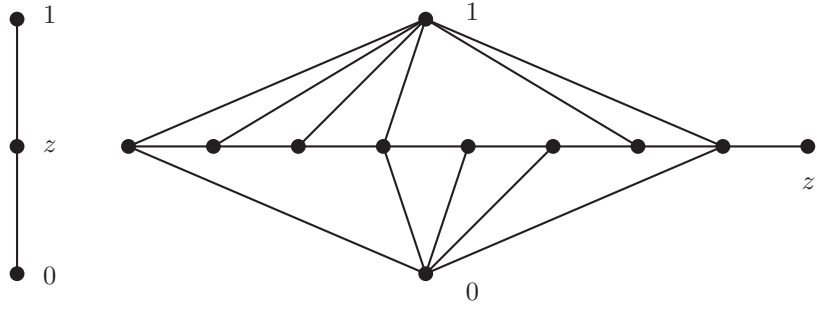
Figure 4: The graphical function $f_{\mathrm{I}}(z)$ and the sequential function $f_{21120012}(z)$. Note that $f_{\mathrm{I}}(z)$ is not a sequential function
because $z$ has `vertical' edges.
\end{center}
\vskip2ex

The sequential function $f_w(z)$ is well-defined in four dimensions if the word $w$ begins with 2 (lemma \ref{convergencelemma}).
Zig-zag periods can be expressed as special values of sequential functions at 0 or 1,
\begin{eqnarray}
P(Z_{2m+3})&\!\!\!=&\!\!\!f_{2(01)^{\{m\}}1}(0)=f_{2(01)^{\{m\}}0}(1)=f_{2(10)^{\{m\}}1}(0)=f_{2(10)^{\{m\}}0}(1),\nonumber\\
P(Z_{2m+4})&\!\!\!=&\!\!\!f_{2(01)^{\{m\}}01}(0)=f_{2(01)^{\{m\}}00}(1)=f_{2(10)^{\{m\}}11}(0)=f_{2(10)^{\{m\}}10}(1).\nonumber\\
&&
\end{eqnarray}

Summarizing the notions of graphical and sequential functions we have
\begin{eqnarray*}
\hbox{graph with vertex }z&\stackrel{\hbox{\small Feynman rules}}{\longrightarrow}&\hbox{graphical function}\\
\hbox{sequential graph with horizontal edge to }z&\stackrel{\hbox{\small Feynman rules}}{\longrightarrow}&\hbox{sequential function.}
\end{eqnarray*}

\subsection{Single-valued multiple polylogarithms}
For any word $w$ in the two letter alphabet $0,1$ we inductively define multiple polylogarithms $L_w(z)$ by
$$\partial_z L_{wa}(z)=\frac{L_{w}(z)}{z-a},\quad\hbox{for $a\in\{0,1\}$},$$
and $L_w(0)=0$ unless $w=0^{\{n\}}$ in which case we have $L_{0^{\{n\}}}(z)=(\ln z)^n/n!$.
Multiple polylogarithms are multi-valued analytic functions on $\CC\backslash\{0,1\}$. The weight of $L_w$ is the length $|w|$ of the word $w$.

By taking appropriate linear combinations of products of multiple polylogarithms with
their complex conjugates one can construct single-valued multiple polylogarithms (SVMPs), see \cite{BrSVMP}, \cite{BrSVMPII}.
SVMPs span a shuffle-algebra over $\CC$ which is graded by the total weight
\begin{equation}
\sP=\bigoplus_{n\geq0}\sP_n.
\end{equation}
For every word $w$ in 0 and 1 of length $n$ there exists a basis element $P_w(z)\in\sP_n$.

By construction the regularized (setting $\ln(0)=0$) limit of SVMPs at 0 vanishes. With the notation
\begin{equation}\label{zetaw}
\zeta_{10^{\{n_1-1\}}\ldots10^{\{n_r-1\}}}=(-1)^r\zeta(n_1,\ldots,n_r)
\end{equation}
we have $L_w(1)=\zeta_w$. Hence at 1 (and at $\infty$) SVMPs evaluate to MZVs. For any ring $R\subseteq\CC$ we set
\begin{equation}
\sH^\sv(R)=\langle P_w(1),\;w\text{ word in 0,1}\rangle_R.
\end{equation}
The $\QQ$ algebra $\sH^\sv=\sH^\sv(\QQ)$ is a proper subalgebra of $\sH$, the $\QQ$ algebra of MZVs.
One can construct $\sH^\sv$ as the largest $\QQ$ sub-algebra of $\sH$ with only odd weight generators on which the Galois coaction on $\sH$ coacts \cite{Bsv}.

By holomorphic and antiholomorphic differentiation $\partial_z$ and $\partial_{\zz}$ SVMPs generate a bi-differential algebra
\begin{equation}\label{defA}
\sA=\CC\left[z,\frac{1}{z},\frac{1}{z-1},\zz,\frac{1}{\zz},\frac{1}{\zz-1}\right]\sP.
\end{equation}
We present results of the theory of SVMPs (many of which are due to F. Brown \cite{BrSVMP}, \cite{BrSVMPII}) in \S \ref{SVMPs}.
In particular there exist integrals with respect to $z$ and $\zz$ in $\sA$,
\begin{equation}
\int_0\dd z,\int_0\dd\zz:\sA\longrightarrow\sA.
\end{equation}
We give an algorithm that allows one to integrate in $\sA$ up to high weights ($\approx$ 30).
A residue theorem in \S \ref{residue} facilitates the integration of functions in $\sA$ over the complex plane.
In practice the integration over the complex plane is more memory and time consuming so that the implementation is limited to smaller weights.

The connection to graphical functions is established by the fact that in many cases graphical functions are expressible in terms of SVMPs in the sense that
\begin{equation}\label{fGinA}
f_G(z)=\frac{g(z)}{z-\zz},
\end{equation}
with $g(z)=-g(\zz)\in\sA$. Let $\sB$ denote the vector space of such $g(z)/(z-\zz)$. Likewise we define $\sB^0\subset\sB$ as the set of functions (\ref{fGinA})
where $g(z)=-g(\zz)\in\sP$. An example is $f_{\mathrm{I}}$, (\ref{fI}), which is in $\sB$ but not in $\sB^0$. By corollary \ref{maincor} all sequential functions
$f_w$ ($w\neq\mathrm{I}$) are in $\sB^0$.

\subsection{Completion}
There exist relations between graphical functions of different graphs. The best way to formulate these relations is by completing the graph.
The completion is obtained by adding a vertex $\infty$ in much the same way as for periods.

The completed graph has four labeled vertices $0,1,z,\infty$. For some edges $e$ in the completed graph we need to introduce propagators of negative integer weight $\nu_e$ which
correspond to quadrics $Q_e$ in the numerator of the integral (\ref{fdef4completed}). Graphically we indicate negative weight propagators by multiple wavy lines.
Completion adds edges to $\infty$ and from 0 to 1 in such a way that all unlabeled vertices have total degree four and the vertices $0,1,z,\infty$ have total degree zero (see figure 5).
Completion is always possible and unique (lemma \ref{completionlemma}).

\begin{center}
\fcolorbox{white}{white}{
  \begin{picture}(345,123) (4,-19)
    \SetWidth{0.8}
    \SetColor{Black}
    \Vertex(20,86){2.8}
    \Vertex(20,38){2.8}
    \Vertex(50,38){2.8}
    \Vertex(20,-10){2.8}
    \Line(20,86)(20,-10)
    \Line(20,86)(50,38)
    \Line(20,-10)(50,38)
    \PhotonArc[double,sep=4](131.026,36.814)(121.032,156.539,202.242){2.5}{5.5}
    \Photon[double,sep=4](20,37)(48,37){2.5}{3}
    \Text(25,85)[lb]{\normalsize{\Black{$1$}}}
    \Text(25,46)[lb]{\normalsize{\Black{$z$}}}
    \Text(25,-14)[lb]{\normalsize{\Black{$0$}}}
    \Text(48,46)[lb]{\normalsize{\Black{$\infty$}}}
    \Vertex(194,86){2.8}
    \Vertex(194,-10){2.8}
    \Vertex(82,38){2.8}
    \Vertex(114,38){2.8}
    \Vertex(146,38){2.8}
    \Vertex(178,38){2.8}
    \Vertex(210,38){2.8}
    \Vertex(242,38){2.8}
    \Vertex(274,38){2.8}
    \Vertex(306,38){2.8}
    \Vertex(338,38){2.8}
    \Vertex(306,86){2.8}
    \Line(82,38)(194,86)
    \Line(82,38)(194,-10)
    \Line(194,86)(114,38)
    \Line(194,86)(146,38)
    \Line(194,86)(178,38)
    \Line(178,38)(194,-10)
    \Line(210,38)(194,-10)
    \Line(242,38)(194,-10)
    \Line(194,86)(274,38)
    \Line(194,-10)(306,38)
    \Line(194,86)(306,38)
    \Line(82,38)(338,38)
    \Line(82,38)(306,87)
    \Line(114,38)(306,87)
    \Line(146,38)(306,87)
    \Line(210,38)(306,87)
    \Line(242,38)(306,87)
    \Line(274,38)(306,87)
    \Photon[double,sep=4](194,-10)(306,86){2.5}{8}
    \Photon(194,89)(305,89){2.5}{5}
    \Photon(194,86)(305,86){2.5}{5}
    \Photon(194,83)(305,83){2.5}{5}
    \Photon(197,85)(197,-9){2.5}{5}
    \Photon(194,85)(194,-9){2.5}{5}
    \Photon(191,85)(191,-9){2.5}{5}
    \Photon(338,39)(306,85){2.5}{3}
    \Text(182,86)[lb]{\normalsize{\Black{$1$}}}
    \Text(182,-20)[lb]{\normalsize{\Black{$0$}}}
    \Text(337,22)[lb]{\normalsize{\Black{$z$}}}
    \Text(311,87)[lb]{\normalsize{\Black{$\infty$}}}
  \end{picture}
}
Figure 5: The completion of the graphical functions $f_{\mathrm{I}}(z)$ and $f_{21120012}(z)$.
\end{center}
\vskip2ex

In four dimensions the graphical function of a completed graph $\Gamma$ is given by
\begin{equation}\label{fdef4completed}
f_\Gamma(z)=\left(\prod_{v\notin\{0,1,z,\infty\}} \int_{\RR^4}\frac{\dd^4x_v}{\pi^2}\right)\frac{1}{\prod_eQ_e^{\nu_e}}.
\end{equation}
Clearly, a graphical function does not change under completion,
\begin{equation}
f_\Gamma(z)=f_G(z).
\end{equation}
Because only completed graphs have a vertex $\infty$ we use the same symbol for their graphical functions.

It will be shown in \S \ref{completion} that a permutation of the four labels $0,1,z,\infty$ induces a M{\"o}bius transformation of the argument $z$.
Concretely, double transpositions of labels leave the argument $z$ invariant while a permutation $\phi$ of $\{0,1,\infty\}$ acts on $z$.
For a graph $\Gamma(a,b,c,d)$ with external labels $a,b,c,d\in\{0,1,z,\infty\}$ we have
\begin{eqnarray}\label{S4trafos}
f_{\Gamma(0,1,z,\infty)}&=&f_{\Gamma(1,0,\infty,z)}\quad=\quad f_{\Gamma(z,\infty,0,1)}\quad =\quad f_{\Gamma(\infty,z,1,0)},\\
f_{\Gamma(0,1,z,\infty)}&=&f_{\Gamma(\phi(0),\phi(1),\phi(z),\phi(\infty))},\nonumber
\end{eqnarray}
where we indicate the action of $\phi(z)\in\{z,1-z,\frac{z-1}{z},\frac{z}{z-1},\frac{1}{1-z},\frac{1}{z}\}$ on the argument $z$ by transforming the label $z$.
In total we obtain a 24-fold permutation symmetry which stabilizes the cross-ratio of the four labels.
The transformation $\phi$ maps $\sB$ into $\sB$ and $\sB^0$ into $\sB^0$ so that the class of graphical functions that is expressible in terms of SVMPs in the sense of (\ref{fGinA})
is invariant under the permutation of external labels.

\subsection{Appending an edge}
\begin{center}
\fcolorbox{white}{white}{
  \begin{picture}(336,79) (-52,-5)
    \SetWidth{0.8}
    \SetColor{Black}
    \GOval(32,34)(20,20)(0){0.882}
    \Vertex(32,54){2.8}
    \Vertex(32,14){2.8}
    \Vertex(52,34){2.8}
    \Text(30,60)[lb]{\normalsize{\Black{$1$}}}
    \Text(29,2)[lb]{\normalsize{\Black{$0$}}}
    \Text(57,32)[lb]{\normalsize{\Black{$z$}}}
    \Text(28,31)[lb]{\normalsize{\Black{$G$}}}

    \Line[arrow,arrowpos=1,arrowlength=5,arrowwidth=2,arrowinset=0.2](75,34)(100,34)

    \GOval(141,34)(20,20)(0){0.882}
    \Line(161,34)(196,34)
    \Vertex(141,54){2.8}
    \Vertex(141,14){2.8}
    \Vertex(198,34){2.8}
    \Text(139,60)[lb]{\normalsize{\Black{$1$}}}
    \Text(139,2)[lb]{\normalsize{\Black{$0$}}}
    \Text(205,32)[lb]{\normalsize{\Black{$z$}}}
    \Text(137,31)[lb]{\normalsize{\Black{$G_1$}}}
  \end{picture}
}
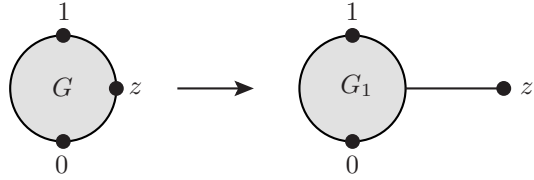
Figure 6: Appending an edge to the vertex $z$ in $G$ gives $G_1$.

\end{center}
\vskip2ex

A key result in the theory of graphical functions is that in many cases the graphical function is mapped from $\sB$ into $\sB^0$ by appending an edge to the vertex $z$
(see figure 6). Under quite general assumptions given in theorem \ref{appendthm} we obtain
\begin{equation}\label{inteq1}
f_{G_1}(z)=-\frac{1}{2(z-\zz)}\left(\int_0\dd z\!\int_0\dd\zz+\!\int_0\dd\zz\!\int_0\dd z\right)(z-\zz)f_G(z).
\end{equation}
The simplicity of this equation is special to four dimensions. In arbitrary dimensions appending an edge is significantly more complicated, see proposition \ref{appendprop}.

\subsection{Constructible graphs}
We have seen that permuting the labeled vertices and appending an edge to $z$ maps $\sB$ into $\sB$.
A third operation that is trivially of this type is adding an edge that connects two labeled vertices. In fact, this edge contributes non-trivially to the graphical function
only when it connects 0 or 1 to $z$. In this case the graphical function picks up a factor of $1/z\zz$ or $1/(z-1)(\zz-1)$, respectively.

The empty graph with four vertices is the completion of the sequential graph of the empty word. Its graphical function is one,
\begin{equation}
f_\emptyset=1.
\end{equation}
With the help of the above three transformations we can construct graphs from the empty graph and calculate their graphical functions (see figure 7).

\begin{center}
\fcolorbox{white}{white}{
  \begin{picture}(336,130) (-75,28)
    \SetWidth{0.8}
    \SetColor{Black}
    \GBox(80,117)(112,149){0.882}
    \Text(74,151)[lb]{\normalsize{\Black{$\bullet$}}}
    \Text(114,151)[lb]{\normalsize{\Black{$\bullet$}}}
    \Text(74,111)[lb]{\normalsize{\Black{$\bullet$}}}
    \Text(114,111)[lb]{\normalsize{\Black{$\bullet$}}}
    \Line[arrow,arrowpos=1,arrowlength=5,arrowwidth=2,arrowinset=0.2](75,128)(42,95)
    \Line[arrow,arrowpos=1,arrowlength=5,arrowwidth=2,arrowinset=0.2](96,112)(96,95)
    \Line[arrow,arrowpos=1,arrowlength=5,arrowwidth=2,arrowinset=0.2](117,128)(150,95)

    \GBox(16,53)(48,85){0.882}
    \Text(0,85)[lb]{\normalsize{\Black{$\pi(\bullet)$}}}
    \Text(0,41)[lb]{\normalsize{\Black{$\pi(\bullet)$}}}
    \Text(46,85)[lb]{\normalsize{\Black{$\pi(\bullet)$}}}
    \Text(46,41)[lb]{\normalsize{\Black{$\pi(\bullet)$}}}

    \GBox(80,53)(112,85){0.882}
    \Text(74,87)[lb]{\normalsize{\Black{$\bullet$}}}
    \Text(78,30)[lb]{\normalsize{\Black{$\bullet$}}}
    \Text(114,87)[lb]{\normalsize{\Black{$\bullet$}}}
    \Text(114,47)[lb]{\normalsize{\Black{$\bullet$}}}
    \Line(80,53)(80,37)

    \GBox(144,53)(176,85){0.882}
    \Arc(160,53)(16,-178.264,-1.736)
    \Text(138,87)[lb]{\normalsize{\Black{$\bullet$}}}
    \Text(178,87)[lb]{\normalsize{\Black{$\bullet$}}}
    \Text(138,47)[lb]{\normalsize{\Black{$\bullet$}}}
    \Text(178,47)[lb]{\normalsize{\Black{$\bullet$}}}
  \end{picture}
}
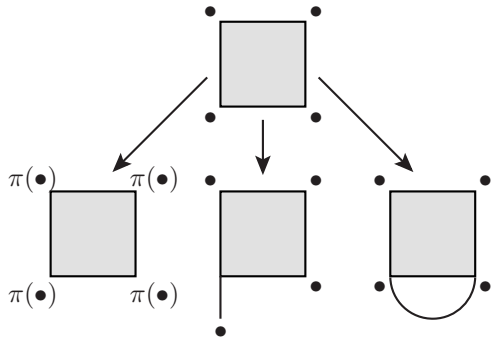
Figure 7: Construction of graphical functions which are in $\sB$. The bullets $\bullet$ stand for one of the four labels $0,1,z,\infty$ and $\pi$ is a permutation
of $0,1,z,\infty$ together with a transformation of $z$ which is generated by (\ref{S4trafos}).

\end{center}
\vskip2ex

For these `constructible' graphs the graphical functions can be calculated by computer up to weights $\approx 30$ which corresponds to 15 internal (unlabeled) vertices.
In particular, all sequential graphs are constructible.

The periods of completed primitive Feynman graphs can be calculated by using two-dimensional complex integration
if the graph decomposes into at most two constructible graphs under the following steps:
\begin{itemize}
\item Label four vertices by $0,1,z,\infty$.
\item Delete $0,1,z,\infty$ and decompose the graph into its connected components.
\item Add $0,1,z,\infty$ to each connected component in the same way the vertices of the component are connected to $0,1,z,\infty$ in the original graph.
\item Complete each component by adding edges between labeled vertices.
\end{itemize}
Such `constructible' periods can be calculated by computer up to eleven loops \cite{Polylogproc}, \cite{Hyperlogproc}.

It was observed in \cite{SchnetzCensus} that the known $\phi^4$ periods have an integer structure
in the sense that they are integer linear combinations of MZVs. This phenomenon is partially explained in remark \ref{Rsvremark} where we argue that constructible periods
are in $\sH^\sv(\ZZ)$ (the proof will be in \cite{PropGF}).

\subsection{Reduction modulo products}
Although integration in $\sA$ is well-suited for computer calculation, closed results are hard to obtain. An all orders result for sequential functions is only available in three cases:
For the case of a 2 followed by a sequence of 0s (or 1s), see example \ref{0ex}, for the case of a 2 followed by a string of 0s which contains one 1, see \cite{Drummond},
and for the case of the zig-zag graphs, see \cite{ZZ}. The result for zig-zag graphs leads to the proof of the zig-zag conjecture (\ref{PZ}).

A practical option to obtain general results is to calculate in the ideal $I_n$ generated by MZVs of weights between two and $n$.
Iterating (\ref{inteq1}) gives results for all sequential functions $f_w$ modulo $I_{|w|-2}$.
By setting the external variable $z$ to 1 this leads to a formula for the periods of sequential graphs modulo products of MZVs:
We extend the definition of $\zeta_w$, (\ref{zetaw}), to the letter 2 by
\begin{equation}
\zeta_{u2v}=\zeta_{u1v}-\zeta_{u0v}
\end{equation}
for words $u,v$ in 0,1,2 and define $\sH_{>0}^2$ as the $\QQ$ vector space spanned by non-trivial products of MZVs.
Then the sequential period of a word $2w2$ is given by (see theorem \ref{PGthm})
\begin{equation}\label{modproducts1}
P(G_{2w2})\equiv2(-1)^{|w|}(\zeta_{\widetilde{w}01w0}-\zeta_{\widetilde{w}10w0}) \mod \sH_{>0}^2,
\end{equation}
where $\widetilde{w}$ is $w$ in reversed order,
\begin{equation}
\widetilde{w}=a_na_{n-1}\cdots a_2a_1\quad\text{if}\quad w=a_1a_2\cdots a_{n-1}a_n.
\end{equation}

\subsection{Zig-zag graphs and generalizations}
The periods of three families of $\phi^4$ graphs are sequential. One family is the zig-zag family (see (\ref{Pzigzag}) and figure 1). Equation (\ref{modproducts1})
gives in this case (proposition \ref{modprop})
\begin{eqnarray}
P(Z_{2n+3})&\equiv&2\zeta(2^{\{n\}},3,2^{\{n\}})-2\zeta(2^{\{n-1\}},3,2^{\{n+1\}})\mod \sH_{>0}^2,\nonumber\\
P(Z_{2n+4})&\equiv&2\zeta(2^{\{n\}},3,2^{\{n+1\}})-2\zeta(2^{\{n+1\}},3,2^{\{n\}})\mod \sH_{>0}^2.
\end{eqnarray}
Due to a result by Zagier on MZVs of the above type \cite{Zagier} (see also \cite{Li}) this proves the zig-zag conjecture modulo products.

\begin{center}
\fcolorbox{white}{white}{
  \begin{picture}(345,305) (11,-10)
    \SetWidth{0.8}
    \SetColor{Black}
    \Arc(272,220)(48,90,450)
    \COval(292,176)(14,18)(0){White}{White}
    \Vertex(272,220){2.8}
    \Vertex(272,268){2.8}
    \Vertex(320,220){2.8}
    \Vertex(224,220){2.8}
    \Vertex(237,186){2.8}
    \Vertex(237,253){2.8}
    \Vertex(304,256){2.8}
    \Vertex(271,172){2.8}
    \Vertex(306,186){2.8}
    \Arc[clock](199.845,292.647)(76.248,-18.86,-72.322)
    \Arc[clock](179.67,219.61)(66.345,30.217,-28.454)
    \Arc[clock](215.351,164.202)(56.193,83.205,7.977)
    \Arc[clock](398.667,225.51)(99.819,-156.683,-196.584)
    \Arc[clock](326.867,129.735)(54.172,-97.283,-173.923)
    \Arc[clock](271,129.8)(66.208,121.914,58.086)
    \Arc[clock](334.327,157.982)(63.884,167.325,103.882)
    \Line(237,253)(272,220)
    \Line(304,256)(272,220)
    \Line(272,220)(307,186)
    \Line(272,220)(271.5,172)
    \Vertex(16,220){2.8}
    \Vertex(176,220){2.8}
    \Vertex(112,220){2.8}
    \Vertex(48,220){2.8}
    \Vertex(144,220){2.8}
    \Vertex(80,220){2.8}
    \Vertex(96,172){2.8}
    \Vertex(96,268){2.8}
    \Line(16,220)(96,268)
    \Line(96,268)(112,220)
    \Line(112,220)(96,172)
    \Line(96,172)(176,220)
    \Line(176,220)(96,268)
    \Line(16,220)(96,172)
    \Line(16,220)(176,220)
    \Line(144,220)(96,172)
    \Line(96,268)(80,220)
    \Line(48,220)(96,172)
    \Text(268,232)[lb]{\normalsize{\Black{$\infty$}}}
    \Text(85,147)[lb]{\normalsize{\Black{$G_{201202}$}}}
    \Text(263,147)[lb]{\normalsize{\Black{$\overline{A_{2,1}}$}}}
    \Vertex(16,76){2.8}
    \Vertex(96,124){2.8}
    \Vertex(48,76){2.8}
    \Vertex(80,76){2.8}
    \Vertex(112,76){2.8}
    \Vertex(144,76){2.8}
    \Vertex(176,76){2.8}
    \Vertex(96,28){2.8}
    \Line(16,76)(96,124)
    \Line(176,76)(96,124)
    \Line(96,28)(176,76)
    \Line(144,76)(96,28)
    \Line(112,76)(96,28)
    \Line(16,76)(176,76)
    \Line(16,76)(96,28)
    \Line(96,124)(48,76)
    \Line(80,76)(96,28)
    \Line(96,124)(112,76)
    \Arc(272,76)(48.166,132,492)
    \Vertex(273,124){2.8}
    \Vertex(320,76){2.8}
    \Vertex(271.5,28){2.8}
    \Vertex(224,77){2.8}
    \Vertex(305,111){2.8}
    \Vertex(238,110){2.8}
    \Vertex(237,44){2.8}
    \Vertex(308,44){2.8}
    \Vertex(272,76){2.8}
    \Arc[clock](201.845,149.647)(76.248,-18.86,-72.322)
    \Arc[clock](180.67,76.61)(66.345,30.217,-28.454)
    \Arc[clock](215.351,22.202)(56.193,83.205,7.977)
    \Arc[clock](399.667,83.51)(99.819,-156.683,-196.584)
    \Arc[clock](326.867,273.735)(54.172,-97.283,-173.923)
    \Text(268,88)[lb]{\normalsize{\Black{$\infty$}}}
    \Line(237,44)(320,76)
    \Line(237,111)(272,76)
    \Line(305,110)(272,76)
    \Line(272,76)(272,28)
    \Line(272,76)(307,45)
    \Text(85,2)[lb]{\normalsize{\Black{$G_{210202}$}}}
    \Text(263,2)[lb]{\normalsize{\Black{$\overline{B_{2,1}}$}}}
    \Text(94,276)[lb]{\normalsize{\Black{$1$}}}
    \Text(94,130)[lb]{\normalsize{\Black{$1$}}}
    \Text(94,160)[lb]{\normalsize{\Black{$0$}}}
    \Text(94,16)[lb]{\normalsize{\Black{$0$}}}
  \end{picture}
}
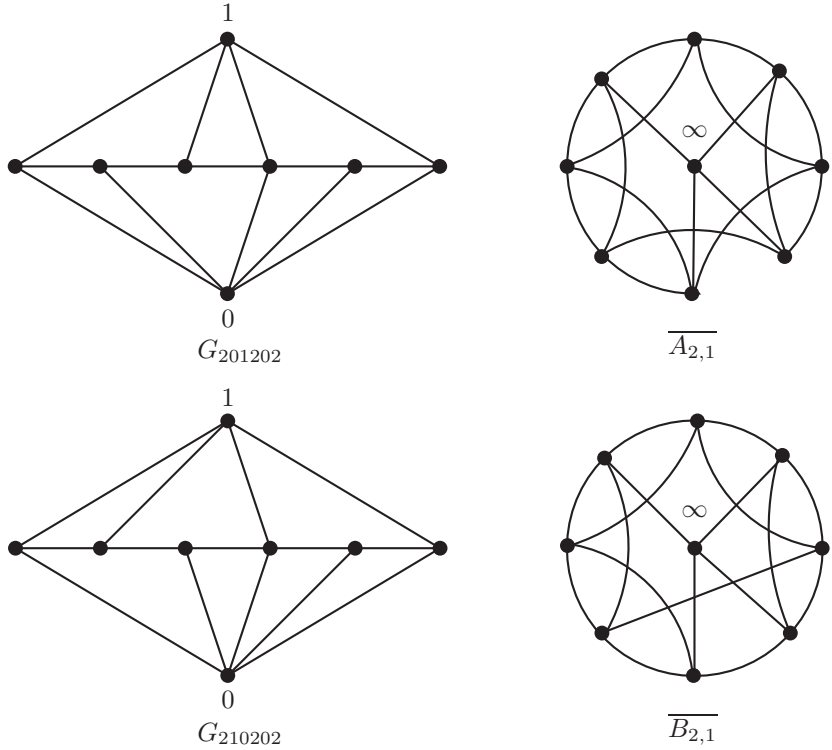
Figure 8: The completed $A$ and $B$ families of $\phi^4$ graphs.
\end{center}
\vskip2ex

It is remarkable that the periods of the zig-zag graphs reduce to MZVs in the Hoffman basis \cite{Hof} whereas
the somewhat similar but much simpler periods of the wheels with spokes (see (\ref{PWSn})) directly reduce to single Riemann zeta sums.
The different scenarios seem to reflect the difference in topology of the corresponding graphs.

The second and the third family of sequential $\phi^4$ periods arise from alternating words in 0 and 1 with one internal letter 2 (on the left of figure 8).
The $A$ type sequence has different letters immediately to the left and to the right of the middle letter 2 whereas the $B$ type sequence has equal letters.
The $A$ and $B$ graphs are the planar duals of the sequential graphs (on the right of figure 8).
The number of letters in the left and right sequences is one less than the number of internal arcs on the left and right side of the $A$ and $B$ graphs.
We have $A_{m,n}=A_{n,m}$, $B_{m,n}=B_{n,m}$, $A_{n,0}=B_{n,0}$; otherwise the graphs are non-isomorphic.
Their periods are given modulo products by MZVs of 2s with 1 or 3 in three slots, see e.g.\ (\ref{Aexample}).

By corollary \ref{Pw} the (constructible) periods of the $A$ and $B$ families are in $\sH^\sv(\ZZ)$. They can be calculated and reduced to a standard basis of MZVs up to loop order 12.
Assuming standard transcendentality conjectures, type $A$ and $B$ periods are examples of $\phi^4$ periods which are proved up to 12 loops to be MZVs which
cannot be expressed in terms of a single Riemann zeta sum.

\subsection{Two dimensions}
A special case are graphical functions in two dimensions. Whereas in higher dimensions propagators correspond to bosonic particles,
in two dimension we consider holomorphic and antiholomorphic propagators which are more closely related to fermionic particles.
A definition of two-dimensional graphical functions is given in \S \ref{2d}. In \ref{concellzeta} it is conjectured that the maximum weight piece of the periods of
graphical functions in two dimensions reduce modulo products to the sum of two cell zeta values \cite{cellzetas}.

\subsection{Computer implementation}
All algorithms of this article are implemented in Maple and available under \cite{Polylogproc}. In particular constructible periods can be calculated
by a $\tt period(edge set)$ command. A more powerful implementation is currently developed \cite{Hyperlogproc}.
\vskip1ex

\noindent{\bf Acknowledgements.} The article was written while the author was visiting scientist at Humboldt University, Berlin. The author is highly indebted
to Francis Brown for sharing his knowledge on SVMPs and for many very valuable discussions.

\section{Single-valued multiple polylogarithms}\label{SVMPs}

\subsection{Preliminaries on shuffle algebras and formal power series}
Let $R$ be a commutative unitary ring. Consider the two letter alphabet $X=\{\x_0,\x_1\}$ and let $X^\ast$ be the set of words in $X$ together with the
empty word 1. The shuffle algebra $\mathrm{Sh}_R\langle X \rangle$ is the free $R$-module over $X^\ast$ together with the shuffle product which is defined recursively
by  $w \sha 1 = 1 \sha w = w$ and 
$$ a u \sha b v  = a ( u \sha bv) + b ( au \sha v)$$
for all  $a,b \in X$, and $u,v,w \in X^\ast$. The shuffle product, extended linearly, makes $\mathrm{Sh}_R\langle X\rangle$ into a commutative unitary ring.

A Lyndon word is a non empty word $l\in X^\ast$ which is inferior to each of its strict right factors (for the lexicographical ordering),
i.e.\ if $l=uv$, $u\neq1$ then $l<v$. By Radford's theorem the $\QQ$ algebra $\mathrm{Sh}_\QQ\langle X \rangle$ is the polynomial algebra generated by Lyndon words.
For every word $w\in X^\ast$ let $\widetilde{w}$ denote the word $w$ in reversed order. We linearly extend $\widetilde{\bullet}$ to elements in $\mathrm{Sh}_R\langle X\rangle$.
The length (i.e.\ the number of letters) of a word $w$ is $|w|$. The shuffle algebra is graded by the length.

The deconcatenation coproduct is defined to be the linear map
\begin{eqnarray*}
\Delta : \mathrm{Sh}_R\langle X\rangle &  \To & \mathrm{Sh}_R\langle X\rangle \otimes_R \mathrm{Sh}_R\langle X\rangle\\
\Delta (w)  &=  &\sum_{uv =w} u \otimes v
\end{eqnarray*}
and the antipode is the linear map defined by $w \mapsto (-1)^{|w|}\widetilde{w}$.
With these definitions, $\mathrm{Sh}_R\langle X\rangle$ is a commutative, graded Hopf algebra over $R$.

The dual of $\mathrm{Sh}_R\langle X \rangle$ is the $R$-module of non-commutative formal power series
$$
R\langle \langle X \rangle \rangle = \{ S= \sum_{w \in X^\ast} S_w w, \quad  S_w \in R \}
$$
equipped with the concatenation product. We define on $R\langle \langle X \rangle \rangle$ a (completed) coproduct
$$
\Delta^\star: R\langle \langle X\rangle \rangle  \To  R \langle \langle X \rangle \rangle \widehat{\otimes}_R R \langle \langle X \rangle\rangle
$$
for which the elements $\x_0, \x_1$ are primitive: $\Delta^\star( \x_i) = 1\otimes \x_i + \x_i \otimes 1$ for $i=0,1$.
The same antipode as in $\mathrm{Sh}_R\langle X\rangle$ turns $R\langle \langle X\rangle \rangle$ into a completed cocommutative but not commutative Hopf algebra.
The duality between $T\in\mathrm{Sh}_R\langle X \rangle$ and $S\in R\langle \langle X\rangle \rangle$ is defined as
\begin{equation}\label{Sp}
(T|S)=\sum_{w\in X^\ast}T_wS_w.
\end{equation}

The set of Lie monomials in $R \langle \langle X \rangle \rangle$ is defined by induction: the letters $\x_0$, $\x_1$ are Lie monomials and the bracket
$[x,y]=xy-yx$ of two Lie monomials $x$ and $y$ is a Lie monomial. A Lie polynomial (respectively a Lie series) is a finite (respectively infinite) $R$-linear
combination of Lie monomials. The set $\Lie_R\langle X\rangle$ of Lie polynomials is a free Lie algebra and the set of Lie series
$\Lie_R\langle\langle X\rangle\rangle$ is its completion with respect to the augmentation ideal ker$\epsilon$, where $\epsilon:R\langle\langle X\rangle\rangle\to R$
projects onto the empty word. The bracket form of Lyndon words is recursively defined as $P(x)=x$ for all $x\in X$ and $P(\ell)=[P(u),P(v)]$ if $\ell=uv$ for Lyndon words
$u,v$ and $v$ being as long as possible. A basis for $\Lie_R\langle X\rangle$ is given by the bracket forms $P(\ell)$ of Lyndon words.

An invertible  series  $S \in R\langle \langle X \rangle \rangle^\times$ (i.e., with invertible leading term $S_1$) is group-like if
$\Delta^\star(S) = S \otimes S$. Equivalently, the coefficients $S_w$ of $S$ define a homomorphism for the shuffle product:
$S_{u \sha v} = S_u S_{v}$ for all $u, v \in X^\ast$, where $S_{\bullet}$ is extended by linearity on the left-hand side.
The condition for $S\in R\langle \langle X \rangle \rangle^\times$ to be group-like is equivalent to the condition that $S$ is a Lie exponential,
i.e.\ that there exists an $L\in\Lie_R\langle X\rangle$ such that $S=\exp(L)$. By the formula for the antipode, it follows that for such a series $S=S(\x_0,\x_1)$,
its inverse is given by
\begin{equation} \label{Sinversion}
S( \x_0, \x_1)^{-1} = \widetilde{S} (-\x_0, -\x_1).
\end{equation}
In the following we often use the letters `0' and `1' for $\x_0$ and $\x_1$.

\subsection{Iterated integrals}\label{itint}
In \cite{Chen} Chen develops the theory of iterated path integration on general manifolds. Here we need only the elementary one-dimensional case.
For a fixed path $\gamma:[0,1]\To\CC\backslash\{0,1\}$ from $y=\gamma(0)$ to $z=\gamma(1)$ and differential forms $\omega_0(t)=\dd t/t$ and $\omega_1(t)=\dd t/(t-1)$ we define
\begin{equation}
I(y;a_1\ldots a_n;z)_\gamma=\int_{0<t_1<\ldots<t_n<1}\gamma^\ast\omega_{a_1}(t_1)\wedge\ldots\wedge\gamma^\ast\omega_{a_n}(t_n),\quad a_1,\ldots,a_n\in\{0,1\}
\end{equation}
(where the simplex ${0<t_1<\ldots<t_n<1}$ is endowed with the standard orientation and $\gamma^\ast\omega$ is the pullback of $\omega$ by $\gamma$)
as the iterated path integral of the word $a_1\ldots a_n$ along $\gamma$. Iterated path integrals have the following properties:
\begin{description}
\item[I0]
$I(y;z)_\gamma=1$ (by definition).
\item[I1]
$I(y;w;z)_\gamma$ is independent of the parametrization of $\gamma$.
\item[I2]
$I(y;w;z)_\gamma$ is a homotopy invariant; it only depends on the homotopy class of $\gamma$.
\item[I3]
$I(z;w;z)_\gamma=0$ for the constant path $\gamma=z$ and $|w|\geq1$.
\item[I4]
$I(z;w;y)_{\gamma^{-1}}=(-1)^{|w|}I(y;\widetilde{w};z)_\gamma$ (path reversal), where the inverse $\gamma^{-1}$ of $\gamma$ is $\gamma$ with reversed orientation.
\item[I5]
$$I(y;a_1\ldots a_n;z)_\gamma=\sum_{k=0}^nI(y;a_1\ldots a_k;x)_{\gamma_1} I(x;a_{k+1}\ldots a_n;z)_{\gamma_2}$$
(path composition), where $\gamma_1(1)=\gamma_2(0)=x\in\CC\backslash\{0,1\}$ and $\gamma=\gamma_1\gamma_2$ is the composition of (first) $\gamma_1$ and (second) $\gamma_2$.
\item[I6]
$$I(y;a_1\ldots a_r;z)_\gamma I(y;a_{r+1}\ldots a_{r+s};z)_\gamma=\sum_{\sigma\in\sS(r,s)}I(y;a_{\sigma(1)}\ldots a_{\sigma(r+s)};z)_\gamma$$
(shuffle product), where $\sS(r,s)$ is the set of $(r,s)$-shuffles:
with $n=r+s$ and $\sS_n$ the group of permutations of $\{1,\ldots,n\}$,
$$\sS(r,s)=\{\sigma\in\sS_n:\sigma^{-1}(1)<\ldots<\sigma^{-1}(r)\hbox{ and }\sigma^{-1}(r+1)<\ldots<\sigma^{-1}(n)\}.$$
\item[I7]
$$I(\phi(y);w;\phi(z))_{\phi(\gamma)}=I(y;\phi^\ast w;z)_\gamma$$
(chain rule), where $\phi$ is a M\"obius transformation that permutes the singular points $0,1,\infty$ and $\phi^\ast(a_i)$ is induced by $\phi^\ast\omega_{a_i}$.
Concretely, $\phi$ is one out of the six transformations $z\mapsto\{z,1-z,\frac{z-1}{z},\frac{z}{z-1},\frac{1}{1-z},\frac{1}{z}\}$ which induce the transformations
of letters `0', `1' (extended linearly to words and to iterated integrals)
\begin{eqnarray}\label{fmap}
z:\hbox{id},\hspace{3.1cm}&&\hspace{-3ex}1-z:(\text{`0'}\mapsto\text{`1'},\text{`1'}\mapsto\text{`0'}),\nonumber\\
\frac{z-1}{z}:(\text{`0'}\mapsto -\text{`0'}+\text{`1'},\text{`1'}\mapsto -\text{`0'}),&&\hspace{-3ex}
\frac{z}{z-1}:(\text{`0'}\mapsto \text{`0'}-\text{`1'},\text{`1'}\mapsto -\text{`1'}),\nonumber\\
\frac{1}{1-z}:(\text{`0'}\mapsto -\text{`1'},\text{`1'}\mapsto \text{`0'}-\text{`1'}),&&\hspace{-3ex}
\frac{1}{z}:(\text{`0'}\mapsto -\text{`0'},\text{`1'}\mapsto -\text{`0'}+\text{`1'}).
\end{eqnarray}
\item[I8]
$$
\partial_zI(y;wa;z)=\frac{1}{z-a}I(y;w;z),\quad\partial_yI(y;aw;z)=-\frac{1}{y-a}I(y;w;z).
$$
\end{description}
By {\bf I6} the vector space of iterated path integrals for fixed $\gamma$ is a shuffle algebra. It is convenient to form the generating (Chen) series
\begin{equation}
S_\gamma(\x_0,\x_1)=\sum_{w\in X^\ast}I(y;w;z)_\gamma\,w\in\CC\langle\langle X\rangle\rangle,
\end{equation}
where we identify the letters $\x_0$ and $\x_1$ with 0 and 1 in $I$. From the properties of the iterated path integral $S_\gamma$ inherits the properties
\begin{description}
\item[S1]
$S_\gamma$ is independent of the parametrization of $\gamma$.
\item[S2]
$S_\gamma$ is a homotopy invariant.
\item[S3]
$S_\gamma=1$ for the constant path $\gamma$.
\item[S4]
$S_{\gamma^{-1}}=(S_\gamma)^{-1}$.
\item[S5]
$S_{\gamma_1}S_{\gamma_2}=S_{\gamma_1\gamma_2}$, if $\gamma_1(1)=\gamma_2(0)$ and $\gamma_1\gamma_2$ is the path $\gamma_1$ followed by $\gamma_2$.
\item[S6]
$S_\gamma$ is a Lie exponential: $\Delta^\star S_\gamma=S_\gamma\otimes S_\gamma$.
\item[S7]
$S_{\phi(\gamma)}=\phi^\ast S_\gamma$ for a M\"obius transformation $\phi$ that permutes $0,1,\infty$.
The action of $\phi$ on $\x_0$, $\x_1$ is dual (transpose) to its action on `0' and `1' in (\ref{fmap})\footnote{Here the letters 0 and 1 are dual to the letters $\x_0$, $\x_1$.},
\begin{eqnarray}\label{fmapdual}
z:\hbox{id},\hspace{3.5cm}&&\hspace{-3ex}1-z:(\x_0\mapsto \x_1,\x_1\mapsto \x_0),\nonumber\\
\frac{z-1}{z}:(\x_0\mapsto -\x_0-\x_1,\x_1\mapsto \x_0),&&\hspace{-3ex}\frac{z}{z-1}:(\x_0\mapsto \x_0,\x_1\mapsto-\x_0 -\x_1),\nonumber\\
\frac{1}{1-z}:(\x_0\mapsto \x_1,\x_1\mapsto -\x_0-\x_1),&&\hspace{-3ex}\frac{1}{z}:(\x_0\mapsto -\x_0-\x_1,\x_1\mapsto \x_1).
\end{eqnarray}
\item[S8]
$$
\partial_zS_\gamma=S_\gamma\left(\frac{\x_0}{z}+\frac{\x_1}{z-1}\right),\quad\partial_yS_\gamma=-\left(\frac{\x_0}{y}+\frac{\x_1}{y-1}\right)S_\gamma,
$$
if $\gamma(0)=y$ and $\gamma(1)=z$.
\end{description}
If the homotopy class of the path $\gamma$ from $a$ to $b$ is clear from the context we write $S(a,b)$ and consider $S$ as a function of the initial and the end point.

If the path $\gamma$ approaches the singular values 0 and 1 the following limits are well-defined \cite{BrSVMPII},
\begin{eqnarray}
\lim_{\epsilon\to0}\mathrm{e}^{\x_0\ln|\epsilon|}S(\epsilon,z)=F_0(z),&&
\lim_{\epsilon\to0}\mathrm{e}^{\x_1\ln|\epsilon|}S(1-\epsilon,z)=F_1(z),\nonumber\\
\lim_{\epsilon\to0}S(y,\epsilon)\mathrm{e}^{-\x_0\ln|\epsilon|}=F_0(y)^{-1},&&
\lim_{\epsilon\to0}S(y,1-\epsilon)\mathrm{e}^{-\x_1\ln|\epsilon|}=F_1(y)^{-1}.
\end{eqnarray}
To write the above equations in more convenient form we define regularized limits by nullifying every positive power of $\ln\epsilon$ in the limit $\epsilon\to0$.
In our context this prescription is equivalent to regularization by tangential base points introduced by Deligne in \cite{Deligne}.
With this notation we have $S(0,z)=S(z,0)^{-1}=F_0(z)$ and $S(1,z)=S(z,1)^{-1}=F_1(z)$.

\subsection{Multiple zeta values}
Regularized iterated integrals for the path $\mathrm{id}:[0,1]\to[0,1]$ span the shuffle algebra $\sH$ of MZVs,
\begin{equation}\label{zetawdef}
\sH=\langle\zeta_w,\,w\in X^\ast\rangle_\QQ,\quad\zeta_w=I(0;w;1).
\end{equation}
Regularization extends $\sH$ to all words. In particular, $\zeta_0=\zeta_1=0$.
If $w$ begins with 1 and ends with 0 then $\zeta_w$ can be converted into a sum by (\ref{zetaw}) and (\ref{MZV}).
The generating function of regularized MZVs (denoting $\x_0,\x_1$ as indices)
\begin{equation}
Z_{\x_0,\x_1}= \sum_{w\in X^\ast} \zeta_w\, w\;=\;S(0,1) \in \CC \langle \langle X\rangle \rangle
\end{equation}
is Drinfeld's associator. It is group-like and by (\ref{Sinversion}) and (\ref{fmapdual}) for $z\to1-z$,
\begin{equation}\label{Zid}
Z_{\x_0,\x_1}=Z_{\x_1,\x_0}^{-1}=\widetilde{Z}_{-\x_1,-\x_0}=\widetilde{Z}_{-\x_0,-\x_1}^{-1}.
\end{equation}

The weight $|w|$ of $\zeta_w$ induces a filtration on $\sH$. The conjecture that the weight gives a grading on $\sH$ is proved in the motivic analogue.
We define subspaces $\sH_k\subseteq\sH$ of weight $k$ MZVs.

Shuffling the summation indices of a product of two MZVs in the sum representation yields the set of quasi-shuffle identities.
Conjecturally regularized shuffle and quasi-shuffle relations generate all relations between MZVs. The dimension of $\sH_k$ is
conjectured (and proved in the motivic analogue \cite{MMZ}) to have the following generating series
\begin{equation}
\sum_{k=0}^\infty \dim{\sH_k}t^k=\frac{1}{1-t^2-t^3}.
\end{equation}
By the shuffle identity, MZVs span a ring over the integers,
\begin{equation}
\sH(\ZZ)=\langle\zeta_w,\,w\in X^\ast\rangle_\ZZ\subset\sH.
\end{equation}

An important factor algebra of $\sH$ is obtained by factorizing out the ideal generated by $\zeta(2)$. This factor algebra coacts on $\sH$ by
\begin{equation}\label{coprod}
\Delta:\sH\To\sH/\zeta(2)\sH\otimes\sH.
\end{equation}
An explicit formula for $\Delta$ is given in \cite{GON}.

The reduction of $\Delta$ modulo $\zeta(2)$ on both sides of $\otimes$ in (\ref{coprod}) turns $\sH/\zeta(2)\sH$ into a Hopf algebra.
A main conjecture on MZVs (a theorem for the motivic analogue) states that this Hopf-algebra is non-canonically isomorphic to the shuffle Hopf-algebra on
words in letters of odd weight greater or equal three \cite{DEC}. Conventionally one uses the `$f$-alphabet' $f_3$, $f_5$, \ldots\ for the letters.
The isomorphism into the $f$-alphabet extends to $\sH$ by adding a generator $f_2$ of weight two.
The polynomial algebra generated by $f_2$ is tensored to the right to the Hopf-algebra in odd letters so that an MZV maps under the isomorphism to a sum of words
in odd letters concatenated to the right by a power of $f_2$. The coaction $\Delta$ deconcatenates the word in the $f$-alphabet and fulfills $\Delta f_2^k=1\otimes f_2^k$.

The number of odd letters in a word $w$ in the $f$-alphabet is the coradical depth of $w$. The coradical depth gives the $f$-alphabet a second grading (together with the weight).
A computer implementation of the basis dependent isomorphism into the $f$-alphabet is in \cite{Zetaproc}.

\subsection{Multiple polylogarithms}
Multiple polylogarithms for words $w$ in the letters 0 and 1 are recursively defined by
\begin{equation}
L_{wa}(z)=\int\frac{L_w(z)}{z-a}\dd z,\quad a\in\{0,1\},
\end{equation}
with initial condition $L_w(0)=0$ unless $w$ is a sequence of zeros in which case we have $L_{0^{\{n\}}}(z)=(\ln z)^n/n!$.
Multiple polylogarithms are multi-valued analytic functions on $\PP^1\CC\backslash\{0,1,\infty\}$ with monodromies around 0, 1, and $\infty$.
They can be expressed as regularized iterated integrals from 0 to $z$,
\begin{equation}
L_w(z)=I(0;w;z)_\gamma,
\end{equation}
where the dependence on the homotopy class of $\gamma$ gives rise to the multi-valuedness of the multiple polylogarithm.
If $z$ is in the unit ball $|z|<1$ and not on the negative real axis we assume that $\gamma$ is homotopic to a straight line from 0 to $z$.

The generating series of multiple polylogarithms is denoted by
\begin{equation} 
L_{\x_0,\x_1}(z)=\sum_{w\in X^\ast } L_w(z)w\;=\;S(0,z),
\end{equation}
where we often suppress the indices $\x_0,\x_1$. It is the unique solution to the Knizhnik-Zamolodchikov equation \cite{K-Z}, \cite{KZ1}
\begin{equation} \label{KZ}
\partial_z L(z) = L(z)\left(\frac{\x_0}{z} + \frac{\x_1}{z-1}\right)
\end{equation}
which satisfies the asymptotic condition (see \cite{BrSVMPII} where the opposite convention is used: differentiation of $L_w(z)$ corresponds to deconcatenation of $w$ to the left.)
\begin{equation} \label{Lat0}
L(z)=\mathrm{e}^{\x_0\ln(z)} h_0(z)
\end{equation}
for all  $z$ in the neighborhood of the origin, where $h_0(z)$ is a function taking values in $\CC\langle\langle X\rangle \rangle$ which is holomorphic at $0$ and satisfies $h(0)=1$.
 
The series $L(z)$ is a group-like formal power series. In particular, the polylogarithms $L_w(z)$ satisfy the shuffle product formula
\begin{equation} \label{Lshuff}
L_{w \sha w'}(z) = L_w(z) L_{w'}(z) \hbox{ for all } w,w' \in X^\ast.
\end{equation}

Drinfeld's associator $Z$ is the regularized limit of $L(z)$ at the point $z=1$.
\begin{lem}\label{h1}
There exists a function $h_1(z)$ taking values in the series $\CC\langle \langle X\rangle \rangle$, which is holomorphic at $z=1$ where it takes the value $h(1)=1$, such that
\begin{equation}\label{Lat1}
L(z) = Z\mathrm{e}^{\x_1 \ln(1-z)} h_1(z).
\end{equation}
\end{lem}
\begin{proof}
For $0<z<1$ we obtain from (\ref{fmapdual}) for $z\to1-z$ that $L_{\x_0,\x_1}(z)=S_{\x_1,\x_0}(1,1-z)$. By {\bf S5} we have $S_{\x_1,\x_0}(1,1-z)=S_{\x_1,x_0}(1,0)L_{\x_1,\x_0}(1-z)$
and from {\bf S4}, (\ref{Sinversion}) and (\ref{Zid}) this equals $Z_{\x_0,x_1}\exp(\x_1\ln(1-z))h_0(1-z)$. Depending on the sheet of $L$ at $z=1$ there exists a $k\in\ZZ$
such that $h_1(z)=h_0(1-z)\exp(2k\pi\mathrm{i}\x_1)$ which has the required properties.
\end{proof}

For $i \in \{0, 1\}$, let $\sM_i$ denote analytic continuation around a path winding once around the point $i$ in the positive direction. The operators $\sM_i$ act on the series
$L(z)$ and its complex conjugate $L(\zz)$, commute with multiplication, $\partial_z$, and $\partial_\zz$.

\begin{lem} \label{lemmonodromy} \cite{L-D}. The monodromy operators $\sM_{0}$, $\sM_1$ act as follows:
\begin{eqnarray}\label{mon}
\sM_{0} L(z) & = &\mathrm{e}^{2\pi\mathrm{i}\x_0} L(z), \\
\sM_{1} L(z) & = &Z\mathrm{e}^{2\pi\mathrm{i}\x_{1}}Z^{-1} L(z).\nonumber 
\end{eqnarray}
\end{lem}
\begin{proof}
The formulae follow from $(\ref{Lat0})$, $(\ref{Lat1})$ and the equations $\sM_0\ln z=\ln z + 2\pi\mathrm{i}$, $\sM_1\ln (1-z)=\ln(1-z) + 2\pi\mathrm{i}$.
\end{proof}

A sum representation of multiple polylogarithms in the unit ball can be derived from
\begin{equation}\label{classicalpolyasL}
(-1)^rL_{10^{\{n_1-1\}}\ldots10^{\{n_r-1\}}}(z) =\Li_{n_1,\ldots,n_r}(z) = \sum_{1\leq k_1,\ldots<k_r} \frac{z^{k_r}}{k_1^{n_1}\ldots k_r^{n_r}},
\end{equation}
which expresses $L_{10^{\{n-1\}}}$ in terms of the classical polylogarithm in the case $r=1$.
\begin{lem}
We have the following explicit expression of $L_w$ in terms of the $\Li$s.
\begin{eqnarray}\label{explicitLw}
&&L_{0^{\{n_0\}}10^{\{n_1-1\}}\ldots10^{\{n_r-1\}}}(z)\\
&=&\sum_{\genfrac{}{}{0pt}{}{k_0\geq0,k_i\geq n_i\,(i\geq1)}{\sum k_i=\sum n_i}}(-1)^{k_0+n_0+r}\prod_{i=1}^r\binom{k_i-1}{n_i-1}\frac{(\ln z)^{k_0}}{k_0!}\Li_{k_1,\ldots,k_r}(z),\nonumber
\end{eqnarray}
where $\Li_\emptyset=1$.
\begin{proof}
With the above expression for $L$ we have $\partial_zL_{wa}(z)=L_w(z)/(z-a)$. Because $L_{0^{\{n\}}}(z)=(\ln z)^n/n!$ and $L_w(0)=0$ if $w\neq0^{\{n\}}$ the lemma follows.
\end{proof}
\end{lem}
\begin{ex}
For $r=1$ we obtain
\begin{equation}\label{r1ex}
L_{0^{\{n_0\}}10^{\{n_1-1\}}}=\sum_{k=0}^{n_0}(-1)^{k+1}\binom{n_1+k-1}{n_1-1}\frac{(\ln z)^{n_0-k}}{(n_0-k)!}\Li_{n_1+k}(z).
\end{equation}
\end{ex}

\subsection{Brown's construction of single-valued multiple polylogarithms}\label{FrancisSVMPs}
Multiple polylogarithms $L_w(z)$ can be combined with their complex conjugates $L_w(\zz)$ to kill the monodromy at 0, 1, and $\infty$, rendering the function
single-valued on $\PP^1\CC\backslash\{0,1,\infty\}$. Because the monodromy (\ref{mon}) is homogeneous in the weight (if one gives $\pi$ the weight one) a single-valued
expression will decompose into components of pure weight, where the weight of a product of multiple polylogarithms with MZVs is the sum of the weights
of the polylogarithms plus the weight of the MZVs (see examples \ref{wt2ex}, \ref{lox0} and corollary \ref{homcor}).
The vector space $\sP$ of single-valued multiple polylogarithms (SVMPs) has a direct decomposition with respect to the weight
\begin{equation}
\sP=\bigoplus_{n\geq0}\sP_n.
\end{equation}
At weight one single-valued logarithms are $P_0(z)=L_0(z)+L_0(\zz)=\ln(z\zz)$ and $P_1(z)=L_1(z)+L_1(\zz)=\ln((z-1)(\zz-1))$. They span the two-dimensional vector space $\sP_1$.
The differential operator $\partial_z$ maps SVMPs into the differential algebra of SVMPs over $\sO=\CC[z,\frac{1}{z},\frac{1}{z-1}]$.
Likewise antiholomorphic differentiation $\partial_\zz$ generates an $\overline{\sO}$ algebra. Together $\sA=\sO\overline{\sO}\sP$ is the $\partial_z,\partial_\zz$ bi-differential
algebra generated by SVMPs. It is a direct sum of its weighted components,
\begin{equation}
\sA=\bigoplus_{n\geq0}\sA_n,
\end{equation}
where functions in $\sO\overline{\sO}$ have weight zero.

Because $\partial_z$ (or $\partial_\zz$) decreases the weight of a SVMP by one while generating a denominator $z$ or $z-1$ it is clear that the dimension of the vector
space $\sP_n$ is at most $2^n$. The following theorem states that $\sP$ is a shuffle algebra with $\dim\sP_n=2^n$.
\begin{thm}[F. Brown, \cite{BrSVMP}]\label{Francisthm}
There exists a unique family of single-valued functions $\{P_w(z):\;w\in X^\ast,\;z\in\CC\backslash\{0,1\}\}$, each of which is an explicit linear combination
of the functions $L_{w}(\zz)L_{w'}(z)$ where $w,w'\in X^\ast$, which satisfy the differential equations
\begin{equation}
\partial_zP_{wa}(z)=\frac{P_w(z)}{z-a},\hbox{ for }a\in\{0,1\},
\end{equation}
such that $P_\emptyset(z)=1$, $P_{0^{\{n\}}}(z)=\frac{1}{n!}(\ln z\zz)^n$ for all $n\in\NN$, and $\lim_{z\to0}P_w(z)=0$ if $w$ is not of the form $0^{\{n\}}$.
The functions $P_w(z)$ satisfy the shuffle relations, and are linearly independent over $\sO\overline{\sO}$. Every linear combination of the functions
$L_{w}(\zz)L_{w'}(z)$, where $w,w'\in X^\ast$, which is single-valued, can be written as a unique linear combination of functions $P_w(z)$.
\end{thm}
The functions $P_w(z)$ can be constructed explicitly. To this end we make for the generating series
\begin{equation}
P_{\x_0,\x_1}(z)=\sum_{w\in X^\ast } P_w(z)\,w
\end{equation}
the ansatz \cite{BrSVMPII}
\begin{equation}\label{Pgenseries}
P_{\x_0,\x_1}(z)=\widetilde{L}_{\x_0,\x_1'}(\zz)L_{\x_0,\x_1}(z)
\end{equation}
for a yet to be determined series $\x_1'\in\CC\langle\langle X\rangle\rangle$ which is substituted in $\widetilde{L}$.
For any $\x_1'$ the monodromy of $P_{\x_0,\x_1}(z)$ at zero vanishes: from (\ref{mon}) we have
$$
\sM_0P_{\x_0,\x_1}(z)=\widetilde{L}_{\x_0,\x_1'}(\zz)\mathrm{e}^{-2\pi\mathrm{i}\x_0}\mathrm{e}^{2\pi\mathrm{i}\x_0}L_{\x_0,\x_1}(z)=P_{\x_0,\x_1}(z).
$$
For the monodromy at one we obtain
$$
\sM_1P_{\x_0,\x_1}(z)=\widetilde{L}_{\x_0,\x_1'}(\zz)(\widetilde{Z}^{-1})_{\x_0,\x_1'}\mathrm{e}^{-2\pi\mathrm{i}\x_{1}'}\widetilde{Z}_{\x_0,\x_1'}
Z_{\x_0,\x_1}\mathrm{e}^{2\pi\mathrm{i}\x_{1}}Z^{-1}_{\x_0,\x_1}L_{\x_0,\x_1}(z).
$$
With (\ref{Zid}) a sufficient condition for trivial monodromy at one is
$$
Z_{-\x_0,-\x_1'}\mathrm{e}^{-2\pi\mathrm{i}\x_1'}Z_{-\x_0,-\x_1'}^{-1}=Z_{\x_0,\x_1}\mathrm{e}^{-2\pi\mathrm{i}\x_{1}}Z^{-1}_{\x_0,\x_1}
$$
which holds if
\begin{equation}\label{x1cond}
Z_{-\x_0,-\x_1'}\x_1'Z_{-\x_0,-\x_1'}^{-1}=Z_{\x_0,\x_1}\x_{1}Z^{-1}_{\x_0,\x_1}.
\end{equation}
Recall that $\sH(\ZZ)$ is the ring of integer MZVs.
\begin{lem}\label{Lielem}
There exists an $\x_1'\in\Lie_{\sH(\ZZ)}\langle\langle X\rangle\rangle$ such that (\ref{x1cond}) holds.
Moreover, $\x_1'\equiv\x_1$ modulo words with at least two $\x_1s$ (depth two) or modulo words of length greater than or equal to four.
\end{lem}
\begin{proof}
Define the function
$$
H_{\x_0,\x_1}(t)=Z_{\x_0,\x_1}\mathrm{e}^{t\x_{1}}Z^{-1}_{\x_0,\x_1}.
$$
Because $Z^{-1}_{\x_0,\x_1}=\widetilde{Z}_{-\x_0,-\x_1}$, equation (\ref{Zid}), we have $H\in\sH(\ZZ)[t]\langle\langle X\rangle\rangle$.
As a product of three Lie exponentials $H(t)=\mathrm{e}^{G(t)}$ is a Lie exponential. Because $G(0)=0$ we have
$\frac{\dd}{\dd t}|_{t=0}H(t)=\frac{\dd}{\dd t}|_{t=0}G(t)\in\Lie_{\sH(\ZZ)}\langle\langle X\rangle\rangle$. Differentiation yields
$$
F(\x_0,\x_1)=Z_{\x_0,\x_1}\x_{1}Z^{-1}_{\x_0,\x_1}-\x_{1}\in\Lie_{\sH(\ZZ)}\langle\langle X\rangle\rangle.
$$
Explicitly,
$$
F(\x_0,\x_1)=\zeta(2)[[\x_0,\x_1],\x_1]+\zeta(3)([[[\x_0,\x_1],\x_1],\x_1]-[\x_0,[[\x_0,\x_1],\x_1]])+\ldots.
$$
Clearly, $F\equiv0$ modulo depth two and $F\equiv\zeta(2)[[\x_0,\x_1],\x_1]$ modulo weight four.
We recursively solve (\ref{x1cond}) by defining
$$
\x_1'^{(0)}=\x_1,\quad\x_1'^{(k+1)}=\x_1+F(\x_0,\x_1)+F(-\x_0,-\x_1'^{(k)}).
$$
Because $F$ has lowest weight three, the recursion converges in the weight filtration. Define $\x_1'$ as the limit of the recursion.
Then $\x_1'=\x_1+F(\x_0,\x_1)+F(-\x_0,-\x_1')$ fulfills (\ref{x1cond}). Because by induction every $\x_1'^{(k)}\in\Lie_{\sH(\ZZ)}\langle\langle X\rangle\rangle$
we have $\x_1'\in\Lie_{\sH(\ZZ)}\langle\langle X\rangle\rangle$. Modulo depth two $\x_1'^{(k)}=\x_1$ for all $k$. Modulo weight four the recursion stabilizes after
$k=1$ for which we have $\x_1'^{(1)}\equiv\x_1$ modulo weight four.
\end{proof}
In (\ref{x1primeinLie}) we will show that only certain (single-valued) MZVs appear in the series of $\x_1'$.
The first non-trivial contributions to $\x_1'$ have four generators. Explicitly,
\begin{equation}\label{x1}
\x_1'=\x_1+2\zeta(3)([[[\x_0,\x_1],\x_1],\x_1]-[\x_0,[[\x_0,\x_1],\x_1]])+\zeta(5)(\dots)+\ldots,
\end{equation}
where the $\zeta(5)$ contribution consists of eight bracket words in six generators (see \cite{Polylogproc}).

For later use we define the following---in general multi-valued---multiple polylogarithms.
\begin{defn}\label{P0def}
For any word $w$ in 0 and 1 let
\begin{equation}\label{P0defeq}
P_w^0(z)=\sum_{uv=w}L_{\widetilde{u}}(\zz)L_v(z).
\end{equation}
\end{defn}
The multiple polylogarithms $P^0_w$ equal $P_w$ modulo $\sH(\ZZ)$. Their generating series is the untwisted version of equation (\ref{Pgenseries})
\begin{equation}
P^0_{\x_0,\x_1}(z)=\widetilde{L}_{\x_0,\x_1}(\zz)L_{\x_0,\x_1}(z).
\end{equation}
Up to weight three or depth one the SVMPs $P_w$ equal $P_w^0$.
\begin{ex}\label{wt2ex}
At weight two we obtain
\begin{eqnarray}\label{wt2}
P_{00}(z)&=&L_{00}(\zz)+L_{0}(\zz)L_{0}(z)+L_{00}(z)\;=\;\frac{1}{2}P_0(z)^2\;=\;\frac{1}{2}(\ln z\zz)^2,\nonumber\\
P_{01}(z)&=&L_{10}(\zz)+L_{0}(\zz)L_{1}(z)+L_{01}(z),\nonumber\\
P_{10}(z)&=&L_{01}(\zz)+L_{1}(\zz)L_{0}(z)+L_{10}(z),\nonumber\\
P_{11}(z)&=&L_{11}(\zz)+L_{1}(\zz)L_{1}(z)+L_{11}(z)\;=\;\frac{1}{2}P_1(z)^2\;=\;\frac{1}{2}(\ln (z-1)(\zz-1))^2.\nonumber\\
\end{eqnarray}
By the shuffle identity $P_{0,1}+P_{1,0}=P_0P_1$ so that there exists only one genuinely new SVMP of weight two. One may take $P_{0,1}-P_{1,0}$ as this new function and,
by the shuffle identity on $L_w$, we obtain
\begin{eqnarray}\label{BWdilog}
P_{01}(z)-P_{10}(z)&=&2L_{10}(\zz)-2L_{10}(z)+(L_1(z)-L_1(\zz))(L_0(z)+L_0(\zz))\nonumber\\
&=&4\mathrm{i\,Im}(\Li_2(z)+\ln(1-z)\ln|z|)\;\;=\;\;4\mathrm{i}D(z),
\end{eqnarray}
where $D$ is the Bloch-Wigner dilogarithm \cite{Zagierdilog}.
\end{ex}
\begin{ex}\label{lox0}
At weight four the following SVMPs gain a $\zeta(3)$-term from (\ref{x1}),
\begin{equation}
P_w(z)=P^0_w(z)+c_wL_1(\zz),\quad\hbox{with}
\end{equation}
\begin{eqnarray}\label{weightfour}
c_{0011}=-2\zeta(3),&&c_{0101}=4\zeta(3),\quad c_{1010}=-4\zeta(3),\quad c_{1100}=2\zeta(3),\nonumber\\
c_{0111}=2\zeta(3),&&c_{1011}=-6\zeta(3),\quad c_{1101}=6\zeta(3),\quad c_{1110}=-2\zeta(3).
\end{eqnarray}
A formula for $c_w$ will be given in lemma \ref{modintlemma}, equation (\ref{modintsconst}).
\end{ex}
\begin{ex}\label{ex0s}
With (\ref{r1ex}) we obtain at weight $n_0+n_1+1$
\begin{eqnarray}\label{formula0s}
P_{0^{\{n_0\}}10^{\{n_1\}}}&=&\sum_{k=0}^{n_0}(-1)^{k+1}\binom{n_1+k}{n_1}\frac{(\ln z\zz)^{n_0-k}}{(n_0-k)!}\Li_{n_1+k+1}(z)\nonumber\\
&&+\;\sum_{k=0}^{n_1}(-1)^{k+1}\binom{n_0+k}{n_0}\frac{(\ln z\zz)^{n_1-k}}{(n_1-k)!}\Li_{n_0+k+1}(\zz).
\end{eqnarray}
\end{ex}

From (\ref{x1cond}) we obtain for the function $\x_1'=\x_1'(\x_0,\x_1)$ the identities
\begin{eqnarray}\label{x1id}
&&\x_1'(-\x_0,-\x_1)=-\widetilde{\x_1'}(\x_0,\x_1),\nonumber\\
&&\x_1'(\x_0,\widetilde{\x_1'})=\widetilde{\x_1'}(\x_0,\x_1')=\x_1.
\end{eqnarray}
The holomorphic and antiholomorphic differentials of the generating series $P_{\x_0,\x_1}$ are
\begin{eqnarray}\label{diffonP}
\partial_z P_{\x_0,\x_1}&=&P_{\x_0,\x_1}\left(\frac{\x_0}{z}+\frac{\x_1}{z-1}\right),\nonumber\\
\partial_\zz P_{\x_0,\x_1}&=&\left(\frac{\x_0}{\zz}+\frac{\x_1'}{\zz-1}\right)P_{\x_0,\x_1}.
\end{eqnarray}
Upon projection $(\bullet|aw)$ (see (\ref{Sp})) onto the word $aw$ the second identity gives
\begin{eqnarray}\label{diffonaw}
\partial_\zz P_{aw}(z)&=&\frac{\delta_{a,0}}{\zz}P_w(z)+\frac{1}{\zz-1}\sum_{aw=uv}(\x_1'|u)P_v(z)\nonumber\\
&=&\frac{P_w(z)}{\zz-a}+\frac{1}{\zz-1}\sum_{aw=uv}(\x_1'-\x_1|u)P_v(z).
\end{eqnarray}

To study the number theoretic content of SVMPs we now reduce to the ground field $\QQ$.
In $P_w$---due to contributions from $\x_1'$---multiple polylogarithms of lower weights mix with MZVs.
\begin{defn}
The total weight of a product of MZVs with holomorphic and antiholomorphic multiple polylogarithms is the sum of the individual weights.
\end{defn}
\begin{cor}\label{homcor}
$P_w$ is homogeneous of total weight $|w|$.
\end{cor}
\begin{proof}
Let the generators $\x_0$, $\x_1$ have weight $-1$. Then $L(z)$, $L(\zz)$, $Z$ have weight 0.
From the proof of lemma \ref{Lielem} it follows that the total weight of $\x_1'$ equals $-1$. Hence $P_{\x_0,\x_1}$ has total weight 0.
\end{proof}

By construction the regularized values of the $P_w(z)$ at 0 vanish. Their regularized values at 1 form a ring which by lemma \ref{Lielem} is a sub-ring of $\sH(\ZZ)$.
\begin{defn}\label{Rsvdef}For any ring $R\subseteq\CC$ define the ring
\begin{equation}
\sH^\sv(R)=\langle P_w(1),\;w\in X^\ast\rangle_R,
\end{equation}
and let $\sH^\sv=\sH^\sv(\QQ)\subset\sH$ be the $\QQ$ algebra of regularized values of SVMPs at 1.
\end{defn}
We will see in the next section that the regularized values $P_w(z)$ at infinity equivalently span $\sH^\sv(R)$. Moreover,
$\x_1'\in\Lie_{\sH^\sv(\ZZ)}\langle\langle X\rangle\rangle$ by theorem \ref{Stabilitythm}.

\begin{remark}
In \cite{Bsv} F. Brown shows that the $\QQ$ algebra $\sH^\sv$ of `single-valued' MZVs has a deep algebraic structure.
In particular, assuming transcendentality conjectures we can characterize $\sH^\sv$.
As a shuffle algebra $\sH^\sv$ is generated as an algebra by MZVs of odd weight on which $\Delta:\sH^\sv\to\sH/\zeta(2)\sH\otimes\sH^\sv$ coacts.
The number of generators at weight $2k+1$ equals the number of weight $2k+1$ Lyndon words
in odd generators of weight greater or equal three. There exists a canonical map $\psi:\sH\rightarrow\sH^\sv$.
\end{remark}

\begin{ex}\label{Rsvex}
Up to weight ten $\sH^\sv$ is generated by $\zeta(3)$, $\zeta(5)$, $\zeta(7)$, $\zeta(9)$. At weight 11 we have in addition to $\zeta(11)$ the MZV
$$
g_{335}=\zeta(3,3,5)-\frac{4}{7}\zeta(5)\zeta(2)^3+\frac{6}{5}\zeta(7)\zeta(2)^2+45\zeta(9)\zeta(2)
$$
corresponding to the Lyndon word 335. In fact, for $\Delta'x=\Delta x-1\otimes x-x\otimes 1$,
$$
\Delta'g_{335}\equiv\zeta(3,5)\otimes\zeta(3)-\frac{5}{2}\zeta(5)\otimes\zeta(3)^2\in\sH/\zeta(2)\sH\otimes\sH^\sv.
$$
The ring $\sH^\sv(\ZZ)$ is spanned up to weight eleven by $2\zeta(3)$, $\zeta(5)$, $2\zeta(3)^2$, $\frac{1}{8}\zeta(7)$, $\zeta(3)\zeta(5)$, $\frac{1}{72}\zeta(9)$,
$\frac{4}{3}\zeta(3)^3$, $\frac{1}{8}\zeta(3)\zeta(7)$, $\frac{1}{2}\zeta(5)^2$, $\frac{1}{384}\zeta(11)$, $\frac{1}{5}(g_{335}+\zeta(11))+\frac{3}{2}\zeta(3)^2\zeta(5)$,
and $2\zeta(3)^2\zeta(5)$.
\end{ex}

\subsection{Permuting $0,1,\infty$}
A special property of SVMPs on $\PP^1\CC\backslash\{0,1,\infty\}$ is that there exists a group of M\"obius transformations that permute the singular points.
\begin{defn}
Let $\sS_3$ be the group of M\"obius transformations of $\PP^1\CC$ that permute $\{0,1,\infty\}$,
\begin{equation}
\sS_3=\left\{z\to \phi(z),\;\phi(z)\in\left\{z,1-z,\frac{z-1}{z},\frac{z}{z-1},\frac{1}{1-z},\frac{1}{z}\right\}\right\}.
\end{equation}
\end{defn}
The generating series $P_{\x_0,\x_1}$ of SVMPs transforms as follows:
\begin{lem}
The following identities hold
\begin{equation}\label{x1trafo}
\x_1'(\x_0,-\x_0-\x_1)=-\x_0-\x_1'(\x_0,\x_1).
\end{equation}
\begin{eqnarray}\label{S3trafos}
P_{\x_0,\x_1}(1-z)&=&\widetilde{L}_{\x_1',\x_0}(\zz)\widetilde{Z}_{\x_0,\x_1'}Z_{\x_0,\x_1}L_{\x_1,\x_0}(z)\nonumber\\
&=&P_{\x_0,\x_1}(1)P_{\x_1,\x_0}(z)\;\;=\;\;P_{\x_1',\widetilde{\x_1'}(\x_1',\x_0)}(z)P_{\x_0,\x_1}(1),\nonumber\\
P_{\x_0,\x_1}\left(\frac{z-1}{z}\right)&=&\widetilde{L}_{-\x_0-\x_1',\x_0}(\zz)\widetilde{Z}_{\x_0,-\x_0-\x_1'}Z_{\x_0,-\x_0-\x_1}L_{-\x_0-\x_1,\x_0}(z)\nonumber\\
&=&P_{\x_0,-\x_0-\x_1}(1)P_{-\x_0-\x_1,\x_0}(z)\nonumber\\
&=&P_{-\x_0-\x_1',\widetilde{\x_1'}(-\x_0-\x_1',\x_0)}(z)P_{\x_0,-\x_0-\x_1}(1),\nonumber\\
P_{\x_0,\x_1}\left(\frac{z}{z-1}\right)&=&\widetilde{L}_{\x_0,-\x_0-\x_1'}(\zz)L_{\x_0,-\x_0-\x_1}(z)\;\;=\;\;P_{\x_0,-\x_0-\x_1}(z),\nonumber\\
P_{\x_0,\x_1}\left(\frac{1}{1-z}\right)&=&\widetilde{L}_{\x_1',-\x_0-\x_1'}(\zz)\widetilde{Z}_{\x_0,\x_1'}Z_{\x_0,\x_1}L_{\x_1,-\x_0-\x_1}(z)\nonumber\\
&=&P_{\x_0,\x_1}(1)P_{\x_1,-\x_0-\x_1}(z)\;\;=\;\;P_{\x_1',\widetilde{\x_1'}(\x_1',-\x_0-\x_1')}(z)P_{\x_0,\x_1}(1),\nonumber\\
P_{\x_0,\x_1}\left(\frac{1}{z}\right)&=&\widetilde{L}_{-\x_0-\x_1',\x_1'}(\zz)\widetilde{Z}_{\x_0,-\x_0-\x_1'}Z_{\x_0,-\x_0-\x_1}L_{-\x_0-\x_1,\x_1}(z)\nonumber\\
&=&P_{\x_0,-\x_0-\x_1}(1)P_{-\x_0-\x_1,\x_1}(z)\nonumber\\
&=&P_{-\x_0-\x_1',\widetilde{\x_1'}(-\x_0-\x_1',\x_1')}(z)P_{\x_0,-\x_0-\x_1}(1).
\end{eqnarray}
\end{lem}
\begin{proof}
We first prove the first set of equations in (\ref{S3trafos}). The first formula for $P_{\x_0,\x_1}(\frac{z}{z-1})$ is a consequence of (\ref{Pgenseries}) and (\ref{fmapdual}).
By the argument in the proof of lemma \ref{h1}
$$
L_{\x_0,\x_1}(1-z)=Z_{\x_0,\x_1}L_{\x_1,\x_0}(z),
$$
leading to the first formula for $P_{\x_0,\x_1}(1-z)$.
If we apply the transformation $z\to\frac{z}{z-1}$ to $P_{\x_0,\x_1}(1-z)$ we obtain the first formula for $P_{\x_0,\x_1}(\frac{1}{1-z})$.
The transformation $z\to1-z$ applied to $P_{\x_0,\x_1}(\frac{z}{z-1})$ gives with the above formula for $L_{\x_0,\x_1}(1-z)$ the result for $P_{\x_0,\x_1}(\frac{z-1}{z})$.
Another transformation $z\to\frac{z}{z-1}$ yields the first formula for $P_{\x_0,\x_1}(\frac{1}{z})$ by (\ref{fmapdual}).

The second set of equations is equivalent to the first because they fulfill the same differential equations with respect to $\partial_z$.
An initial value is easily checked for the transformations $\tau:z\to1-z$ and $\sigma:z\to\frac{z}{z-1}$.
The other cases follow by iterating $\tau$ and $\sigma$.

The third set of equations follows like the second set of equations by antiholomorphic differentiation $\partial_\zz$ using (\ref{diffonP})
and the second identity in (\ref{x1id}).

Because $L_{\x_0,-\x_0-\x_1}(z)$ is invertible (it is a Lie exponential) we obtain from the transformation $z\to\frac{z}{z-1}$
$$
\widetilde{L}_{\x_0,\x_1'(\x_0,-\x_0-\x_1)}(\zz)=\widetilde{L}_{\x_0,-\x_0-\x_1'}(\zz).
$$
By the independence of the multiple polylogarithms in $\zz$ we obtain (\ref{x1trafo}).
\end{proof}
If we substitute $z=1$ in $P_{\x_0,\x_1}(\frac{z}{z-1})=P_{\x_0,-\x_0-\x_1}(z)$ we obtain $P_{\x_0,\x_1}(\infty)=P_{\x_0,-\x_0-\x_1}(1)$.
The ring of SVMPs generates at $\infty$ the same MZVs as at 1.

If we set $z=1$ in $P_{\x_0,\x_1}(\frac{1}{z})$ we get the following corollary:
\begin{cor}
Let $a,b,c$ be linear expressions in $\x_0,\x_1$ such that $a+b+c=0$. Then $X_{\x_0,\x_1}=Z_{a,b}Z_{b,c}Z_{c,a}$ fulfills the identity
\begin{equation}\label{Xid}
\widetilde{X}_{\x_0,\x_1'}X_{\x_0,\x_1}=1.
\end{equation}
\end{cor}
\begin{proof}
If we substitute $z=1$ into $P_{\x_0,\x_1}(\frac{1}{z})$ in (\ref{S3trafos}) we obtain the identity
$$
\widetilde{Z}_{-\x_0-\x_1',\x_1'}\widetilde{Z}_{\x_0,-\x_0-\x_1'}Z_{\x_0,-\x_0-\x_1}Z_{-\x_0-\x_1,\x_1}=\widetilde{Z}_{\x_0,\x_1'}Z_{\x_0,\x_1}.
$$
By (\ref{Zid}) we have $Z_{\x_0,\x_1}=Z_{\x_1,\x_0}^{-1}$ and (\ref{Xid}) is true for $a=\x_0$, $b=-\x_0-\x_1$, $c=\x_1$. By a change of variables this proves (\ref{Xid}) for
linearly independent $a$ and $b$. If $a$ and $b$ are linearly dependent then---because $Z$ is a Lie-exponential---$Z_{a,b}=Z_{b,c}=Z_{c,a}=1$.
\end{proof}
Note that in general $X_{\x_0,\x_1}$ is not 1. It follows from the second identity for $P_{\x_0,\x_1}(\frac{1}{z})$ that the canonical map
$\psi:\sH\rightarrow\sH^\sv$ maps $X_{\x_0,\x_1}$ to 1.

\subsection{Integration of single-valued multiple polylogarithms}\label{integration}
A consequence of theorem \ref{Francisthm} is that the integrals $\int\dd z/z$, $\int\dd z/(z-1)$ of SVMPs of weight $n$ are unique in $\sP_{n+1}$.
If we integrate (\ref{diffonaw}) and use induction over the weight we see that the antiholomorphic integrals $\int\dd\zz/\zz$ and $\int\dd\zz/(\zz-1)$ also exist.
However from the induction they pick up terms of lower weights so that the antiholomorphic integrals are unique in $\sP_{>0}=\bigoplus_{n\geq1}\sP_n$ where only weights
up to $n+1$ contribute. The weight grading hence singles out holomorphic integration; only the filtration by weights less or equal $n$ is symmetric under complex conjugation.

Both holomorphic and antiholomorphic integrals have the property that their regularized limit at zero vanishes. They are henceforth written as $\int_0$.
\begin{defn}
For $P\in\sP$ let $\int_0\frac{P(z)}{z-a}\dd z$, $a\in\{0,1\}$ be the unique function $F\in\sP_{>0}$ such that $\partial_zF(z)=\frac{P(z)}{z-a}$.
Likewise $G(z)=\int_0\frac{P(z)}{\zz-a}\dd\zz\in\sP_{>0}$ is uniquely defined by $\partial_\zz G(z)=\frac{P(z)}{\zz-a}$.
For $b\in\{1,\infty\}$ we define $\int_b\frac{P(z)}{z-a}\dd z=F(z)-F(b)$ and $\int_b\frac{P(z)}{\zz-a}\dd\zz=G(z)-G(b)$, where $F(b)$ and $G(b)$ are regularized limits.
\end{defn}
The integral $\int_0\dd z/(z-a)$ is calculated by appending the letter $a$ to the basis SVMPs $P_w$. The antiholomorphic integral is inductively defined by integrating
(\ref{diffonaw}),
\begin{equation}\label{intonw}
\int_0\frac{P_w(z)}{\zz-a}\dd\zz=P_{aw}(z)-\sum_{aw=uv}(\x_1'-\x_1|u)\int_0\frac{P_v(z)}{\zz-1}\dd\zz,
\end{equation}
where $a\in\{0,1\}$. Note that $u$ has at least length 4 so that an inductive algorithm rapidly terminates.
\begin{ex}\label{lox}
For words $w$ with length at most two we have
\begin{equation}\label{loworderexamples}
\int_0\frac{P_w(z)}{\zz-a}\dd\zz=P_{aw}(z),\quad\hbox{if }|w|\leq2.
\end{equation}
For words $w$ of length three with $c_{aw}\neq0$ in (\ref{weightfour}) antiholomorphic integration picks up a $\zeta(3)$ term (for $P^0_w(z)$ see definition \ref{P0def}),
$$
\int_0\frac{P_w(z)}{\zz-a}\dd\zz=P_{aw}^0(z)-c_{aw}L_1(z)\;=\;P_{aw}(z)-c_{aw}P_1(z).
$$
\end{ex}
Equation (\ref{intonw}) uses $\x_1'$ to give a formula for antiholomorphic integration. Likewise the construction of the basis $P_w$ uses $\x_1'$.
The series $\x_1'$ grows rapidly with the weight. Therefore $\x_1'$ is only accessible up to weight $\approx11$.
An alternative approach to integration of MZVs provides holomorphic and antiholomorphic integration without using $\x_1'$ and enables us
to construct the basis $P_w$ to high weights ($\approx30$). We need the following propositions:
\begin{prop}\label{limprop0}
Let $P\in\sP$ be a SVMP. Then
\begin{equation}\label{limeq0}
\lim_{z\to0}z\partial_z\int_0\frac{P(z)}{\zz-a}\dd\zz=\delta_{a,0}P(0),\hbox{ for }a\in\{0,1\},
\end{equation}
where $\delta_{a,0}$ is the Kronecker delta.
\end{prop}
\begin{proof}
By linearity we may assume without restriction that $P=P_w$ for some word $w$. We use induction over the length of $w$.
If $w=\emptyset$ then $\int_0\frac{1}{\zz-a}\dd\zz=\ln(z-a)(\zz-a)$ and (\ref{limeq0}) follows.
For the word $wb$, $b\in\{0,1\}$, we obtain from (\ref{intonw}) by induction
\begin{eqnarray*}
\lim_{z\to0}z\partial_z\int_0\frac{P_{wb}(z)}{\zz-a}\dd\zz&=&\lim_{z\to0}z\partial_zP_{awb}(z)-\sum_{awb=uv}(\x_1'-\x_1|u)\lim_{z\to0}z\partial_z\int_0\frac{P_v(z)}{\zz-1}\dd\zz\\
&=&\lim_{z\to0}z\frac{P_{aw}(z)}{z-b}\;=\;\delta_{b,0}P_{aw}(0)\;=\;0.\\
\end{eqnarray*}
\end{proof}
\begin{prop}\label{limprop1}
Let $P\in\sP$ be a SVMP. Then
\begin{equation}\label{limeq1}
\lim_{z\to1}(z-1)\partial_z\int_0\frac{P(z)}{\zz-a}\dd\zz=\delta_{a,1}P(1),\hbox{ for }a\in\{0,1\}.
\end{equation}
\end{prop}
\begin{proof}
Substituting $z\to1-z$ in (\ref{limeq0}) yields for all $Q\in\sP$, $b\in\{0,1\}$
$$
\lim_{1-z\to0}(z-1)\partial_z\int_1\frac{Q(1-z)}{\zz-(1-b)}\dd\zz=\delta_{b,0}Q(0).
$$
Because $\partial_z\int_1=\partial_z\int_0$ the proposition follows from a substitution $a=1-b$, $P(z)=Q(1-z)$.
\end{proof}
The main tool for calculating holomorphic and antiholomorphic integrals is an explicit formula for their commutator.
\begin{lem}\label{commutelem}
Let $P\in\sP$ and $a,b\in\{0,1\}$. The integrals $\int_0\frac{\dd z}{z-a}$, $\int_0\frac{\dd\zz}{\zz-b}$ fulfill the following commutation relation:
\begin{eqnarray}\label{commute}
&&\hspace{-1cm}\left(\int_0\frac{\dd z}{z-a}\int_0\frac{\dd\zz}{\zz-b}-\int_0\frac{\dd\zz}{\zz-b}\int_0\frac{\dd z}{z-a}\right)P(z)\nonumber\\
&=&\left(\left.\int_0\frac{\dd\zz}{\zz-b}P(z)\right|_{z=a}-\left.\int_0\frac{\dd z}{z-a}P(z)\right|_{z=b}\right)\cdot P_1(z).
\end{eqnarray}
\end{lem}
\begin{proof}
Applying $\partial_z\partial_{\zz}=\partial_{\zz}\partial_z$ to the left hand side of the above equation leads to the ansatz
$$
\left(\int_0\frac{\dd z}{z-a}\int_0\frac{\dd\zz}{\zz-b}-\int_0\frac{\dd\zz}{\zz-b}\int_0\frac{\dd z}{z-a}\right)P(z)=c+c_0P_0(z)+c_1P_1(z).
$$
In the regularized limit $z\to0$ the left hand side vanishes, hence $c=0$.

Applying $\partial_z$ to both sides of the above equation yields
\begin{equation}\label{proof1}
\left(\frac{1}{z-a}\int_0\frac{\dd\zz}{\zz-b}-\partial_z\int_0\frac{\dd\zz}{\zz-b}\int_0\frac{\dd z}{z-a}\right)P(z)=\frac{c_0}{z}+\frac{c_1}{z-1}.
\end{equation}

Multiplication by $z$ and taking the limit $z\to0$ gives
$$
c_0=\lim_{z\to0}\left(\frac{z}{z-a}\int_0\frac{\dd\zz}{\zz-b}P(z)-z\partial_z\int_0\frac{\dd\zz}{\zz-b}Q(z)\right),
$$
where $Q(z)=\int_0\frac{\dd z}{z-a}P(z)$.
The first term on the right hand side vanishes. By proposition \ref{limprop0}
we get for the second term on the right hand side $-\delta_{b,0}Q(0)$. From $Q(0)=0$ we obtain $c_0=0$.

Multiplication of (\ref{proof1}) by $z-1$ and taking the limit $z\to1$ gives
$$
c_1=\lim_{z\to1}\left(\frac{z-1}{z-a}\int_0\frac{\dd\zz}{\zz-b}P(z)-(z-1)\partial_z\int_0\frac{\dd\zz}{\zz-b}Q(z)\right).
$$
With proposition \ref{limprop1} we have
$$
c_1=\delta_{a,1}\left.\int_0\frac{\dd\zz}{\zz-b}P(z)\right|_{z=1}-\delta_{b,1}Q(1).
$$
This completes the proof.
\end{proof}
\begin{remark}\label{comm0}
The integrals $\int_0\dd z/z$ and $\int_0\dd\zz/\zz$ commute.
\end{remark}
\begin{remark}\label{commrem}
The above lemma allows us to construct the integrals $\int_0\dd z$ and $\int_0\dd\zz$ by the following steps.
Assume $P\in\sP_n$ of weight $n>0$. Then $\partial_{\zz}P(z)=Q_0(z)/\zz+Q_1(z)/(\zz-1)$ with $Q_0,Q_1\in\sP_{n-1}$ and
$$
P(z)=\int_0\left(\frac{Q_0(z)}{\zz}+\frac{Q_1(z)}{\zz-1}\right)\dd\zz.
$$
The integral $\sI(z)=\int_0\frac{P(z)}{z-a}\dd z$ is given up to a multiple polylogarithm in $\zz$ by integrating the holomorphic multiple polylogarithms in $P(z)$.
Using the above equation and (\ref{commute}) we obtain
\begin{equation}\label{commrem1}
\sI(z)=\sum_{b=0,1}\int_0\frac{\dd\zz}{\zz-b}\int_0\frac{Q_b(z)}{z-a}\dd z+\left(P(a)-\left.\sum_{b=0,1}\int_0\frac{Q_b(z)}{z-a}\dd z\right|_{z=b}\right)\cdot P_1(z).
\end{equation}
Now, $Q_0,Q_1$ are of smaller weight than $P$. We hence may assume that $\int_0\frac{Q_b(z)}{z-a}\dd z$ is known by induction for $b=0,1$. The above equation fixes
$\sI$ by integration of the antiholomorphic multiple polylogarithms up to a holomorphic multiple polylogarithm. Altogether $\sI$ is given up to a constant which is fixed by the
condition that the regularized limit of $\sI$ at $z=0$ vanishes.

Antiholomorphic integration is given by the analogous algorithm.
\end{remark}
\begin{defn}\label{Indef}
Let $I_n$ be the ideal in $\sH$ generated by MZVs of weights between two and $n$.
\end{defn}
So (assuming transcendentality conjectures, otherwise we only have inclusions) $I_0=I_1=0$, $I_2=\zeta(2)\sH$ and $I_3=\zeta(2)\sH+\zeta(3)\sH$.
Recall the definition of $P_w^0(z)$ in (\ref{P0defeq}). We have for $a,b\in\{0,1\}$,
\begin{equation}\label{diffeq0}
\partial_z P_{wa}^0(z)=\frac{P_w^0(z)}{z-a},\quad\partial_{\zz} P_{bw}^0(z)=\frac{P_w^0(z)}{\zz-b}.
\end{equation}
Moreover, from (\ref{Sinversion}) we have for words $w$ of positive length
\begin{equation}\label{zetaop}
\zeta_w+(-1)^{|w|}\zeta_{\widetilde{w}}\equiv0\mod I_{|w|-1}.
\end{equation}
Hence
\begin{equation}\label{P0w1}
P_w^0(1)\equiv(1-(-1)^{|w|})\zeta_w\mod I_{|w|-1}.
\end{equation}
Because $P_w$ is of total weight $|w|$ and $P_w(0)=0$ we have
\begin{equation}\label{PP0}
P_w(z)\equiv P_w^0(z)\mod I_{|w|-1}.
\end{equation}

\begin{lem}\label{modintlemma}
Let $a,b\in\{0,1\}$ and $w$ be a word in 0 and 1. Then
\begin{eqnarray}\label{modints}
\int_0\frac{P_{bw}(z)}{z-a}\dd z&\equiv&P^0_{bwa}(z)+c_{bwa}L_1(\zz)\mod I_{|w|},\nonumber\\
\int_0\frac{P_{wa}(z)}{\zz-b}\dd\zz&\equiv&P^0_{bwa}(z)-c_{bwa}L_1(z)\mod I_{|w|},
\end{eqnarray}
where (see examples \ref{lox0} and \ref{lox})
\begin{equation}\label{modintsconst}
c_{bwa}=(1+(-1)^{|w|})(\delta_{a,1}\zeta_{bw}-\delta_{b,1}\zeta_{wa}).
\end{equation}
\end{lem}
\begin{proof}
We prove (\ref{modints}) by induction over the length of $w$ using remark \ref{commrem}. The statement reduces to lemma \ref{Lielem} and (\ref{loworderexamples}) if $w=\emptyset$.
Taking $\partial_z$ on both sides, the first equation in (\ref{modints}) reduces to an identity by (\ref{PP0}). Hence the result holds up to an antiholomorphic function.
Moreover, we have by induction and (\ref{PP0})
$$
P_{bw}(z)\equiv\int_0\frac{P_w(z)}{\zz-b}\dd\zz\mod I_{|w|}.
$$
Using lemma \ref{commutelem} we obtain
$$
\int_0\frac{P_{bw}(z)}{z-a}\dd z\equiv\int_0\frac{\dd\zz}{\zz-b}P_{wa}(z)+(P_{bw}(a)-P_{wa}(b))P_1(z)\mod I_{|w|}.
$$
Applying $\partial_\zz$ to the right hand side we obtain from (\ref{PP0}) and the second identity in (\ref{diffeq0})
$$
\int_0\frac{P_{bw}(z)}{z-a}\dd z\equiv P^0_{bwa}+(\delta_{a,1}P^0_{bw}(1)-\delta_{b,1}P^0_{wa}(1))L_1(\zz)+f(z)\mod I_{|w|}
$$
for some analytic function $f$. The constant in front of $L_1(\zz)$ reduces to (\ref{modintsconst}) by (\ref{P0w1}).
Altogether the first equation in (\ref{modints}) holds up to a constant. The constant is zero by the regularized limit $z\to0$.

Because $P^0_w(\zz)=P^0_{\widetilde{w}}(z)$ the second equation in (\ref{modints}) follows from complex conjugating the first and swapping $a$ and $b$.
\end{proof}
Note that for any word $w$
\begin{equation}\label{csym}
c_w=-c_{\widetilde{w}}.
\end{equation}
We obtain a refinement of (\ref{PP0}), see example \ref{lox0},
\begin{equation}
P_w(z)\equiv P_w^0(z)+c_wL_1(\zz)\mod I_{|w|-2}.
\end{equation}

\subsection{A residue theorem}\label{residue}
From the series expansions of the multiple polylogarithms $L_w$ at zero we know that every $f\in\sA$ (see (\ref{defA})) has a Laurent series of the form
\begin{equation}\label{series0}
f(z)=\sum_{k=0}^{K_0}\sum_{m=M_0}^\infty\sum_{n=N_0}^\infty c_{k,m,n}^0(\ln z\zz)^kz^m\zz^n,
\end{equation}
for $K_0\in\NN$, $M_0,N_0\in\ZZ$ and constants $c^0_{k,m,n}\in\CC$. The Laurent series converges for $0<|z|<1$.
Likewise at $z=1$ we have the expansion
\begin{equation}\label{series1}
f(z)=\sum_{k=0}^{K_1}\sum_{m=M_1}^\infty\sum_{n=N_1}^\infty c_{k,m,n}^1(\ln (z-1)(\zz-1))^k(z-1)^m(\zz-1)^n,
\end{equation}
for $K_1\in\NN$, $M_1,N_1\in\ZZ$, $c^1_{k,m,n}\in\CC$, and $0<|z-1|<1$. At $z=\infty$ we have
\begin{equation}\label{series8}
f(z)=\sum_{k=0}^{K_\infty}\sum_{m=-\infty}^{M_\infty}\sum_{n=-\infty}^{N_\infty} c_{k,m,n}^\infty(\ln z\zz)^kz^m\zz^n,
\end{equation}
for $K_\infty\in\NN$, $M_\infty,N_\infty\in\ZZ$, $c^\infty_{k,m,n}\in\CC$, and $|z|>1$.
\begin{defn}
With the above expansions of $f\in\sA$ the holomorphic and the antiholomorphic residues at $z=a$, $a\in\{0,1,\infty\}$, are
\begin{equation}
\mathrm{res}_a(f)=c_{0,-1,0}^a,\quad\overline{\mathrm{res}}_a(f)=c_{0,0,-1}^a.
\end{equation}
\end{defn}
For certain $f\in\sA$ the two-dimensional integral over the complex plane $\int_\CC f(z)\dd^2z$ is well-defined. In this case the integral is given by residues.
\begin{thm}\label{residuethm}
Assume $f\in\sA$ such that the integral $\int_\CC f(z)\dd^2z$ exists. Let $F\in\sA$ be an antiholomorphic primitive of $f$, $\partial_\zz F(z)=f(z)$. Then
\begin{equation}\label{residueeq}
\frac{1}{\pi}\int_\CC f(z)\dd^2z=\mathrm{res}_\infty(F)-\mathrm{res}_0(F)-\mathrm{res}_1(F).
\end{equation}
\end{thm}
\begin{proof}
Firstly, notice that
$$
f(z)\dd^2z=-\frac{f(z)}{2\mathrm{i}}\dd z\wedge\dd\zz=\dd\frac{F(z)}{2\mathrm{i}}\dd z
$$
is exact on $\PP^1\CC\backslash\{0,1,\infty\}$. Let $S^\pm_a(r)$ be the $\pm$ oriented sphere around $a$ with radius $r$ and $0<\epsilon<1$.
Let $M_\epsilon$ be the oriented manifold with boundaries $S^+_0(\epsilon^{-1})$, $S^-_0(\epsilon)$, and $S^-_1(\epsilon)$. Then $f(z)\dd^2z$ is exact on $M_\epsilon$
and by Stokes' theorem we have
$$
\int_{M_\epsilon} f(z)\dd^2z=\frac{1}{2\mathrm{i}}\left(\int_{S^+_0(\epsilon^{-1})}+\int_{S^-_0(\epsilon)}+\int_{S^-_1(\epsilon)}\right)F(z)\dd z.
$$
Using (\ref{series0}), (\ref{series1}), (\ref{series8}), and the parametrization $S_a(\epsilon)=\{a+\epsilon\mathrm{e}^{\mathrm{i}\phi},\phi\in[0,2\pi)\}$ we obtain
\begin{eqnarray*}
\int_{S^-_a(\epsilon)}F(z)\dd z&=&-2\pi\mathrm{i}\sum_{k,m}c^a_{k,m,m+1}(2\ln\epsilon)^k\epsilon^{2m+2},\hbox{ for }a\in\{0,1\},\\
\int_{S^+_0(\epsilon^{-1})}F(z)\dd z&=&2\pi\mathrm{i}\sum_{k,m}c^\infty_{k,m,m+1}(-2\ln\epsilon)^k\epsilon^{-2m-2}.
\end{eqnarray*}
Because $\int_\CC f(z)\dd^2z$ exists we have for $a=0,1$ that $c^a_{k,m,m+1}=0$ if $m<-1$ or if $m=-1$ and $k\neq0$.
Likewise $c^\infty_{k,m,m+1}=0$ if $m>-1$ or if $m=-1$ and $k\neq0$. In the limit $\epsilon\to0$ only the $m=-1$, $k=0$ terms survive.
These terms give (\ref{residueeq}).
\end{proof}
\begin{remark}
Theorem \ref{residuethm} remains valid if we interchange the role of $z$ and $\zz$.

Moreover, notice that $F=\int f\dd\zz$ always exists in $\sA$. It is unique up to a function $G\in\sO$. After a partial fraction decomposition $G$ is a linear
combination of $z^m$ for $m\in\ZZ$ and $(z-1)^n$ for a negative integer $n$. If $m,n\neq-1$ these terms do not contribute to the residues in (\ref{residueeq}).
If $m,n=-1$ the contributions of these terms cancel in (\ref{residueeq}). This confirms that (\ref{residueeq}) is independent of the choice of $F$.

Note that the right hand side of (\ref{residueeq}) is well-defined even if the integral on the left hand side diverges.
\end{remark}

\begin{ex}
The two-dimensional integral of powers of the Bloch-Wigner dilogarithm $\int_\CC D(z)^n\dd^2z$ exist for $n\geq3$. Because of the reflection symmetry $D(1-z)=-D(z)$ the integrals
of odd powers of $D$ vanish. A computer calculation yields with $g_{335}$ from example \ref{Rsvex}
\begin{eqnarray}
\frac{1}{\pi}\int_\CC D(z)^4\dd^2z&=&\frac{9}{2}\zeta(3)-\frac{27}{4}\zeta(5)+\frac{189}{32}\zeta(7),\nonumber\\
\frac{1}{\pi}\int_\CC D(z)^6\dd^2z&=&-\frac{2025}{8}\zeta(5)+\frac{17145}{64}\zeta(7)-\frac{585}{64}\zeta(9)-\frac{3304683}{1024}\zeta(11)\nonumber\\
&&+\,135\zeta(3)^2\zeta(5)+\frac{81}{2}g_{335}.
\end{eqnarray}
Note that by theorem \ref{Stabilitythm} and (\ref{BWdilog}) the right hand side is in $\sH^\sv(\ZZ)/4^n$. The right hand side of (\ref{residueeq}) gives for $n=2$ the value
$\zeta(3)/2$.
\end{ex}

\subsection{Stability of SVMPs with coefficients in $\sH^\sv(\ZZ)$}
SVMPs over the ring $\sH^\sv(\ZZ)$ (see definition \ref{Rsvdef}) are stable under canonical operations.
\begin{defn}\label{Adef}
Let
\begin{equation}
\sP^\sv=\langle P_w,\;w\in X^\ast\rangle_{\sH^\sv(\ZZ)}
\end{equation}
be the shuffle ring of SVMPs over $\sH^\sv(\ZZ)$ and
\begin{equation}
\sA^\sv=\sO_\QQ\overline{\sO}_\QQ\sP^\sv,
\end{equation}
where $\sO_\QQ=\QQ[z,\frac{1}{z},\frac{1}{z-1}]$ and $\overline{\sO}_\QQ=\QQ[\zz,\frac{1}{\zz},\frac{1}{\zz-1}]$, its bi-differential $\QQ$ algebra.
Let $\sH^\sv(\ZZ)_n$, $\sP_n^\sv$ and $\sA_n^\sv$ denote the subspaces of total weight $n$.
\end{defn}
We do not assume that the decomposition into total weight subspaces is direct. In the following theorem and the subsequent corollary we prove that the natural
number theoretic framework of SVMPs are the ring $\sH^\sv(\ZZ)$ and the $\QQ$ algebra $\sH^\sv$ of single-valued MZVs.
\begin{thm}\label{Stabilitythm}
The $\ZZ$ module $\sP^\sv_n$ is stable under the $\sS_3$ group of M\"obius transformations permuting $\{0,1,\infty\}$. The integrals
$\int_a\dd z/z$, $\int_a\dd z/(z-1)$, $\int_a\dd \zz/\zz$, $\int_a\dd \zz/(\zz-1)$ for $a\in\{0,1,\infty\}$ map $\sP^\sv_n$ into $\sP^\sv_{n+1}$. Moreover,
\begin{equation}\label{diffstab}
\partial_z\sP^\sv_n=\frac{\sP^\sv_{n-1}}{z}+\frac{\sP^\sv_{n-1}}{z-1},\quad\partial_\zz\sP^\sv_n=\frac{\sP^\sv_{n-1}}{\zz}+\frac{\sP^\sv_{n-1}}{\zz-1}.
\end{equation}
We have
\begin{equation}
\sP_n^\sv(0)=\sP_n^\sv(1)=\sP_n^\sv(\infty)=\sH^\sv(\ZZ)_n\subset\sH(\ZZ)_n,
\end{equation}
where $\sH(\ZZ)_n$ is the $\ZZ$ module of integer MZVs of weight $n$. The series of $\x_1'$ is single-valued,
\begin{equation}\label{x1primeinLie}
\x_1'\in\Lie_{\sH^\sv(\ZZ)}\langle\langle X\rangle\rangle,
\end{equation}
with total weight $-1$. Finally, for $f\in\frac{\sP^\sv}{z^\alpha\zz^\beta(z-1)^\gamma(\zz-1)^\delta}$ with $\alpha,\beta,\gamma,\delta\in\{0,1,2\}$,
\begin{equation}\label{resstab}
\mathrm{res}_a(f)\in\sH^\sv(\ZZ),\quad\overline{\mathrm{res}}_a(f)\in\sH^\sv(\ZZ),\quad\hbox{for }a\in\{0,1,\infty\}\\
\end{equation}
and
\begin{equation}\label{2dintstab}
\frac{1}{\pi}\int_\CC f(z)\dd^2z\in\sH^\sv(\ZZ).
\end{equation}
\end{thm}
\begin{proof}
The following statements are obvious: the stability of $\sP^\sv$ under the transformation $z\to\frac{z}{z-1}$ by (\ref{S3trafos}) and the stability of $\sP^\sv$
under $\int_0\dd z/z$, $\int_0\dd z/(z-1)$. $\sP^\sv(0)=\sH^\sv(\ZZ)$ is trivial, $\sP^\sv(1)=\sH^\sv(\ZZ)$ because $\sH^\sv(\ZZ)$ is a ring and $\sP^\sv(\infty)=\sH^\sv(\ZZ)$ because
$\sP^\sv(\infty)=\sP^\sv(1)$ by (\ref{S3trafos}). From lemma \ref{Lielem} we have $\x_1'\in\Lie_{\sH(\ZZ)}\langle\langle X\rangle\rangle$ hence $\sH^\sv(\ZZ)=\sP^\sv(1)\subset\sH(\ZZ)$.
The integrals $\int_1\dd z/(z-a)$, $\int_\infty\dd z/(z-a)$, $a\in\{0,1\}$, differ from $\int_0\dd z/(z-a)$ by a value in $\sP^\sv(1)$, or $\sP^\sv(\infty)$, respectively.
The first equation in (\ref{diffstab}) is also clear.

The group $\sS_3$ is generated by $z\to\frac{z}{z-1}$ and $z\to1-z$. The stability of $\sP^\sv$ under $z\to1-z$ is proved by induction over the total weight.
The statement is trivial for weight zero. Assume $P\in\sP_n$. Then from
$$
P(z)=\int_1\left(\frac{Q_0(z)}{z}+\frac{Q_1(z)}{z-1}\right)\dd z+P(1)
$$
we have $Q_0,Q_1\in\sP^\sv_{n-1}$ and
$$
P(1-z)=\int_0\left(\frac{Q_0(1-z)}{z-1}+\frac{Q_1(1-z)}{z}\right)\dd z+P(1)\in\sP^\sv_n
$$
by induction.

Next, we show by induction that $\x_1'\in\Lie_{\sH^\sv(\ZZ)}\langle\langle X\rangle\rangle$. Consider $P_w$ for a word $w$ of length $n$.
We start with the observation from (\ref{diffonaw}) that for any word $u$, $\lim_{z\to0}\zz\partial_\zz P_{au}(z)=\delta_{a,0}P_u(0)\in\sH^\sv(\ZZ)$. By stability of $\sP^\sv$
under $z\to1-z$ we obtain
$$
\lim_{z\to1}(\zz-1)\partial_\zz P_w(z)=\lim_{z\to0}\zz\partial_\zz P_w(1-z)\in\sH^\sv(\ZZ).
$$
On the other hand from (\ref{diffonaw}) we obtain
$$
\lim_{z\to1}(\zz-1)\partial_\zz P_w(z)=\sum_{w=uv}(\x_1'|u)P_v(1)=(\x_1'|w)+\sum_{\genfrac{}{}{0pt}{}{w=uv}{|u|<|w|}}(\x_1'|u)P_v(1).
$$
This proves (\ref{x1primeinLie}) by induction over the length of $w$.
It was already proved in corollary \ref{homcor} that $\x_1$ is of total weight $-1$ (the weight of $\x_1$).

From (\ref{diffonaw}) we obtain the second equation in (\ref{diffstab}). It follows by induction and (\ref{intonw}) that the antiholomorphic integrals
$\int_a\dd \zz/\zz$, $\int_a\dd \zz/(\zz-1)$ map from $\sP^\sv_n$ into $\sP^\sv_{n+1}$.

The stability of the residue (\ref{resstab}) follows for $a=0$ by expanding all multiple polylogarithms $L_w(z)$ and $L_{w'}(\zz)$ in $f$ up to the linear term:
In (\ref{explicitLw}) only the terms with $k_0=0$ contribute to the residue so that this expansion does not generate denominators.
At $a=1$ the result follows from $a=0$ by $\mathrm{res}_1f(z)=-\mathrm{res}_0f(1-z)$.
At $a=\infty$ the result follows from $a=0$ by $\mathrm{res}_\infty f(z)=\mathrm{res}_0 z^{-2}f(1/z)$. The result for $\overline{\mathrm{res}}$ follows by complex
conjugation.

For (\ref{2dintstab}) we first prove using integration by parts and induction over the weight that $\int_0\dd\zz(\zz-a)^m$ for $m\in\{-2,-1,0\}$
and $a\in\{0,1\}$ maps from $\sP^\sv$ into $\sP^\sv(\zz-a)^{m+1}+\sP^\sv$.
Then we use the residue theorem \ref{residuethm}. Repeating the argument that lead to (\ref{resstab}) the residues in $z$ of $\int_0f(z)\dd\zz$ are in $\sH^\sv(\ZZ)$.
\end{proof}

\begin{cor}
The $\QQ$ algebra $\sA^\sv$ is stable under $\sS_3$ transformations, $\partial_z$, $\partial_\zz$, $\int_a\dd z$, $\int_a\dd \zz$ for $a\in\{0,1,\infty\}$. We have
\begin{equation}\label{Asveq1}
\sA^\sv_n(0)=\sA^\sv_n(1)=\sA^\sv_n(\infty)=\sH^\sv_n\subset\sH_n
\end{equation}
whenever the regularized limits exist. For $f\in\sA^\sv$ and $a\in\{0,1,\infty\}$,
\begin{equation}\label{Asveq2}
\mathrm{res}_af\in\sH^\sv,\quad\overline{\mathrm{res}}_af\in\sH^\sv,\quad\frac{1}{\pi}\int_\CC f(z)\dd^2z\in\sH^\sv.
\end{equation}
\end{cor}
\begin{proof}
The corollary follows from the previous theorem by taking the tensor product over $\ZZ$ with $\QQ$ and using integration by parts.
To show (\ref{Asveq1}) and (\ref{Asveq2}) one also needs (\ref{explicitLw}).
\end{proof}

\section{Graphical functions}\label{Gf}
\subsection{Definition and convergence}\label{convergence}
In this subsection we define graphical functions in
$$d=2\lambda+2>2$$
dimensions. Although we are mainly interested in four dimensions, graphical functions exist in any dimensions greater than or equal to two.
Two dimensions, however, are a special case which is postponed to \S \ref{2d}.

The position space propagator of an edge $e$ with vertices $x$ and $y$ is
$$\frac{1}{Q_e^\lambda}=\frac{1}{||x-y||^{2\lambda}}.$$
The power in the denominator originates from Fourier transforming a $||p||^{-2}$ momentum space propagator.
We first define uncompleted graphical functions and turn to completion in \S \ref{completion}.
\begin{defn}\label{fGdef}
The graphical function $f_G^{(\lambda)}(z)$ is defined by applying position space Feynman rules to the graph $G$ with distinguished vertices
$0,1,z$, namely
\begin{equation}\label{fdefnoncomp}
f_G^{(\lambda)}(z)=\left(\prod_{v\notin\{0,1,z\}} \int_{\RR^d}\frac{\dd^dx_v}{\pi^{d/2}}\right)\frac{1}{\prod_eQ_e^{\lambda}},
\end{equation}
where the first product is over all vertices $\neq0,1,z$ and the second product is over all edges of $G$.
The vertex 0 corresponds to the origin in $\RR^d$ whereas $1$ stands for any unit vector $e_1\in\RR^d$ and $z\neq0,e_1$.
\end{defn}

By rotational symmetry $f_G^{(\lambda)}(z)$ only depends on $||z||$ and the angle between $z$ and $e_1$. If we identify $e_1$ with $1\in\CC$ we can consider
$f_G^{(\lambda)}(z)$ as a function on the complex plane with the symmetry $f_G^{(\lambda)}(z)=f_G^{(\lambda)}(\zz)$ under complex conjugation.
In the following we always do so unless the argument of $f_G^{(\lambda)}$ is explicitly specified as a $d$-dimensional vector.
We can calculate the graphical functions of a complex argument $z$ by the integral (\ref{fdefnoncomp}) with the identification $\RR^d\cong\CC\times\RR^{d-2}$,
\begin{equation}\label{coords}
e_1\sim(1,0^{\{d-1\}}),\quad z\sim\left(\frac{z+\zz}{2},\frac{z-\zz}{2i},0^{\{d-2\}}\right).
\end{equation}
In four dimensions ($\lambda=1$) we often drop the superscript $(1)$.

A variant of graphical functions can be defined in momentum space where massive propagators are algebraic. We do not pursue this here.

For graphical functions in $N=4$ supersymmetric Yang-Mills theory see \cite{Drummond}, \cite{SYM}.

\begin{ex}\label{exI}
Our main examples are sequential graphs (see figure 4). The initial case is
\begin{equation}\label{Ieq}
f_{\mathrm{I}}^{(\lambda)}(z)=\frac{1}{[z\zz(z-1)(\zz-1)]^\lambda}.
\end{equation}
Sequentially appending edges gives for any word $w$ that begins with 2 the graphical functions $f^{(\lambda)}_w(z)$ as a $d|w|$-dimensional integral.
In particular,
$$f^{(\lambda)}_2(z)=\frac{1}{\pi^{d/2}}\int_{\RR^d}\frac{\dd^dx}{||x||^{2\lambda}||x-1||^{2\lambda}||x-z||^{2\lambda}}$$
which can be calculated with Gegenbauer techniques (see example \ref{onevertex}).
\end{ex}

Next, we give a criterion when the integral (\ref{fdefnoncomp}) is well-defined. We need the following definition:
\begin{defn}\label{internaldef}
For a subgraph $g$ of a graph $G$ a vertex $v$ is called internal if $v\neq0,1,z$ and all edges adjacent to $v$ in $G$ are in $g$. All other vertices
are called external. An edge in $g$ is internal if both vertices are internal otherwise it is external.
\end{defn}
Note that the notion `internal' crucially depends on the subgraph $g$. A vertex that is internal in $g$ may fail to be internal in a smaller subgraph.
Only if $g=G$ all vertices $\neq0,1,z$ are internal.
\begin{lem}\label{finitelem}
The integral in definition \ref{fdefnoncomp} is convergent if and only if it is `infrared' and `ultraviolet' finite.
The integral $f_G^{(\lambda)}$ is infrared finite if and only if  for every subgraph $g$ of $G$ with $V^{\mathrm{int}}_g$ internal vertices and $N_g>0$ edges,
\begin{equation}\label{infrared}
(d-2)N_g>dV^{\mathrm{int}}_g.
\end{equation}
The integral $f_G^{(\lambda)}$ is ultraviolet finite if and only if  for every subgraph $g$ of $G$ with $N_g>0$ and with at most one of its $V_g$ vertices in $\{0,1,z\}$ one has
\begin{equation}\label{ultraviolet}
(d-2)N_g<d(V_g-1).
\end{equation}
\end{lem}
\begin{proof}
In \cite{Wein} it is shown that a Feynman integral is finite if the integral is finite by power-counting for every subgraph $g$ of $G$.
If $N_g=0$ convergence is trivial. Otherwise there are two sources for divergences: infrared, when the integral does not converge for large values of the integration variables;
ultraviolet, when the coincidence of integration variables generates a singularity that cannot be integrated over.

Infrared convergence by power-counting is equivalent to the condition that the sum of degrees of the denominators is larger than the dimension of the integral.
This gives (\ref{infrared}). If $g$ is minimal with an ultraviolet divergence by power-counting then the integral diverges when all vertices in $g$ coincide in the integral of $f_g$.
Because $\{0,1,z\}$ are distinct $g$ has at most one labeled vertex.
The locus where all integration variables coincide is $d(V_g-1)$ dimensional. Power-counting gives (\ref{ultraviolet}).
\end{proof}
\begin{remark}
A $\phi^4$ graph in $d=4$ dimensions is never infrared divergent: counting half-edges gives $2N_g^{\mathrm{int}}+N_g^{\mathrm{ext}}=4V_g^{\mathrm{int}}$ which
implies (\ref{infrared}) because the number $N_g^{\mathrm{ext}}$ of external edges is positive.
\end{remark}
\begin{lem}\label{convergencelemma}
The graphical function $f^{(1)}_w(z)$ of a non-empty word $w$ is well-defined in $d=4$ dimensions if and only if $w$ begins with 2.
\end{lem}
\begin{proof}
If $w$ does not begin with 2 then the leftmost vertex has valence two leading to an infrared divergence.

Let $w$ begin with 2. To show infrared finiteness we observe that every internal vertex has at least valence three. Counting half-edges gives
$$3V^{\mathrm{int}}_g\leq 2N^{\mathrm{int}}_g+N^{\mathrm{ext}}_g.$$
Moreover, every internal vertex is connected to at least one external vertex 0 or 1,
$$V^{\mathrm{int}}_g<N^{\mathrm{ext}}_g.$$
The inequality is strict because either $V^{\mathrm{int}}_g=0$ or the leftmost vertex in $V_g^{\mathrm{int}}$ has at least two external edges.
Together we obtain
$$
4V^{\mathrm{int}}_g<2N^{\mathrm{int}}_g+2N^{\mathrm{ext}}_g=2N_g.
$$

Ultraviolet convergence is trivial if neither 0 nor 1 is in $g$ because in this case $0<N_g\leq V_g-1$.
Otherwise we can assume without restriction that $0\in g$. Let $v$ be the valence of 0 in $g$. Every other vertex in $g$ has at most
valence three, hence by counting half-edges
$$
v+3(V_g-1)>2N_g.
$$
The inequality is strict because the leftmost vertex in $g$ has valence at most two. Moreover, $v\leq V_g-1$ because every edge in $g$ which is connected to 0 is adjacent
to a vertex in $g\backslash\{0\}$. Together we obtain
$$
4(V_g-1)>2N_g.
$$
\end{proof}
\begin{cor}\label{Pwelldef}
The sequential period $P(G_w)$ of a non-empty word $w$ is well-defined in four dimensions if and only if $w$ begins and ends in 2.
\end{cor}
\begin{proof}
We add a disconnected vertex $z$ to $G_w$ to make $P_w$ a constant graphical function. With lemma \ref{finitelem} the proof is analogous to the proof of corollary \ref{convergencelemma}.
\end{proof}

\subsection{General properties}
Graphical functions have the following general properties
\begin{lem}\label{prop123}
Let $G$ be a graph with three distinguished vertices $0,1,z$ which has a graphical function $f_G^{(\lambda)}:\CC\backslash\{0,1\}\to\RR$. Then
\begin{enumerate}
\item[(G1)]
\begin{equation}\label{reflection}
f_G^{(\lambda)}(z)=f_G^{(\lambda)}(\zz).
\end{equation}
\item[(G2)] $f_G^{(\lambda)}$ is a single-valued.
\item[(G3)] $f_G^{(\lambda)}$ is real analytic in $\CC\backslash\{0,1\}$.
\end{enumerate}
\end{lem}
\begin{proof}
(G1) follows from using coordinates (\ref{coords}) in (\ref{fdefnoncomp}) and the substitution $x_v=(x^1_v,x^2_v,x^3_v,\ldots,x^d_v)\mapsto(x^1_v,-x^2_v,x^3_v,\ldots,x^d_v)$.

To prove (G2) we first consider the graphical function as a function of $z\in\RR^d$.
We add a small $\epsilon>0$ to the quadrics in the denominator of the propagators so that the singular locus of the integrand does not intersect the chain of integration.
If we vary the argument $z$ along a closed smooth path in $\RR^d$ the graphical function changes smoothly until it is back to its initial value.
Transition to $\CC$ by the decomposition $\RR^d\cong\CC\otimes\RR^{d-2}$ yields (G2) in the limit $\epsilon\to0$.

(G3) can be proved in general using parametric representations of graphical functions \cite{PropGF}. We will only need the result for sequential functions
which we state separately in the following corollary. For self-containedness we prove the corollary in the next subsection.
\end{proof}

\begin{cor}\label{realanalyticcor}
Let $w$ be a word in 0,1,2 that begins with 2. Then $f_w^{(1)}(z)$ is real analytic in $\CC\backslash\{0,1\}$.
\end{cor}

\subsection{Gegenbauer polynomials}
Most conveniently Gegenbauer polynomials are defined by the following generating series \cite{C4}, \cite{Vilenkin}, \cite{Askey},
\begin{equation}\label{seriesdef}
\frac{1}{(1-2xt+t^2)^\lambda}=\sum_{n=0}^\infty C_n^{(\lambda)}(x)t^n.
\end{equation}
With (\ref{seriesdef}) we can expand a position space propagator in Gegenbauer polynomials. Let $\angle(x,y)$ be the angle between $x$ and $y$ and let $x_<=\min\{x,1/x\}$. Then
\begin{equation}\label{propexpand}
\frac{1}{||x-y||^{2\lambda}}=\frac{1}{||x||^\lambda||y||^\lambda}\sum_{n=0}^\infty C_n^{(\lambda)}(\cos\angle(x,y))\left(\frac{||x||}{||y||}\right)^{n+\lambda}_{\!<}.
\end{equation}
The main property of Gegenbauer polynomials is their orthogonality with respect to $d$-dimensional angular integrations. If we denote the integral
over the $d-1$-dimensional unit-sphere in $d$ dimensions by $\int\dd\hat{x}$, normalized by $\int1\dd\hat{x}=1$, we have
\begin{equation}\label{ortho}
\int\dd\hat{y}\;C_m^{(\lambda)}(\cos\angle(x,y))C_n^{(\lambda)}(\cos\angle(y,z))=\frac{\lambda\delta_{m,n}}{n+\lambda}C_n^{(\lambda)}(\cos\angle(x,z)).
\end{equation}

\begin{ex}\label{onevertex}
The sequential function $f_2^{(\lambda)}$ exists for all $\lambda>1/2$. It can be calculated with Gegenbauer techniques.
Series expansion of $||x-e_1||^{-2\lambda}$ and $||x-z||^{-2\lambda}$ with (\ref{ortho}) plus an elementary radial integration gives
for $||z||<1$ (one may specify $a_{\ell,m,n}=b_{\ell,m,n}=\delta_{\ell,0}\delta_{m,n}$, $p=0$, $q=2\lambda$, $\alpha=\lambda$ in (\ref{Gegenbauerint1}))
\begin{equation}
f_2^{(1)}(z)=2\sum_{n=0}^\infty C^{(1)}_n(\cos\angle(z,e_1))\left(\frac{1}{(n+1)^2}-\frac{\ln||z||}{n+1}\right)||z||^n
\end{equation}
and for $\lambda>1$
\begin{equation}
f_2^{(\lambda)}(z)=\frac{1}{\Gamma(\lambda)(\lambda-1)}\sum_{n=0}^\infty C^{(\lambda)}_n(\cos\angle(z,e_1))\left(\frac{||z||^{n-2\lambda+2}}{n+1}-\frac{||z||^n}{n+2\lambda-1}\right).
\end{equation}
The result for $||z||>1$ can be deduced from $||z||<1$ by the inversion formula
\begin{equation}
f_2^{(\lambda)}(z/||z||^2)=||z||^{4\lambda-2}f_2^{(\lambda)}(z),
\end{equation}
which follows from the defining integral by a coordinate transformation $x\mapsto x/||z||$.

Now we set $z\in\CC$ and $e_1=1$. In the case $\lambda=1$ we use the identity
\begin{equation*}
C^{(1)}_n(\cos\angle(z,1))|z|^n=\frac{z^{n+1}-\zz^{n+1}}{z-\zz}
\end{equation*}
which we obtain from the generating series. This relates $f_2^{(1)}$ to the dilogarithm,
\begin{equation}\label{f24dim}
f_2^{(1)}(z)=\frac{2}{z-\zz}(\Li_2(z)-\Li_2(\zz)+[\ln(1-z)-\ln(1-\zz)]\ln|z|).
\end{equation}
The combination of logarithms and dilogarithms in the bracket is $2\mathrm{i}$ times the Bloch-Wigner dilogarithm (\ref{BWdilog}). We obtain
\begin{equation}\label{4iD}
f_2^{(1)}(z)=\frac{4iD(z)}{z-\zz}.
\end{equation}
By the symmetry of $D$ under inversion $z\mapsto 1/z$ the above equation holds also for $|z|>1$. We will re-derive equation (\ref{4iD}) using SVMPs in example \ref{f2example}.

In the case $\lambda>1$ we obtain from the generating series of the Gegenbauer polynomials by elementary integration for $z\in\CC$ and $|z|<1$,
\begin{equation}
\sum_{n=0}^\infty\frac{C^{(\lambda)}_n(\cos\angle(z,1))|z|^n}{n+a}=\int_0^1\frac{t^{a-1}\dd t}{(1-tz)^\lambda(1-t\zz)^\lambda}.
\end{equation}
This allows us to express $f_2^{(\lambda)}$ as an integral,
\begin{equation}\label{f2lambda}
f_2^{(\lambda)}(z)=\frac{1}{\Gamma(\lambda)(\lambda-1)}\int_0^1\frac{(z\zz)^{1-\lambda}-t^{2\lambda-2}}{(1-tz)^\lambda(1-t\zz)^\lambda}\dd t.
\end{equation}
For integer $\lambda$ the above integral has vanishing residues in $t$. For $\lambda=2$ we obtain
\begin{equation*}
f_2^{(2)}(z)=\frac{1}{z\zz(z-1)(\zz-1)}.
\end{equation*}
The simplicity of the result stems from the fact that in $d=6$ dimensions the three-valent vertex is `unique' in the sense of \cite{Kazakov}.
For $\lambda=3$ we have
\begin{equation*}
f_2^{(3)}(z)=\frac{2z\zz-z-\zz+2}{2[z\zz(z-1)(\zz-1)]^2}.
\end{equation*}
For general even dimensions greater than four we obtain a rational expression in $z$ and $\zz$ with singularities in $z=0$ and $z=1$.
It is trivially single-valued and extends unchanged
to $|z|>1$.

For odd dimensions one also obtains a logarithm-free result. It is a rational expression in $z,\zz,|z|,|z-1|$ with singularities at $z=0$ and $z=1$
which extends unchanged to $|z|>1$. The expression is explicitly single-valued.
\end{ex}

To prove corollary \ref{realanalyticcor} we need to define certain classes of functions
\begin{defn}\label{Cdef}
A function $f:\RR^d\to\RR$ is in $\sC_{p,q}^{(\lambda)}$ for $p,q\in\ZZ$ if there exist constants $A,B,\alpha,\beta>0$, $L\in\NN$ such that $f$ admits the following expansions,
\begin{equation}\label{eqle1}
f(z)=\sum_{\ell=0}^L\sum_{m=0}^\infty\sum_{n=0}^ma_{\ell,m,n}(\ln||z||)^\ell||z||^{m-p}C_n^{(\lambda)}(\cos\angle(z,e_1))\quad\hbox{for }||z||<1
\end{equation}
and
\begin{equation}\label{eqgr1}
f(z)=\sum_{\ell=0}^L\sum_{m=0}^\infty\sum_{n=0}^mb_{\ell,m,n}(\ln||z||)^\ell||z||^{-m-q}C_n^{(\lambda)}(\cos\angle(z,e_1))\quad\hbox{for }||z||>1,
\end{equation}
with
\begin{equation}\label{conditionAB}
|a_{\ell,m,n}|\leq Am^\alpha\quad\hbox{and}\quad |b_{\ell,m,n}|\leq Bm^\beta\quad\hbox{for all }\ell\leq L,m,n\leq m.
\end{equation}
\end{defn}

\begin{ex}\label{expluse1}
By (\ref{propexpand}) we have $||x-e_1||^{-2\lambda}\in\sC_{0,2\lambda}^{(\lambda)}$.
Replacing $t$ by $-t$ in the generating series (\ref{seriesdef}) gives
\begin{equation}
\frac{1}{||x+e_1||^{2\lambda}}=\frac{1}{||x||^\lambda}\sum_{n=0}^\infty (-1)^nC_n^{(\lambda)}(\cos\angle(x,e_1))||x||^{n+\lambda}_<,
\end{equation}
hence also $||x+e_1||^{-2\lambda}\in\sC_{0,2\lambda}^{(\lambda)}$.
\end{ex}

\begin{prop}\label{prop1}
If $f$ in $\sC_{p,q}^{(\lambda)}$ for some $p,q\in\ZZ$ then $f$ is real analytic in $\RR^d\backslash(\{0\}\cup\{||z||=1\})$.
\end{prop}
\begin{proof}
Assume without restriction that $e_1=(1,0,\ldots,0)$. The Gegenbauer polynomials $C_n^{(\lambda)}(\cos\angle(z,e_1))$ are polynomials in $z^1/||z||$ and hence real analytic for $z\neq0$.
The norm $||z||$ is real analytic for $z\neq0$ and so is $\ln||z||$. Due to (\ref{conditionAB}) the sum in (\ref{eqle1}) is absolutely convergent for
$||z||<1$ and the sum in (\ref{eqgr1}) is absolutely convergent for $||z||>1$. Hence the analytic expansions commute with the sums and
the claim follows.
\end{proof}

\begin{prop}\label{prop2}
The function classes $\sC_{p,q}^{(\lambda)}$ have the following properties:
\begin{enumerate}
\item[(C1)]
$$\sC_{p_1,q_1}^{(\lambda)}\sC_{p_2,q_2}^{(\lambda)}\subseteq\sC_{p_1+p_2,q_1+q_2}^{(\lambda)}$$
\item[(C2)] If $f\in\sC_{p,q}^{(\lambda)}$ and $\alpha\in\RR$ with
\begin{equation}\label{alphacondition}
2-q<2\alpha<2\lambda+2-p
\end{equation}
then
\begin{equation}\label{addedge}
\int_{\RR^d}\frac{\dd^dx}{\pi^{d/2}}\frac{f(x)}{||x||^{2\alpha}||x-z||^{2\lambda}}\in\sC^{(\lambda)}_{\max\{0,p+2\alpha-2\},\min\{2\lambda,q+2\alpha-2\}}.
\end{equation}
\end{enumerate}
\end{prop}
The proof of the proposition is technical, so we moved it to appendix A.

\begin{prop}\label{prop3}
Let $w$ be a word that begins with 2. Then $f_w^{(1)}\in\sC_{0,2}^{(1)}$.
\end{prop}
\begin{proof}
We use example \ref{expluse1} and (\ref{addedge}) to prove the result for $f_2$.
The general case follows by straight forward induction over the length of $w$ using proposition \ref{prop2}.
\end{proof}

\begin{prop}\label{prop4}
Let $w$ be a word that begins with 2 and let $g(y)=f_w^{(1)}((y+e_1)/2)$. Then $g\in\sC_{0,2}^{(1)}$.
\end{prop}
\begin{proof}
We substitute $x_v\mapsto (x_v+e_1)/2$ in the integral (\ref{fdefnoncomp}) defining $f_w^{(1)}$.
The integration measure changes by a constant. After extracting all factors $1/2$ from the propagators we find that
a propagator from $x_v$ to 0 is changed to a propagator from $x_v$ to $-e_1$ and the propagator from
$x_n$ (say) to $z$ is changed to a propagator from $x_n$ to $2z-e_1$. After a substitution $z=(y+e_1)/2$ the latter becomes a propagator from $x_n$ to $y$.
All other propagators remain intact. Because the propagator from $x_1$ to $-e_1$ is in $\sC_{0,2}^{(1)}$ (example \ref{expluse1}) we can use
induction over the length of the word $w$ to prove the result with proposition \ref{prop2}
\end{proof}

\begin{proof}[Proof of corollary \ref{realanalyticcor}]
By propositions \ref{prop1} and \ref{prop3} $f_w^{(1)}(z)$ is real analytic for all $z$ except 0 and the unit sphere.
By propositions \ref{prop1} and \ref{prop4} $f_w^{(1)}((y+e_1)/2)$ is real analytic for all $y$ except 0 and the unit sphere in $y$.
Hence $f_w^{(1)}(z)$ is real analytic except for $z=e_1/2$ and the vectors $z$ with $||2z-e_1||=1$ which is a sphere of radius $1/2$ around $e_1/2$.
The intersection of the two loci of possible non-analycities are 0 and $e_1$. Transition to $\CC$ proves the lemma.
\end{proof}

\subsection{Completion}\label{completion}
There exists a 24-fold symmetry relating graphical functions of different graphs. For the formulation of this symmetry
we need to complete the graph in a way similar to the completion of periods.

Completion adds a vertex with label $\infty$. We weight the edges of the graph by real exponents. Positive integer weights may be graphically
represented by multiple lines. For the weight $-1$ which lifts the edge-quadric into the numerator we choose to draw a wavy line. Integer negative weights may
be represented by multiple wavy lines (see figure 5). We will see that integer weights suffice to complete graphs in three and four dimensions.

\begin{defn}
A graph $\Gamma$ is completed in $d=2\lambda+2$ dimensions if it has the labels $0,1,z,\infty$ and weights such that every unlabeled (internal) vertex has weighted valence
\begin{equation}\label{nuint}
\nu^{\mathrm{int}}=\frac{2d}{d-2}
\end{equation}
and every labeled (external) vertex has valence
\begin{equation}\label{nuext}
\nu^{\mathrm{ext}}=0.
\end{equation}
The graphical function of $\Gamma$ is
\begin{equation}\label{fGamma}
f_\Gamma^{(\lambda)}(z)=\left(\prod_{v\notin\{0,1,z,\infty\}} \int_{\RR^d}\frac{\dd^dx_v}{\pi^{d/2}}\right)\frac{1}{\prod_eQ_e^{\lambda\nu_e}}.
\end{equation}
If $\Gamma\backslash\{\infty\}$ equals $G$ up to an edge from 0 to 1 of any weight then $f_\Gamma^{(\lambda)}=f_G^{(\lambda)}$ (definition \ref{fGdef}).
In this case $\Gamma$ is the completion of $G$.
\end{defn}
Because $\Gamma$, in contrast to $G$, has a vertex $\infty$ we may use the same symbol $f^{(\lambda)}_\bullet$ for completed and for uncompleted graphical functions.

Every graph has a unique completion.
\begin{lem}\label{completionlemma}
Let $G$ be a graph with three labeled vertices $0,1,z$. Then $G$ has a unique completion. In dimensions three and four the completed graph has integer edge-weights
if $G$ has integer edge-weights.
\end{lem}
\begin{proof}
Let $v$ be an internal vertex of $G$ with weighted edge-valence $n_v$. Completion uniquely connects $v$ to $\infty$ by an edge of weight $\nu_{v\infty}=2d/(d-2)-n_v$.
Clearly, in three or four dimensions $\nu_{v\infty}\in\ZZ$ if $n_v\in\ZZ$. Let the vertex $z$ in $G$ have weighted edge-valence $n_z$. Then
completion connects $z$ to $\infty$ by an edge of weight $-n_z$ which is integer if $n_z$ is.
With these edges we partially complete $G$ to $\Gamma_0$ in which the vertices $0,1,\infty$ have weighted valence $n_0,n_1,n_\infty$, respectively.
In $\Gamma$ we have to add to $\Gamma_0$ a weighted triangle $0,1,\infty$ with the edge-weights
$$
\nu_{01}=\frac{-n_0-n_1+n_\infty}{2},\quad\nu_{0\infty}=\frac{-n_0+n_1-n_\infty}{2},\quad\nu_{1\infty}=\frac{n_0-n_1-n_\infty}{2}.
$$
This completes the graph. The completion is unique: Changing the edge-weights $\nu_{ij}$ by $\delta\nu_{ij}$ ($i,j\in\{0,1,\infty\}$) without changing the valence
$\nu_i=0$ of the external vertices leads to the non-degenerate linear system
$$\delta\nu_{01}+\delta\nu_{0\infty}=0,\quad\delta\nu_{01}+\delta\nu_{1\infty}=0,\quad\delta\nu_{0\infty}+\delta\nu_{1\infty}=0.$$
Finally, counting weighted half-edges in $\Gamma_0$ gives
$$
\frac{2d}{d-2}V^{\mathrm{int}}_{\Gamma_0}+n_0+n_1+n_\infty=2N_{\Gamma_0},
$$
where $V^{\mathrm{int}}_{\Gamma_0}$ is the number of internal vertices in $\Gamma_0$ and $N_{\Gamma_0}$ is the weighted sum of edges in $\Gamma_0$.
In three and four dimensions we have $\frac{d}{d-2}V^{\mathrm{int}}_{\Gamma_0}\in\ZZ$ and hence $n_0+n_1+n_\infty\in2\ZZ$ if $N_{\Gamma_0}\in\ZZ$.
This gives the integrality of $\nu_{01}$, $\nu_{0\infty}$, and $n_{1\infty}$.
\end{proof}
\begin{lem}
A completed graph $\Gamma$ has a convergent graphical function (\ref{fGamma}) if and only if every subgraph $\gamma$ of $\Gamma$ with at most one of its
$V_\gamma\geq2$ vertices in $\{0,1,z,\infty\}$ has
\begin{equation}\label{ultravioletcompleted}
(d-2)N_\gamma<d(V_\gamma-1),
\end{equation}
where $N_\gamma$ is the weighted edge sum of $\gamma$.
\end{lem}
\begin{proof}
For subgraphs $\gamma$ with $\infty\notin\gamma$ condition (\ref{ultravioletcompleted}) is equivalent to (\ref{ultraviolet}) for subgraphs in $G=\Gamma\backslash\{\infty\}$.
It remains to show that for $\gamma$ with $\infty\in\gamma$ condition (\ref{ultravioletcompleted}) is equivalent to (\ref{infrared}).

Let $g$ be a subgraph of $G$. To check infrared finiteness we can assume without restriction that every external edge of $g$ connects to exactly one internal vertex.
From $g$ we construct a graph $\gamma$ by connecting all internal vertices to $\infty$ such that their weighted valence becomes $\frac{2d}{d-2}$ and then cutting all other external edges.
Now, $\gamma$ is a subgraph of $\Gamma$ which connects to $\infty$ but not to $0,1,z$. Conversely deleting $\infty$ in $\gamma$ and
adding external edges such that all vertices become internal leads back to $g$. We have to show that (\ref{ultravioletcompleted}) for $\gamma$ is equivalent to (\ref{infrared})
for $g$.

We have $N_\gamma=N_\infty+N^{\mathrm{int}}_g$ is the weighted sum of edges that connect $\gamma$ with $\infty$ plus the number of internal edges in $g$.
Moreover, by completion we have
$$
N_\infty=\sum_{v\,\mathrm{int}}\left(\frac{2d}{d-2}-n_v\right),
$$
where the sum is over internal vertices $v$ of $g$ and $n_v$ is the valence of $v$ in $g$. Because
$$
V_\gamma-1=V^{\mathrm{int}}_g=\sum_{v\,\mathrm{int}}1
$$
we have that (\ref{ultravioletcompleted}) is equivalent to
$$
dV^{\mathrm{int}}_g<(d-2)\left[\left(\sum_{v\,\mathrm{int}}n_v\right)-N^{\mathrm{int}}_g\right].
$$
Counting half-edges in $g$ gives $\sum_{v\,\mathrm{int}}n_v=2N^{\mathrm{int}}_g+N^{\mathrm{ext}}_g$ implying (\ref{infrared}).
Reversing the arguments shows that (\ref{infrared}) implies (\ref{ultravioletcompleted}).
\end{proof}

\begin{center}
\fcolorbox{white}{white}{
  \begin{picture}(336,130) (-45,38)
    \SetWidth{0.8}
    \SetColor{Black}
    \GBox(80,117)(112,149){0.882}
    \Text(74,151)[lb]{\normalsize{\Black{$0$}}}
    \Text(74,111)[lb]{\normalsize{\Black{$1$}}}
    \Text(114,111)[lb]{\normalsize{\Black{$z$}}}
    \Text(114,151)[lb]{\normalsize{\Black{$\infty$}}}
    \Line[arrow,arrowpos=1,arrowlength=5,arrowwidth=2,arrowinset=0.2](75,128)(42,95)
    \Line[arrow,arrowpos=1,arrowlength=5,arrowwidth=2,arrowinset=0.2](96,112)(96,95)
    \Line[arrow,arrowpos=1,arrowlength=5,arrowwidth=2,arrowinset=0.2](117,128)(150,95)

    \GBox(16,53)(48,85){0.882}
    \Text(8,87)[lb]{\normalsize{\Black{$\infty$}}}
    \Text(8,43)[lb]{\normalsize{\Black{$1$}}}
    \Text(42,40)[lb]{\normalsize{\Black{$1/z$}}}
    \Text(50,87)[lb]{\normalsize{\Black{$0$}}}
    
    \GBox(80,53)(112,85){0.882}
    \Text(74,87)[lb]{\normalsize{\Black{$1$}}}
    \Text(74,43)[lb]{\normalsize{\Black{$0$}}}
    \Text(102,42)[lb]{\normalsize{\Black{$1\!-\!z$}}}
    \Text(114,87)[lb]{\normalsize{\Black{$\infty$}}}

    \GBox(144,53)(176,85){0.882}
    \Text(138,87)[lb]{\normalsize{\Black{$0$}}}
    \Text(138,40)[lb]{\normalsize{\Black{$1/z$}}}
    \Text(178,43)[lb]{\normalsize{\Black{$1$}}}
    \Text(178,87)[lb]{\normalsize{\Black{$\infty$}}}
  \end{picture}
}
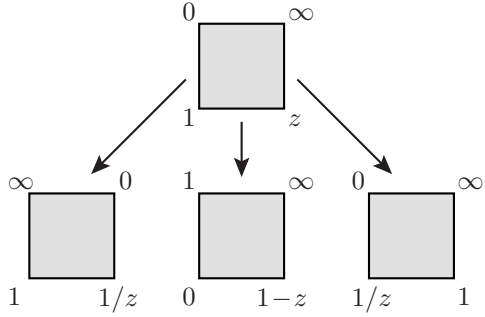
Figure 9: The $\sS_4$ symmetry of completed graphs is generated by three transformations.

\end{center}
\vskip2ex

Let $\sS_4$ be the symmetric group that permutes the four external vertices $0,1,z,\infty$. It is generated by the transpositions $(0,\infty)$, $(0,1)$, $(1,z)$.
The symmetric group $\sS_4$ has a factor group $\sS_3$ which in our context is realized as M{\"o}bius transformations
on $\PP^1\CC\backslash\{0,1,\infty\}$. The homomorphism $\phi:\sS_4\rightarrow\sS_3$ is determined by the images of the generating transpositions. Concretely we have
\begin{equation}\label{transpos}
\phi[(0,\infty)]=z\mapsto\frac{1}{z},\quad\phi[(0,1)]=z\mapsto 1-z,\quad\phi[(1,z)]=z\mapsto\frac{1}{z}.
\end{equation}
Let $\sigma\in\sS_4$ be a permutation of the external labels and $\pi=\phi(\sigma)\circ\sigma$ be $\sigma$ followed by a transformation of the label $z$
(see figure 9) indicating that the transformed graphical function has the argument $\phi(\sigma)$.

\begin{thm}\label{completionthm}
Completed graphical functions are invariant under $\pi=\phi(\sigma)\circ\sigma$,
\begin{equation}
f^{(\lambda)}_\Gamma(z)=f^{(\lambda)}_{\sigma(\Gamma)}(\phi(\sigma)(z)).
\end{equation}
\end{thm}
If we consider the graphical function as a function on $\RR^d$ we have to interpret $1-z$ as $e_1-z$ for the unit vector $e_1\in\RR^d$ and $1/z$ as $z/||z||^2$.
\begin{proof}
Because the $\sS_4$ is generated by the three transpositions in (\ref{transpos}) it is sufficient to prove the identities depicted in figure 9.
Let $x_i$, $i=5,6,\ldots,V$ denote the internal vertices of $\Gamma$ and for $v\in\{0,1,z,\infty\}$ let $N_{xv}$ be the sum of the weights of all the edges from internal vertices
to $v$. The weighted sum of internal edges is $N^{\mathrm{int}}$. Counting weights of half-edges that connect to internal vertices we obtain
\begin{equation}\label{halfedges}
\frac{d}{\lambda}(V-4)=2N^{\mathrm{int}}+N_{x0}+N_{x1}+N_{xz}+N_{x\infty}.
\end{equation}
For $v,w\in\{0,1,z,\infty\}$ let $\nu_{vw}$ be the weight of the edge connecting $v$ with $w$. In particular, $\nu_{vw}=0$ if there exists no edge from $v$ to $w$.
Indexed by the labels $0,1,z,\infty$ the integral in (\ref{fGamma}) is of the shape (working in $d$ dimensions)
\begin{equation}\label{ff0}
f_{0,1,z,\infty}=\frac{1}{||z||^{2\lambda\nu_{0z}}}\frac{1}{||z-e_1||^{2\lambda\nu_{1z}}}f^0_{0,1,z,\infty},
\end{equation}
where $f^0$ has no edges that connect external vertices.

A variable transformation $x_i\mapsto1-x_i$ leaves the integration measure invariant. Propagators between internal vertices remain intact, whereas
a propagator from an internal vertex $x_i$ to an external vertex $v$ is mapped to a propagator from $x_i$ to $1-v$.
Hence $f^0_{0,1,z,\infty}=f^0_{1,0,1-z,\infty}$ and by swapping the labels 0 and 1 in (\ref{ff0}) we have
$$f_{1,0,1-z,\infty}=\frac{1}{||e_1-z||^{2\lambda\nu_{1z}}}\frac{1}{||-z||^{2\lambda\nu_{0z}}}f^0_{1,0,1-z,\infty}=f_{0,1,z,\infty}.$$
This proves the middle identity in figure 9.

A variable transformation $x\mapsto x/||x||^2$ changes the integration measure by $\dd^dx_i\mapsto||x_i||^{-2d}\dd^dx_i$ (as can be seen
e.g.\ in angular coordinates). Propagators between internal vertices change according to
\begin{eqnarray*}
||x_i-x_j||^{-2\lambda\nu_{ij}}&=&(||x_i||^2-2||x_i||||x_j||\cos\angle(x_i,x_j)+||x_j||^2)^{-\lambda\nu_{ij}}\\
&\mapsto&\left(\frac{||x_i||\cdot||x_j||}{||x_i-x_j||}\right)^{2\lambda\nu_{ij}},
\end{eqnarray*}
where $\nu_{ij}$ is the weight of the edge between $x_i$ and $x_j$. Similarly internal edges that connect to $0,1,z$ transform as
\begin{eqnarray*}
||x_i||^{-2\lambda\nu_{i0}}&\mapsto&||x_i||^{2\lambda\nu_{i0}},\\
||x_i-e_1||^{-2\lambda\nu_{i1}}&\mapsto&\left(\frac{||x_i||}{||x_i-e_1||}\right)^{2\lambda\nu_{i1}},\\
||x_i-z||^{-2\lambda\nu_{iz}}&\mapsto&\left(\frac{||x_i||}{||z||\cdot||x_i-z/||z||^2||}\right)^{2\lambda\nu_{iz}}.
\end{eqnarray*}
If $x_i$ is connected to $\infty$ by a weight $\nu_{i\infty}$ then by completeness (\ref{nuint})
$$\lambda(\nu_{i0}+\nu_{i1}+\nu_{iz}+\sum_j\nu_{ij})=d-\lambda\nu_{i\infty}.$$
Altogether, after the transformation $x_i$ still connects to $x_j$ by the weight $\nu_{ij}$, to $z/||z||^2$ by weight $\nu_{iz}$
but the connection to 0 has now the weight $\nu_{i\infty}$. Moreover we pick up a total factor of $||z||^{-2\lambda N_{xz}}$ yielding
$f^0_{0,1,z,\infty}=||z||^{-2\lambda N_{xz}}f^0_{\infty,1,1/z,0}$. Including fully external edges and swapping labels 0 and $\infty$ we obtain from (\ref{ff0})
\begin{eqnarray*}
f_{\infty,1,1/z,0}&=&||z||^{2\lambda\nu_{z\infty}}\left(\frac{||z||}{||e_1-z||}\right)^{2\lambda\nu_{1z}}f^0_{\infty,1,1/z,0}\\
&=&||z||^{2\lambda(\nu_{z\infty}+\nu_{0z}+\nu_{1z}+N_{xz})}f_{0,1,z,\infty}.
\end{eqnarray*}
Due to (\ref{nuext}) this equals $f_{0,1,z,\infty}$ establishing the left hand side of figure 9.

A variable transformation $x\mapsto||z||x$ changes the integration measure by $\dd^dx_i\mapsto||z||^d\dd^dx_i$. Propagators between internal vertices
and propagators that connect internal vertices to $0,1,z$ change according to
\begin{eqnarray*}
||x_i-x_j||^{-2\lambda\nu_{ij}}&\mapsto&(||z||\cdot||x_i-x_j||)^{-2\lambda\nu_{ij}},\\
||x_i||^{-2\lambda\nu_{i0}}&\mapsto&(||z||\cdot||x_i||)^{-2\lambda\nu_{i0}},\\
||x_i-e_1||^{-2\lambda\nu_{i1}}&\mapsto&(||z||\cdot||x_i-e_1/||z||\,||)^{-2\lambda\nu_{i1}},\\
||x_i-z||^{-2\lambda\nu_{iz}}&\mapsto&(||z||\cdot||x_i-z/||z||\,||)^{-2\lambda\nu_{iz}}.
\end{eqnarray*}
Now, we rotate the coordinate system in such a way that $z$ points into the direction of $e_1$ and $e_1$ points into the direction of $z$ changing
$z/||z||$ to $e_1$ and vice versa. This swaps labels $1$ and $z$ together with a substitution $z\mapsto z/||z||^2$ plus an overall power of $||z||$,
$f^0_{0,1,z,\infty}=f^0_{0,1/z,1,\infty}||z||^{d(V-4)-2\lambda(N^{\mathrm{int}}+N_{x0}+N_{x1}+N_{xz})}$. By (\ref{halfedges}) and by swapping labels 1 and $z$
in (\ref{ff0}) we obtain
\begin{eqnarray*}
f_{0,1/z,1,\infty}&=&||z||^{2\lambda\nu_{01}}\left(\frac{||z||}{||e_1-z||}\right)^{2\lambda\nu_{1z}}f^0_{0,1/z,1,\infty}\\
&=&||z||^{\lambda(2\nu_{01}+2\nu_{0z}+2\nu_{1z}+N_{x0}+N_{x1}+N_{xz}-N_{x\infty})}f_{0,1,z,\infty}.
\end{eqnarray*}
Zero valence at external vertices (\ref{nuext}) gives rise to the four equations
\begin{eqnarray*}
\nu_{01}+\nu_{0z}+\nu_{0\infty}+N_{x0}=0,&&\nu_{01}+\nu_{1z}+\nu_{1\infty}+N_{x1}=0,\\
\nu_{0z}+\nu_{1z}+\nu_{z\infty}+N_{xz}=0,&&\nu_{0\infty}+\nu_{1\infty}+\nu_{z\infty}+N_{x\infty}=0,
\end{eqnarray*}
Adding the first three and subtracting the fourth equation proves the right hand side of figure 9.

This establishes the theorem in $d$ dimensions. Transition to $\CC$ maps $z/||z||^2$ to $1/\zz$. Because of (\ref{reflection}) we may replace $1/\zz$ by $1/z$.
\end{proof}
\begin{center}
\fcolorbox{white}{white}{
  \begin{picture}(336,130) (-45,38)
    \SetWidth{0.8}
    \SetColor{Black}
    \GBox(80,117)(112,149){0.882}
    \Text(74,151)[lb]{\normalsize{\Black{$0$}}}
    \Text(74,111)[lb]{\normalsize{\Black{$1$}}}
    \Text(114,111)[lb]{\normalsize{\Black{$z$}}}
    \Text(114,151)[lb]{\normalsize{\Black{$\infty$}}}
    \Line[arrow,arrowpos=1,arrowlength=5,arrowwidth=2,arrowinset=0.2](75,128)(42,95)
    \Line[arrow,arrowpos=1,arrowlength=5,arrowwidth=2,arrowinset=0.2](96,112)(96,95)
    \Line[arrow,arrowpos=1,arrowlength=5,arrowwidth=2,arrowinset=0.2](117,128)(150,95)

    \GBox(16,53)(48,85){0.882}
    \Text(8,87)[lb]{\normalsize{\Black{$\infty$}}}
    \Text(8,47)[lb]{\normalsize{\Black{$z$}}}
    \Text(50,47)[lb]{\normalsize{\Black{$1$}}}
    \Text(50,87)[lb]{\normalsize{\Black{$0$}}}
    
    \GBox(80,53)(112,85){0.882}
    \Text(74,87)[lb]{\normalsize{\Black{$1$}}}
    \Text(74,47)[lb]{\normalsize{\Black{$0$}}}
    \Text(114,47)[lb]{\normalsize{\Black{$\infty$}}}
    \Text(114,87)[lb]{\normalsize{\Black{$z$}}}

    \GBox(144,53)(176,85){0.882}
    \Text(138,87)[lb]{\normalsize{\Black{$z$}}}
    \Text(138,47)[lb]{\normalsize{\Black{$\infty$}}}
    \Text(178,47)[lb]{\normalsize{\Black{$0$}}}
    \Text(178,87)[lb]{\normalsize{\Black{$1$}}}
  \end{picture}
}
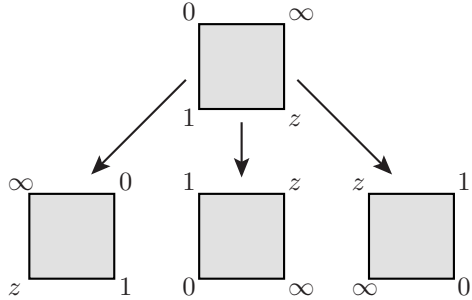
Figure 10: Double transpositions leave the label $z$ unchanged.

\end{center}
\vskip2ex

Theorem \ref{completionthm} is equivalent to (\ref{S4trafos}). In particular, the homomorphism $\phi$ has the kernel $\ZZ/2\ZZ\times\ZZ/2\ZZ$ of double transpositions which
is a normal subgroup in $\sS_4$. These transformations leave the argument $z$ unchanged (see figure 10). They give rise to the twist identity on periods \cite{SchnetzCensus}:
In a completed primitive graph $\Gamma$ with three labeled vertices $\{0,1,\infty\}$ we specify a fourth vertex $z$ and write the period as integral over $z$ as
in \S \ref{constper}. If deleting $0,1,z,\infty$ splits $\Gamma$ into two connected components then the integrand splits into two factors, each
of which is given by a graphical function in $z$. These graphical functions can be completed and thereafter one of the two may be transformed by a double transposition.
If the result can be interpreted again as the period of a primitive graph $\widetilde{\Gamma}$ then $\widetilde{\Gamma}$ has the same period as $\Gamma$. The smallest
example of such an identity is $P_{7,4}=P_{7,7}$ at seven loops with the notation from \cite{SchnetzCensus}.

\begin{ex}
The graphical function of the complete graph with four vertices was calculated in example \ref{onevertex}: after adding the external edges we obtain from (\ref{4iD})
in four dimensions
\begin{equation}
f_{K_4}^{(1)}(z)=\frac{4iD(z)}{z\zz(z-1)(\zz-1)(z-\zz)}.
\end{equation}
\begin{center}
\fcolorbox{white}{white}{
  \begin{picture}(319,91) (6,2)
    \SetWidth{0.8}
    \SetColor{Black}
    \Vertex(30,46){2.8}
    \Vertex(70,46){2.8}
    \Vertex(50,71){2.8}
    \Vertex(10,71){2.8}
    \Vertex(90,71){2.8}
    \Vertex(130,71){2.8}
    \Vertex(170,71){2.8}
    \Vertex(210,71){2.8}
    \Vertex(150,46){2.8}
    \Vertex(190,46){2.8}
    \Vertex(170,21){2.8}
    \Vertex(250,71){2.8}
    \Vertex(290,71){2.8}
    \Vertex(270,46){2.8}
    \Vertex(310,46){2.8}
    \Vertex(290,21){2.8}
    \Line(290,71)(290,21)
    \Line(170,71)(170,21)
    \Line(10,71)(30,46)
    \Line(50,71)(70,46)
    \Line(130,71)(150,46)
    \Line(170,71)(190,46)
    \Line(250,71)(270,46)
    \Line(290,71)(310,46)
    \Line(50,71)(30,46)
    \Line(90,71)(70,46)
    \Line(170,71)(150,46)
    \Line(210,71)(190,46)
    \Line(290,71)(270,46)
    \Line(30,46)(70,46)
    \Line(150,46)(190,46)
    \Line(270,46)(310,46)
    \Line(10,71)(70,46)
    \Line(130,71)(190,46)
    \Line(250,71)(310,46)
    \Arc(159.375,53.5)(34.193,149.216,288.104)
    \PhotonArc[double,sep=4,clock](180.625,53.5)(34.193,30.784,-108.104){2.5}{7.5}
    \Line(30,46)(90,71)
    \Line(150,46)(210,71)
    \Photon(130,71)(170,71){2.5}{5}
    \Photon(250,71)(290,71){2.5}{5}
    \PhotonArc[clock](150,33.5)(42.5,118.072,61.928){2.5}{5.5}
    \PhotonArc[clock](270,33.5)(42.5,118.072,61.928){2.5}{5.5}
    \PhotonArc(150,108.5)(42.5,-118.072,-61.928){2.5}{5.5}
    \PhotonArc(270,108.5)(42.5,-118.072,-61.928){2.5}{5.5}
    \Arc(279.375,53.5)(34.193,149.216,288.104)
    \Text(8,81)[lb]{\normalsize{\Black{$1$}}}
    \Text(48,81)[lb]{\normalsize{\Black{$0$}}}
    \Text(88,81)[lb]{\normalsize{\Black{$z$}}}
    \Text(128,81)[lb]{\normalsize{\Black{$1$}}}
    \Text(168,81)[lb]{\normalsize{\Black{$0$}}}
    \Text(208,81)[lb]{\normalsize{\Black{$z$}}}
    \Text(248,81)[lb]{\normalsize{\Black{$0$}}}
    \Text(288,81)[lb]{\normalsize{\Black{$1$}}}
    \Text(166,7)[lb]{\normalsize{\Black{$\infty$}}}
    \Text(288,7)[lb]{\normalsize{\Black{$z$}}}
  \end{picture}
}
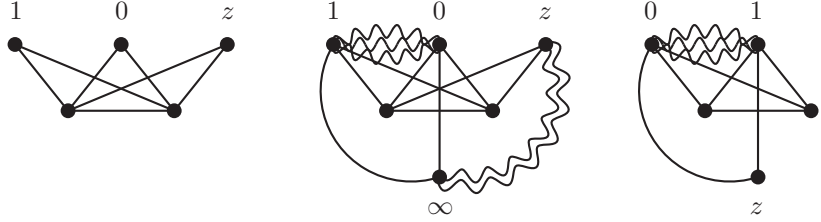
Figure 11: The graphical function of $K_5$ can be calculated by completion.
\vskip2ex

\end{center}
The graphical function of the complete graph with five vertices reduces after the removal of the external edges $0z$, $1z$, $01$ to three triangles that are
glued together at a common edge. Its completion is depicted in the middle of figure 11. After swapping $z$ with $\infty$ and $0$ with $1$ we obtain a graphical function that
upon removal of the edges $0z$ and $1z$ reduces to a constant. By adding four edges 01 this constant becomes the period of the uncompleted $K_4$ which is $6\zeta(3)$.
Altogether we have
\begin{equation}
f_{K_5}^{(1)}(z)=\frac{6\zeta(3)}{[z\zz(z-1)(\zz-1)]^2}.
\end{equation}
\end{ex}

If a completed graph $\Gamma$ has $n$ external edges we can consider $\Gamma$ as the Feynman graph of an $n$-point function in $\phi^4$ theory where certain external
points are identified. The full amplitude of this Feynman graph is encapsulated in the graphical function $f_\Gamma$ \cite{PropGF}.

\subsection{Appending an edge}
In this subsection $G$ is a graph with three marked vertices $0,1,z$ and $G_1$ results from $G$ by appending an edge to the vertex $z$ thus
creating a new vertex $z$ (see figure 6). We will assume that $G$ and $G_1$ are uncompleted graphs of well-defined graphical functions.
\begin{prop}\label{appendprop}
The following identity holds
\begin{equation}\label{diffeq}
\left(\frac{1}{(z-\zz)^\lambda}\,\partial_z\partial_{\zz}\,(z-\zz)^\lambda+\frac{\lambda(\lambda-1)}{(z-\zz)^2}\right)f^{(\lambda)}_{G_1}(z)
=-\frac{1}{\Gamma(\lambda)}f^{(\lambda)}_G(z),
\end{equation}
where $\Gamma(\lambda)=\int_0^\infty x^{\lambda-1}\exp(-x)\dd x$ is the gamma function.
\end{prop}
\begin{proof}
In this proof we use $d$-dimensional angular coordinates
\begin{eqnarray}\label{angularcoords}
x^1&=&r\cos(\phi^1),\\
x^2&=&r\sin(\phi^1)\cos(\phi^2),\nonumber\\
&\vdots&\nonumber\\
x^{d-1}&=&r\sin(\phi^1)\cdots\sin(\phi^{d-2})\cos(\phi^{d-1}),\nonumber\\
x^d&=&r\sin(\phi^1)\cdots\sin(\phi^{d-2})\sin(\phi^{d-1}),\nonumber
\end{eqnarray}
where $r$ ranges from 0 to $\infty$, $\phi^1$,\ldots,$\phi^{d-2}$ range from 0 to $\pi$, and $\phi^{d-1}$ ranges from 0 to $2\pi$.
The metric tensor $g_{\mu\nu}$ is diagonal with entries 1, $r^2$, $(r\sin\phi^1)^2$, $(r\sin\phi^1\sin\phi^2)^2$, \ldots,
$(r\sin\phi^1\cdots\sin\phi^{d-2})^2$. The volume measure is
$$\sqrt{g}=r^{d-1}(\sin\phi^1)^{d-2}(\sin\phi^2)^{d-3}\cdots \sin\phi^{d-2}$$
and the Laplacian is
\begin{eqnarray*}
\Delta^{(d)}&=&\frac{1}{\sqrt{g}}\sum_{\mu,\nu}\partial_\mu\sqrt{g}g^{\mu\nu}\partial_\nu\\
&=&\frac{1}{r^{d-1}}\partial_rr^{d-1}\partial_r+\frac{1}{r^2(\sin\phi^1)^{d-2}}\partial_{\phi^1}(\sin\phi^1)^{d-2}\partial_{\phi^1}\\
&&\quad+\;\frac{1}{(r\sin\phi^1)^2(\sin\phi^2)^{d-3}}\partial_{\phi^2}(\sin\phi^2)^{d-3}\partial_{\phi^2}\\
&&\quad+\;\ldots\\
&&\quad+\;\frac{1}{(r\sin\phi^1\cdots\sin\phi^{d-3})^2\sin\phi^{d-2}}\partial_{\phi^{d-2}}\sin\phi^{d-2}\partial_{\phi^{d-2}}\\
&&\quad+\;\frac{1}{(r\sin\phi^1\cdots\sin\phi^{d-2})^2}\partial_{\phi^{d-1}}^2.
\end{eqnarray*}
The propagator $||y-z||^{-2\lambda}$ is proportional to the Green's function of the $d$-di\-men\-sional Laplacian:
\begin{equation}\label{greens}
\Delta^{(d)}_z\int F(y)\frac{1}{||y-z||^{2\lambda}}\dd^dy=-\frac{4\pi^{\lambda+1}}{\Gamma(\lambda)}F(z)
\end{equation}
whenever the integral on the left hand side exists. To prove the above equation we use angular coordinates to Fourier transform the function $||x||^{-\alpha}$:
A standard calculation using the definition and the properties of the gamma function yields for $0<\alpha<d$ in the limit $\epsilon\searrow0$
$$
\int\frac{\mathrm{e}^{\mathrm{i}x\cdot p-\epsilon||x||}}{||x||^\alpha}\dd^dx=\pi^{d/2}\left(\frac{2}{p}\right)^{d-\alpha}\frac{\Gamma((d-\alpha)/2)}{\Gamma(\alpha/2)}.
$$
Specifying $\alpha=2\lambda$ yields (\ref{greens}) by the convolution property of the Fourier transform.

If we specify $F(y)$ to the graphical function $f^{(\lambda)}_G(y)$ in (\ref{greens}) we obtain
\begin{equation}
\Delta^{(d)}f^{(\lambda)}_{G_1}(z)=-\frac{4}{\Gamma(\lambda)}f^{(\lambda)}_G(z).
\end{equation}
If we choose the unit vector $e_1$ in $f^{(\lambda)}_{G_1}$ to point into the 1-direction then $f^{(\lambda)}_{G_1}(z)$ becomes in angular
coordinates a function of $r$ and $\phi^1$ only. The above formula for the $d$-dimensional Laplacian hence simplifies when applied on $f^{(\lambda)}_{G_1}$ to
$$
\Delta^{(d)}f^{(\lambda)}_{G_1}(r,\phi^1)=
\left(\partial_r^2+\frac{2\lambda+1}{r}\partial_r+\frac{1}{r^2}\partial_{\phi^1}^2+\frac{2\lambda\cos\phi^1}{r^2\sin\phi^1}\partial_{\phi^1}\right)f^{(\lambda)}_{G_1}(r,\phi^1).
$$
Comparison with
$$
\frac{1}{(r\sin\phi^1)^\lambda}\Delta^{(2)}(r\sin\phi^1)^\lambda=
\partial_r^2+\frac{2\lambda+1}{r}\partial_r+\frac{1}{r^2}\partial_{\phi^1}^2+\frac{2\lambda\cos\phi^1}{r^2\sin\phi^1}\partial_{\phi^1}+\frac{\lambda(\lambda-1)}{(r\sin\phi^1)^2}
$$
yields the proposition upon transition to $\CC$ where $\Delta^{(2)}=4\partial_z\partial_{\zz}$ and $z-\zz=2\mathrm{i}r\sin\phi^1$.
\end{proof}
\begin{ex}\label{f2ex0}
Let $G$ be $I$ plotted in figure 4 (example \ref{exI}). Then $f^{(\lambda)}_{G_1}=f^{(\lambda)}_2$ is the sequential function of the word 2.
Its graphical function was calculated in example \ref{onevertex}. For the case $\lambda=1$ differentiation of (\ref{f24dim}) yields
$$
-\frac{1}{z-\zz}\left(\frac{1}{\zz(z-1)}-\frac{1}{z(\zz-1)}\right)
$$
which gives (\ref{Ieq}). If $\lambda>1$ we write the differential operator in (\ref{diffeq}) as
$$
\frac{1}{(z-\zz)^\lambda}\,\partial_z\partial_{\zz}\,(z-\zz)^\lambda+\frac{\lambda(\lambda-1)}{(z-\zz)^2}=\partial_z\partial_{\zz}-\lambda\frac{\partial_z-\partial_{\zz}}{z-\zz}
$$
and check by direct computation (compare equations (\ref{Ieq}) and (\ref{f2lambda}))
$$
\left(\partial_z\partial_{\zz}-\lambda\frac{\partial_z-\partial_{\zz}}{z-\zz}\right)
\frac{(z\zz)^{1-\lambda}-t^{2\lambda-2}}{(1-tz)^\lambda(1-t\zz)^\lambda}=-\partial_t\frac{(\lambda-1)t}{[z\zz(tz-1)(t\zz-1)]^\lambda}.
$$
\end{ex}

\begin{ex}\label{zbarzexample}
The graphical function of $G_6$ in figure 12 leads in four dimensions to the differential equation (see (\ref{4iD}))
$$
-\frac{1}{z-\zz}\partial_z\partial_\zz(z-\zz)f^{(1)}_{G_6}=\left(\frac{4{\mathrm i}D(z)}{z-\zz}\right)^2.
$$
The above differential equation cannot be solved in terms of SVMPs. For the solution one needs primitives of iterated integrals with differential form $\dd z/(z-\zz)$.
This more general case is first discussed by F. Chavez and C. Duhr in \cite{zbarz}. A comprehensive treatment of this setup will be given in \cite{GSVH} with an
implementation in \cite{Hyperlogproc}.

\begin{center}
\fcolorbox{white}{white}{
  \begin{picture}(315,105) (-110,-10)
    \SetWidth{1.0}
    \SetColor{Black}
    \Vertex(53,79){2.8}
    \Vertex(43,36){2.8}
    \Vertex(63,36){2.8}
    \Vertex(90,10){2.8}
    \Vertex(16,10){2.8}
    \Vertex(135,10){2.8}
    \Line(53,79)(43,36)
    \Line(53,79)(63,36)
    \Line(43,36)(90,10)
    \Line(63,36)(90,10)
    \Line(16,10)(43,36)
    \Line(16,10)(63,36)
    \Line(135,10)(90,10)
    \Text(51,86)[lb]{\normalsize{\Black{$1$}}}
    \Text(13,-4)[lb]{\normalsize{\Black{$0$}}}
    \Text(133,-4)[lb]{\normalsize{\Black{$z$}}}
    \Text(-20,42)[lb]{\normalsize{\Black{$G_6=$}}}
  \end{picture}
}
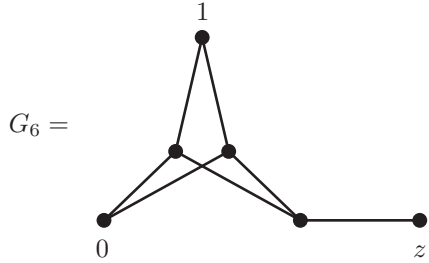
Figure 12: The graphical function of $G_6$ cannot be expressed in terms of SVMPs.
\vskip2ex

\end{center}

Up to six vertices $f^{(1)}_{G_6}$ is the only graphical function which is not expressible in terms of SVMPs. At seven vertices there exist graphical functions
which involve multiple elliptic polylogarithms.
\end{ex}

\begin{defn}
Let
\begin{equation}\label{Bdef}
\sB=\left\{\frac{f(z)}{z-\zz}:f\in\sA,f(z)=-f(\zz)\right\}.
\end{equation}
Likewise $\sB^0,\sB^{0\sv}$ are defined with $f\in\sP,f\in\sP^\sv$ in (\ref{Bdef}), respectively. Let $\sB^\bullet_n$ be the piece of $\sB^\bullet$ with (total) weight $n$.
\end{defn}
The differential operator in equation (\ref{diffeq}) generates denominators of the form $z-\zz$. If $f^{(\lambda)}_{G_1}$ is an element of the bi-differential algebra $\sA$
divided by $(z-\zz)^k$ then $f^{(\lambda)}_G$ is an element in $\sA$ divided by $(z-\zz)^{k+2}$.
It is unclear if solving (\ref{diffeq}) for $f^{(\lambda)}_{G_1}$ stays in the space of SVMPs even if one allows for denominators with arbitrary
powers of $(z-\zz)$ and square roots in odd dimensions. Due to completion none of these complications shows up in $f^{(\lambda)}_2$.

The structure of (\ref{diffeq}) is substantially different in four dimensions. If $\lambda=1$ the second term drops out and the differential operator maps graphical
function in $\sB$ into $\sB$. By integration in $\sA$ the differential operator can be inverted in $\sB$:

\begin{thm}\label{appendthm}
Assume the graphical function $f^{(1)}_{G_1}(z)$ is real analytic in $\CC\backslash\{0,1\}$
and for $z\in\RR^4$, $f^{(1)}_{G_1}(z)$, $f^{(1)}_{G_1}(e_1-z)\in\sC^{(1)}_{1,1}$.
Moreover, let
\begin{equation}\label{decomposition}
f^{(1)}_G(z)=\sum_{a,b\in\{0,1\}}\frac{g_{a,b}(z)}{(z-a)(\zz-b)(z-\zz)}\in\sB
\end{equation}
for $g_{a,b}\in\sP$ be a SVMP. Then
\begin{equation}\label{inteq}
f^{(1)}_{G_1}(z)=-\frac{1}{2(z-\zz)}\left(\int_0\dd z\!\int_0\dd\zz+\!\int_0\dd\zz\!\int_0\dd z\right)(z-\zz)f^{(1)}_G(z).
\end{equation}
The function $f^{(1)}_{G_1}$ is the unique solution of (\ref{diffeq}) in $\sB^0$.
\end{thm}
\begin{proof}
We denote the right hand side of (\ref{inteq}) by $F$. Clearly $F$ solves the differential equation (\ref{diffeq}). The kernel of $\partial_z\partial_{\zz}$ is a
sum of a function $h$ with $\partial_\zz h=0$ and a function $\overline{h}$ with $\partial_z\overline{h}=0$.
By (\ref{reflection}) the graphical functions $f^{(1)}_G$ and $f^{(1)}_{G_1}$ are symmetric under exchanging $z$ and $\zz$.
This symmetry is preserved in $F$ which restricts $f^{(1)}_{G_1}(z)$ to
$$
f^{(1)}_{G_1}(z)=F(z)+\frac{h(z)-h(\zz)}{z-\zz}.
$$
Because by assumption $f^{(1)}_{G_1}$ and by construction $F$ are real analytic in $\CC\backslash\{0,1\}$
we may take the derivative of $z-\zz$ times the above equation with respect to $z$ yielding
$$
h'(z)=\partial_z(z-\zz)(f^{(1)}_{G_1}(z)-F(z)).
$$
Now, $h'$ is a single-valued real analytic function in $\CC\backslash\{0,1\}$ with $\partial_\zz h'=0$. Hence $h'$ is holomorphic in $\CC\backslash\{0,1\}$.
We study its behavior at the singularities 0, 1, $\infty$.

Because $f^{(1)}_{G_1}(z)\in\sC^{(1)}_{1,1}$ we have $f^{(1)}_{G_1}(z)=o(|z|^{-1-\epsilon})$ for $z\to0$ and $\epsilon>0$
(meaning $\lim_{z\rightarrow0}f^{(1)}_{G_1}(z)|z|^{1+\epsilon}=0$). Therefore $\partial_z(z-\zz)f^{(1)}_{G_1}(z)=o(|z|^{-1-\epsilon})$.
From (\ref{decomposition}) we have $(z-\zz)F(z)\in\sP$. Hence $(z-\zz)F(z)=o(|z|^{-\epsilon})$ and $\partial_z(z-\zz)F(z)=o(|z|^{-1-\epsilon})$.
We conclude $h'(z)=o(|z|^{-1-\epsilon})$ and therefore $z^2h'(z)=o(|z|^{1-\epsilon})$. In particular $z^2h'(z)$ is continuous at $z=0$ and by Riemann's
removable singularity theorem it is holomorphic at 0. It vanishes at 0 and therefore $zh'(z)$ is holomorphic at 0.

The analogous argument for $z=1$ gives that $(z-1)h'(z)$ is holomorphic at 1. For $z\to\infty$ we have from $f^{(1)}_{G_1}(z)\in\sC^{(1)}_{1,1}$
that $\partial_z(z-\zz)f^{(1)}_{G_1}(z)=o(|z|^{-1+\epsilon})$. Because $(z-\zz)F(z)\in\sP$ we also have $\partial_z(z-\zz)F(z)=o(|z|^{-1+\epsilon})$.
Altogether this restricts $h'$ to be of the form $c/z+d/(z-1)$, for some constants $c,d\in\CC$. Therefore 
$$
f^{(1)}_{G_1}(z)-F(z)=\frac{c(\ln(z)-\ln(\zz))+d(\ln(z-1)-\ln(\zz-1))}{z-\zz}.
$$
The single-valuedness of the left hand side gives $c=d=0$ and hence $f^{(1)}_{G_1}=F$.

Moreover, $f^{(1)}_{G_1}\in\sB^0$ is clear from (\ref{decomposition}) and (\ref{inteq}). Assume $f\in\sB^0$ solves (\ref{diffeq}).
The kernel of $\partial_z\partial_{\zz}$ in $\sP$ is the constant function. Hence $f^{(1)}_{G_1}(z)=f(z)+c/(z-\zz)$.
By the symmetry of $f^{(1)}_{G_1}$ and $f$ under $z\leftrightarrow\zz$ we have $c=0$.
\end{proof}
\begin{ex}\label{f2example}
Straightforward integration of (\ref{Ieq}) gives
$$
f^{(1)}_2(z)=\frac{L_{\x_1\x_0}(\zz)+L_{\x_0}(\zz)L_{\x_1}(z)+L_{\x_0\x_1}(z)-L_{\x_1\x_0}(z)-L_{\x_0}(z)L_{\x_1}(\zz)-L_{\x_0\x_1}(\zz)}{z-\zz}
$$
which by the shuffle identity $L_{\x_0}L_{\x_1}=L_{\x_0\x_1}+L_{\x_1\x_0}$ is equivalent to (\ref{f24dim}).

The fact that the right hand side of (\ref{f24dim}) is in $\sB^0$ and that it reduces by (\ref{diffeq}) to $f^{(1)}_{\mathrm{I}}$, see example \ref{f2ex0},
gives an alternative proof of (\ref{f24dim}). 
\end{ex}

Our main application of theorem \ref{appendthm} is on sequential functions. The following corollary is used in the proof of the zig-zag conjecture \cite{ZZ}.
\begin{cor}\label{maincor}
Let $w$ be a word in 0,1,2 that begins with 2.
The sequential function in four dimensions $f^{(1)}_w$ is the unique solution in $\sB^0$ of the tower of differential equations
\begin{equation}
-\frac{1}{z-\zz}\partial_z\partial_{\zz}(z-\zz)f^{(1)}_{wa}(z)=\left\{\begin{array}{ll}
\frac{1}{(z-a)(\zz-a)}f^{(1)}_w(z)&\hbox{if }a=0,1,\\
\frac{1}{z\zz(z-1)(\zz-1)}f^{(1)}_w(z)&\hbox{if }a=2\end{array}\right.
\end{equation}
with initial condition (\ref{4iD}).
\end{cor}
\begin{proof}
The proof is induction over the length of the word $w$.
Corollary \ref{realanalyticcor} states that $f^{(1)}_w$ is real analytic in $\CC\backslash\{0,1\}$.
From proposition \ref{prop3} we have $f^{(1)}_w\in\sC^{(1)}_{0,2}\subset\sC^{(1)}_{1,1}$.
The function $f^{(1)}_w(e_1-z)$ is by theorem \ref{completionthm}
the sequential function of the word with 0s and 1s flipped. Hence $f^{(1)}_w(e_1-z)\in\sC^{(1)}_{1,1}$.
We add edges from 0 to $z$ or from 1 to $z$ and append an edge to $z$ to go from $f_w$ to $f_{wa}$.
Equation (\ref{decomposition}) is immediate for $a=0,1$ and follows by a partial fraction decomposition in $z$ and $\zz$ for $a=2$.
The statement of the corollary for $wa$ follows from the last statement in theorem \ref{appendthm}.
\end{proof}
The solution of the tower of differential equations can be given explicitly.
\begin{cor}\label{intcor}
Let $w$ be a word in 0,1,2 that begins with 2. The sequential function in four dimensions $f^{(1)}_w$ is recursively given by equation (\ref{4iD}) and
\begin{equation}\label{intcoreq}
f^{(1)}_{wa}(z)=-\frac{1}{2(z-\zz)}\left(\int_0\dd z\int_0\dd\zz+\int_0\dd\zz\int_0\dd z\right)
\left\{\!\begin{array}{ll}
\frac{(z-\zz)f^{(1)}_w(z)}{(z-a)(\zz-a)}&\hbox{if }a=0,1,\hspace{-10pt}\\
\raisebox{15pt}{}\frac{(z-\zz)f^{(1)}_w(z)}{z\zz(z-1)(\zz-1)}&\hbox{if }a=2.\hspace{-10pt}\end{array}\right.
\end{equation}
Moreover, $f^{(1)}_w\in\sB^{0\sv}_{2|w|}$.
\end{cor}
\begin{proof}
Equation (\ref{intcoreq}) is clear from theorem \ref{appendthm} and corollary \ref{maincor}. To prove that the coefficients of $f^{(1)}_w$ are in $\sH^\sv(\ZZ)$ we proceed by induction.
The case $w=2$ is (\ref{4iDP}). If $f^{(1)}_w\in\sB^\sv$ then in the case $a=0$ we have $\int_0\dd z\int_0\dd\zz=\int_0\dd\zz\int_0\dd z$
from remark \ref{comm0}. The factor of two in the denominator cancels and by theorem \ref{Stabilitythm} we obtain $f^{(1)}_{w0}\in\sB^{0\sv}$.

If $a=1$ we use (\ref{commute}) to trade $\int_0\dd\zz\int_0\dd z$ for $\int_0\dd z\int_0\dd\zz$ canceling the factor of two in the denominator.
Using the notation
$$Q_1(z)=\int_0\frac{(z-\zz)f^{(1)}_w(z)\dd\zz}{\zz-1}$$
we pick up the term $(Q_1(1)+\overline{Q_1(1)})P_1(z)$ (note that $f^{(1)}_w(z)=f^{(1)}_w(\zz)$). Because $f^{(1)}_w\in\sB^{0\sv}$ we obtain from theorem
\ref{Stabilitythm} that $Q_1\in\sP^\sv$ and $Q_1(1)\in\sH^\sv(\ZZ)$. In particular $Q_1(1)=\overline{Q_1(1)}$. Altogether
$$
f^{(1)}_{w1}(z)=\frac{1}{z-\zz}\left(-\int_0\frac{Q_1(z)\dd z}{z-1}+Q_1(1)P_1(z)\right)$$
and $f^{(1)}_{w1}\in\sB^{0\sv}$ follows from theorem \ref{Stabilitythm}.

If $a=2$ we analogously obtain
$$
f^{(1)}_{w2}(z)=\frac{1}{z-\zz}\left(-\int_0\frac{Q_2(z)\dd z}{z(z-1)}+Q_2(1)P_1(z)\right)\hbox{ with }Q_2(z)=\int_0\frac{(z-\zz)f^{(1)}_w(z)\dd\zz}{\zz(\zz-1)}.$$
By theorem \ref{Stabilitythm} $Q_2\in\sP^\sv$ and $f^{(1)}_{w2}\in\sB^{0\sv}$. The total weight $2|w|$ also follows by induction from theorem \ref{Stabilitythm}.
\end{proof}

\begin{cor}\label{Pw}
In four dimensions $P(G_w)\in\sH^\sv(\ZZ)_{2|w|-1}$ for every word $w$ which begins and ends in 2.
\end{cor}
\begin{proof}
For $w=v2$ we have $P(G_{v2})=f^{(1)}_{v1}(0)$ with $f^{(1)}_{v1}(z)=g(z)/(z-\zz)$ and $g\in\sP^\sv_{2|w|}$.
Substituting (e.g.) $z=\mathrm{i}\epsilon$ and taking the limit $\epsilon\to0$
using L'H\^opital gives $P(G_{v2})=\frac{1}{2}(\partial_zg(0)-\partial_\zz g(0))$. Because $g(z)=-g(\zz)$ we have $P(G_{v2})=\partial_zg(0)=Q_1(0)-Q_1(1)\in\sH^\sv(\ZZ)$
with $Q_1$ from the proof of corollary \ref{intcor} for $w=v$.
\end{proof}

\begin{ex}\label{0ex}
The graphical function of the word $20^{\{n-1\}}$ can be calculated explicitly.
The initial case $f^{(1)}_2$ is given by (\ref{f24dim}) and (\ref{wt2}),
\begin{equation}\label{4iDP}
f^{(1)}_2(z)=\frac{P_{01}(z)-P_{10}(z)}{z-\zz}.
\end{equation}
Integration of $P_{0^{\{a\}}10^{\{b\}}}$ with $\int_0\dd z/z$ or $\int_0\dd\zz/\zz$ is given by appending or prepending a 0 (see (\ref{intonw}) and lemma \ref{Lielem}).
Therefore corollary \ref{intcor} gives (see also \cite{Drummond})
\begin{equation}\label{0exeq}
f^{(1)}_{20^{\{n-1\}}}(z)=(-1)^{n-1}\frac{P_{0^{\{n\}}10^{\{n-1\}}}(z)-P_{0^{\{n-1\}}10^{\{n\}}}(z)}{z-\zz}.
\end{equation}
Using (\ref{formula0s}) the above expression can be converted into polylogarithms,
\begin{equation}\label{0exeqLi}
f^{(1)}_{20^{\{n-1\}}}(z)=\sum_{k=0}^n(-1)^{n-k}\binom{k+n}{n}\frac{(\ln z\zz)^{n-k}}{(n-k)!}\frac{\Li_{n+k}(z)-\Li_{n+k}(\zz)}{z-\zz}.
\end{equation}
This result was first derived by N. Ussyukina and A. Davydychev in 1993 \cite{Ladder}.

If we specify the variable $z$ to 1 we obtain the period of the wheel with $n+1$ spokes. Substituting $z=1+\mathrm{i}\epsilon$ in (\ref{0exeq}) yields in the limit
$\epsilon\searrow0$
\begin{equation}
P(W\!S_{n+1})=\left.\frac{(-1)^{n-1}}{2}(\partial_z-\partial_{\zz})(P_{0^{\{n\}}10^{\{n-1\}}}(z)-P_{0^{\{n-1\}}10^{\{n\}}}(z))\right|_{z=1}.
\end{equation}
The polylogarithms are differentiated by deconcatenating the index. Because the second polylogarithm is the complex conjugate of the first we obtain
\begin{equation}
P(W\!S_{n+1})=(-1)^{n-1}(P_{0^{\{n\}}10^{\{n-2\}}}(1)-P_{0^{\{n-1\}}10^{\{n-1\}}}(1)).
\end{equation}
In particular the period of the wheel is in $\sH^\sv(\ZZ)$. With (\ref{formula0s}) we obtain
\begin{equation}\label{PWSn}
P(W\!S_{n+1})=\binom{2n}{n}\zeta(2n-1).
\end{equation}
This result was first derived by D. Broadhurst in 1985 \cite{B2}.
\end{ex}

\subsection{A convolution product}
\begin{center}
\fcolorbox{white}{white}{
  \begin{picture}(345,76) (-15,8)
    \SetWidth{0.5}
    \SetColor{Black}
    \GOval(60,42)(20,55)(0){0.882}
    \GOval(200,42)(20,55)(0){0.882}
    \GOval(255,42)(20,55)(0){0.882}
    \SetWidth{0.8}
    \COval(60,42)(20,20)(0){Black}{White}
    \COval(200,42)(20,20)(0){Black}{White}
    \COval(255,42)(20,20)(0){Black}{White}
    \SetColor{White}
    \CBox(200,12)(255,107){White}{White}
    \SetColor{Black}
    \Vertex(60,62){2.8}
    \Vertex(60,22){2.8}
    \Vertex(5,42){2.8}
    \Vertex(115,42){2.8}
    \Vertex(145,42){2.8}
    \Vertex(200,62){2.8}
    \Vertex(200,22){2.8}
    \Vertex(255,62){2.8}
    \Vertex(255,22){2.8}
    \Vertex(310,42){2.8}
    \Text(27,12)[lb]{\normalsize{\Black{$G$}}}
    \Text(162,12)[lb]{\normalsize{\Black{$G_1$}}}
    \Text(288,12)[lb]{\normalsize{\Black{$G_2$}}}
    \Text(-2,48)[lb]{\normalsize{\Black{$1$}}}
    \Text(139,48)[lb]{\normalsize{\Black{$1$}}}
    \Text(252,29)[lb]{\normalsize{\Black{$1$}}}
    \Text(57,68)[lb]{\normalsize{\Black{$0$}}}
    \Text(198,68)[lb]{\normalsize{\Black{$0$}}}
    \Text(253,68)[lb]{\normalsize{\Black{$0$}}}
    \Text(118,48)[lb]{\normalsize{\Black{$z$}}}
    \Text(198,29)[lb]{\normalsize{\Black{$z$}}}
    \Text(314,48)[lb]{\normalsize{\Black{$z$}}}
  \end{picture}
}
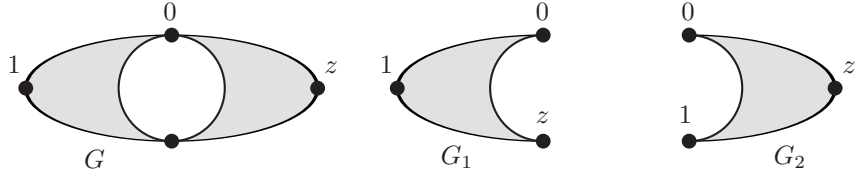
Figure 13: The convolution product: the graphical function of $G$ can be calculated as two dimensional convolution of the graphical functions
$G_1$ and $G_2$.
\end{center}
\vskip2ex

In four dimensions we can generalize the formula for appending an edge by a convolution product. Because we will not use this in the following we state it as a remark which
can be proved using Gegenbauer polynomials.
\begin{remark}
With the notation of figure 13 we have
\begin{eqnarray}\label{convol}
(z\partial_z-\zz\partial_{\zz})(z-\zz)f^{(1)}_G(z)&=&\frac{1}{\pi}\int_{\CC}\dd^2y\,(y-\overline{y})f^{(1)}_{G_1}(y)\nonumber\\
&&\cdot\,\left(\frac{z}{y}-\frac{\overline{z}}{\overline{y}}\right)f^{(1)}_{G_2}\!\left(\frac{z}{y}\right)(y\overline{y})^{2V_{G_2}^{\mathrm{int}}-N_{G_2}+1},
\end{eqnarray}
where $N_{G_2}$ and $V_{G_2}^{\mathrm{int}}$ are the number of edges and internal vertices of $G_2$, respectively.
The residue theorem \ref{residuethm} generalizes to hyperlogarithms defined in \cite{BrSVMPII}. This implies that there exists a convolution product
$$\begin{array}{ccl}m:\sA\otimes\sA&\rightarrow&\sA\\
m(f\otimes g)&\mapsto&\displaystyle\frac{1}{\pi}\int_\CC f(y)\,g\!\left(\frac{z}{y}\right)\frac{\dd^2y}{y\overline{y}}\end{array}
$$
which guarantees that the right hand side of (\ref{convol}) is in $\sA$ if $f^{(1)}_{G_1}$ and $f^{(1)}_{G_2}$ are in $\sB$.
The above equation unambiguously determines $f^{(1)}_G$ because the kernel of the differential operator on the left hand side are functions of $|z|$ and
$f^{(1)}_G$ has the symmetry (\ref{reflection}).
\end{remark}
\begin{ex}
The graphical function of figure 14 can be solved using the convolution product. The result is an expression in terms of SVMPs of weight six divided by $(z-\zz)(1-z\zz)$
\cite{Hyperlogproc}.

\begin{center}
\fcolorbox{white}{white}{
  \begin{picture}(315,90) (-110,-1)
    \SetWidth{1.0}
    \SetColor{Black}
    \Vertex(55,62){2.8}
    \Vertex(10,36){2.8}
    \Vertex(55,36){2.8}
    \Vertex(100,36){2.8}
    \Vertex(10,10){2.8}
    \Vertex(100,10){2.8}
    \Line(55,62)(10,36)
    \Line(55,62)(100,36)
    \Line(10,36)(100,36)
    \Line(10,36)(10,10)
    \Line(10,10)(55,36)
    \Line(100,36)(100,10)
    \Line(55,36)(100,10)
    \Text(1,7)[lb]{\normalsize{\Black{$1$}}}
    \Text(53,68)[lb]{\normalsize{\Black{$0$}}}
    \Text(105,7)[lb]{\normalsize{\Black{$z$}}}
  \end{picture}
}
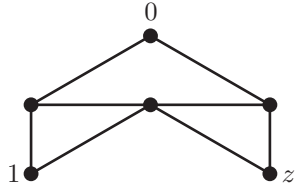
Figure 14: The above six vertex graph has a graphical function which is expressible in terms of SVMPs.
\end{center}
\end{ex}

\subsection{Constructing periods}\label{constper}
\begin{center}
\fcolorbox{white}{white}{
  \begin{picture}(345,122) (28,-5)
    \SetWidth{0.8}
    \SetColor{Black}
    \Vertex(32,64){2.8}
    \Vertex(80,112){2.8}
    \Vertex(64,80){2.8}
    \Vertex(64,16){2.8}
    \Vertex(80,48){2.8}
    \Vertex(112,64){2.8}
    \Line(32,64)(80,112)
    \Line(32,64)(64,80)
    \Line(32,64)(80,48)
    \Line(32,64)(64,16)
    \Line(112,64)(64,16)
    \Line(112,64)(80,48)
    \Line(112,64)(64,80)
    \Line(112,64)(80,112)
    \Line(80,112)(64,80)
    \Line(64,80)(64,16)
    \Line(64,16)(80,48)
    \Line(80,48)(80,112)
    \Line(208,64)(240,16)
    \Line[arrow,arrowpos=1,arrowlength=5,arrowwidth=2,arrowinset=0.2](144,64)(176,64)
    \Vertex(208,64){2.8}
    \Line(208,64)(256,48)
    \Line(208,64)(240,80)
    \Line(208,64)(256,112)
    \Vertex(240,16){2.8}
    \Vertex(256,48){2.8}
    \Vertex(240,80){2.8}
    \Vertex(256,112){2.8}
    \Text(249,78)[lb]{\normalsize{\Black{$0$}}}
    \Text(249,14)[lb]{\normalsize{\Black{$z$}}}
    \Text(265,46)[lb]{\normalsize{\Black{$\infty$}}}
    \Text(265,110)[lb]{\normalsize{\Black{$1$}}}
    \Photon(240,16)(256,48){2.5}{3}
    \Photon(240,80)(256,112){2.5}{3}
    \Line(374,64)(326,16)
    \Line(374,64)(342,48)
    \Line(374,64)(326,80)
    \Line(374,64)(342,112)
    \Vertex(326,16){2.8}
    \Vertex(342,48){2.8}
    \Vertex(326,80){2.8}
    \Vertex(342,112){2.8}
    \Photon(326,16)(342,48){2.5}{3}
    \Photon(326,80)(342,112){2.5}{3}
    \Text(315,14)[lb]{\normalsize{\Black{$z$}}}
    \Text(326,46)[lb]{\normalsize{\Black{$\infty$}}}
    \Text(315,78)[lb]{\normalsize{\Black{$0$}}}
    \Text(331,110)[lb]{\normalsize{\Black{$1$}}}
    \Text(62,-2)[lb]{\normalsize{\Black{$\Gamma$}}}
    \Text(224,-2)[lb]{\normalsize{\Black{$\Gamma_1$}}}
    \Text(336,-2)[lb]{\normalsize{\Black{$\Gamma_2$}}}
  \end{picture}
}
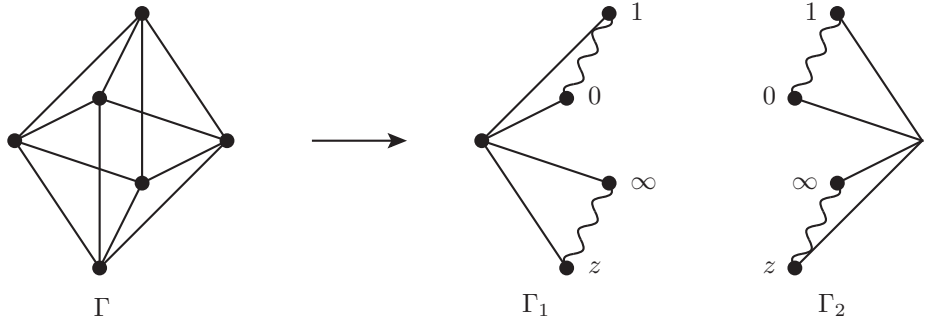
Figure 15: Constructing periods. The completed primitive graph $\Gamma$ is split into two completed graphs $\Gamma_1$ and $\Gamma_2$.
\end{center}
\vskip2ex

As explained in the introduction we can construct graphical functions from the empty function $f_\emptyset=1$ by adding external edges, appending edges to external vertices and
by permuting external vertices (see figure 7). This construction is programmed as $\tt graphical\_function(edgeset)$ in \cite{Polylogproc}. The algorithm is quick
and works for graphs up to 15 internal vertices.

Sequential functions at $z=0$ or $z=1$ give sequential periods. In the cases of the zig-zags and the type $A$ and $B$ families (see \S \ref{phi4})
one can relate these periods by the Fourier identity to $\phi^4$ theory. Their calculation is quick and up to loop order twelve the MZVs in the result
are of weights less than or equal to 21 which allows us to reduce them to a standard basis by the datamine \cite{Datamine} or by {\tt zeta\_procedures} \cite{Zetaproc}.

There exists another method to calculate periods which is more general and directly applicable to $\phi^4$ graphs.
This method works for periods which are `constructible' by the following procedure (see figure 15):
\begin{enumerate}
\item Start from a completed $\phi^4$ graph $\Gamma$ and label any four vertices by $0,1,z,\infty$.
\item Reduce $\Gamma\backslash\{0,1,z,\infty\}$ into its connected components $G_i$, $i=1,\ldots,N$.
\item Every component $G_i$ is turned into a completed graph $\Gamma_i$ by completing the graph one obtains by attaching $0,1,z,\infty$ to $G_i$ in the way they were attached
in $\Gamma$.
\end{enumerate}
Then
\begin{equation}\label{constreq}
P(\Gamma)=\int\frac{\dd^dz}{\pi^{d/2}}\prod_{i=1}^Nf^{(\lambda)}_{\Gamma_i}(z)\frac{1}{||z||^{2\lambda\nu_{0z}}||z-1||^{2\lambda\nu_{1z}}},
\end{equation}
where $\nu_{vz}=1$ is the weight of the edge from $v=0,1$ to $z$ in $\Gamma$.

The integral is converted into an integral over the complex plane by the following lemma (see proposition 4 in \cite{S3}):
\begin{lem}
Let $f:\RR^d\rightarrow\RR$ be a function that depends after the introduction of angular coordinates (\ref{angularcoords}) on $r$ and $\phi^1$ only.
Let $f_\CC(z)$ with $z=r\exp(\mathrm{i}\phi^1)$ be its complex counterpart. Then
\begin{equation}\label{dint}
\frac{1}{\pi^{d/2}}\int_{\RR^d}f(z)\dd^dz=\frac{1}{(2\mathrm{i})^{d-2}\sqrt{\pi}\Gamma((d-1)/2)}\int_{\CC}f_\CC(z)(z-\zz)^{d-2}\dd^2z.
\end{equation}
\end{lem}
\begin{proof}
Use angular coordinates to write the integral on the left hand side of (\ref{dint}) as
$$
\frac{2}{\sqrt{\pi}\Gamma((d-1)/2)}\int_0^\pi\int_0^\infty(r\sin\phi^1)^{d-2}f_\CC(r\mathrm{e}^{\mathrm{i}\phi^1})r\dd r\dd\phi^1.
$$
With $2\int_0^\pi\dd\phi^1=\int_0^{2\pi}\dd\phi^1$ the result follows.
\end{proof}
If the graphical functions $f^{(1)}_{\Gamma_i}$ are constructible in four dimensions then $f^{(1)}_{\Gamma_i}\in\sB$.
If $N\leq2$ then the integrand in (\ref{constreq}) is in $\sA$ and the two-dimensional integral can be evaluated by the residue theorem \ref{residuethm}.

\begin{ex}
The example pictured in figure 15 gives a construction of the $Z_4=W\!S_4$ period. With (\ref{4iD}) and the residue theorem \ref{residuethm} we have (compare (\ref{PWSn}))
\begin{equation}
P(W\!S_4)=-\frac{1}{2\pi}\int_{\CC}\dd^2z\frac{-16D(z)^2}{z\zz(z-1)(\zz-1)}=20\zeta(5).
\end{equation}
\end{ex}

\begin{remark}\label{Rsvremark}
It can be proved by analyzing equation (\ref{decomposition}) in all possible situations that constructible graphical functions are in $\sB^{0\sv}$.
Because for the proof one also needs (G3) in lemma \ref{prop123} which is beyond the scope of this article we refer the reader to \cite{PropGF}.

As a consequence one can also prove that primitive constructible $n$ loop periods are in $\sH^\sv(\ZZ)_{2n-3}$.
In \cite{BrFeyn} it was shown that uncompleted Feynman graphs with `vertex-width' $\leq3$ give MZV periods.
The completion of a Feynman graph with vertex-width $\leq3$ has vertex-width $\leq4$. Because graphs with vertex-width $\leq4$ are constructible
we obtain the stronger result that completed Feynman graphs with vertex width $\leq4$ have periods of maximum weight in $\sH^\sv(\ZZ)$.
It is easy to show that those graphs are always `$K_5$-descendants' defined in \cite{SchnetzCensus}.
\end{remark}

Note that beginning at six loops there exist periods in $\phi^4$ theory which are MZVs but not in $\sH^\sv(\ZZ)$.
Integration of $P_{6,3}$ or $P_{6,4}$ in \cite{SchnetzCensus} using graphical functions inevitably leads to the problem of finding
single-valued primitives of iterated integrals with differential forms $\dd z/(z-\zz)$ or $\dd\zz/(z-\zz)$ (see example \ref{zbarzexample}).
The periods $P_{6,3}$ and $P_{6,4}$ can be calculated with \cite{Hyperlogproc}.

A Maple algorithm that uses theorems \ref{residuethm} and \ref{appendthm} to calculate constructible periods is $\tt period(edgeset)$ in \cite{Polylogproc}.
The algorithm works (by memory and time limitations) for graphs with up to eleven loops. In particular, the zig-zags and the type $A$ and $B$ periods are constructible.
Up to five loops all irreducible primitive periods are zig-zags.
At six loops two out of four periods are constructible, at seven loops three out of nine, at eight loops six out of at most 31 periods. All these periods are
zig-zag or type $A$ or $B$. At nine loops there exist at most 134 different primitive periods ten of which are constructible.
In the labeling of \cite{CENSUS1} they are $P_{9,1}=Z_9$, $P_{9,2}=A_{5,0}=B_{5,0}$, $P_{9,3}=A_{4,1}$, $P_{9,4}=B_{4,1}$, $P_{9,6}=P_{9,19}=A_{3,2}$,
$P_{9,7}$, $P_{9,8}=P_{9,9}=B_{3,2}$, $P_{9,17}=P_{9,28}$, $P_{9,18}=P_{9,22}=P_{9,26}$, $P_{9,20}=P_{9,21}=P_{9,24}$.
At ten loops 54 out of 1182 irreducible primitive graphs are constructible, at eleven loops 154 out of 8687.

The $\tt period$ command in a recent implementation of `generalized single-valued hyperlogarithms' \cite{GSVH}, \cite{Hyperlogproc} is able to calculate
constructible and many more periods up to eleven loops. The main result of these calculations is the `coaction conjecture' for $\phi^4$ periods \cite{coact}.

\section{Sequential functions and $\phi^4$ periods}
In this section we study sequential functions in four dimensions. We drop the superscript $\bullet^{(1)}$ throughout this section.

\subsection{Reduction modulo products}
In corollary \ref{intcor} we showed that we can construct sequential functions starting from the Bloch-Wigner dilogarithm
by successively integrating in $\sA$. In general these integrals generate multiple polylogarithms multiplied by MZVs.
To obtain closed expressions for sequential functions we calculate modulo the ideal $I_n$ generated by MZVs of weights between two and $n$ (definition \ref{Indef}).

Vertices in sequential graphs that connect to 0 and 1 are encoded by the letter 2. They lead to differential forms $\dd z/z(z-1)=\dd z/(z-1)-\dd z/z$ and
$\dd\zz/\zz(\zz-1)=\dd\zz/(\zz-1)-\dd\zz/\zz$. Accordingly the definition \ref{P0def} of $P^0_w$ is extended to words with the letter 2.
\begin{defn}
For a word $w$ in 0,1,2 let $P^0_w$ be given by (\ref{P0defeq}) if $w$ has no 2s and inductively by
\begin{equation}
P^0_{u2v}(z)=P^0_{u1v}(z)-P^0_{u0v}(z)
\end{equation}
for any words $u,v$ in 0,1,2. Similarly we extend the definition of $\zeta_w$ in (\ref{zetawdef}) by
\begin{equation}
\zeta_{u2v}=\zeta_{u1v}-\zeta_{u0v}
\end{equation}
and the definition of $c_w$ in (\ref{modintsconst}) by
\begin{equation}
c_{u2v}=c_{u1v}-c_{u0v}.
\end{equation}
\end{defn}
With this definition we have in addition to (\ref{diffeq0})
\begin{equation}\label{diffeq2}
\partial_z P_{w2}^0(z)=\frac{P_w^0(z)}{z(z-1)},\quad\partial_{\zz} P_{2w}^0(z)=\frac{P_w^0(z)}{\zz(\zz-1)}.
\end{equation}
Equations (\ref{zetaop}) and (\ref{P0w1}) remain valid for words that include the letter 2.
\begin{prop}
Let $w$ be a word in 0,1,2. Then
\begin{equation}\label{modI2w}
f_{2w}(z)\equiv(-1)^{|w|}\frac{P^0_{\widetilde{w}01w}(z)-P^0_{\widetilde{w}10w}(z)-c_{\widetilde{w}01w}(L_1(z)-L_1(\zz))}{z-\zz}\mod I_{2|w|}.
\end{equation}
\end{prop}
\begin{proof}
The proof is by induction over $|w|$. The statement reduces to (\ref{4iDP}) for the empty word.
Let $w=ua$ for $a\in\{0,1,2\}$. Because $c_w$ is of weight $|w|-1$ we have by induction from (\ref{PP0})
$$
f_{2u}(z)\equiv(-1)^{|u|}\frac{P_{\widetilde{u}01u}(z)-P_{\widetilde{u}10u}(z)}{z-\zz}\mod I_{2|w|}.
$$
We use corollary \ref{intcor} and obtain by lemma \ref{modintlemma} and (\ref{PP0}) for $a\in\{0,1\}$
\begin{eqnarray*}
f_{2ua}(z)&\equiv&-\frac{1}{2(z-\zz)}\left[\int_0\frac{\dd z}{z-a}(-1)^{|u|}(P_{a\widetilde{u}01u}(z)-P_{a\widetilde{u}10u}(z))\right.\\
&&\quad+\;\left.\int_0\frac{\dd\zz}{\zz-a}(-1)^{|u|}(P_{\widetilde{u}01ua}(z)-P_{\widetilde{u}10ua}(z))\right]\mod I_{2|w|}
\end{eqnarray*}
Again with lemma \ref{modintlemma} we have for $f_{2ua}(z)$ modulo $I_{2|w|}$
$$
\frac{(-1)^{|u|+1}}{z-\zz}\left[P^0_{a\widetilde{u}01ua}(z)-P^0_{a\widetilde{u}10ua}(z)-\frac{1}{2}(c_{a\widetilde{u}01ua}-c_{a\widetilde{u}10ua})(L_1(z)-L_1(\zz))\right]
$$
and by (\ref{csym}) equation (\ref{modI2w}) follows.

If $a=2$ then a partial fraction decomposition in $z$ and $\zz$ leads to (\ref{modI2w}) by an analogous calculation.
\end{proof}

\begin{thm}\label{PGthm}
Let $w$ be a word in 0,1,2. Then
\begin{equation}\label{PG}
P(G_{2w2})\equiv2(-1)^{|w|}(\zeta_{\widetilde{w}01w0}-\zeta_{\widetilde{w}10w0}) \mod \sH_{>0}^2.
\end{equation}
\end{thm}
\begin{proof}
To derive the period from the sequential function we use the word $2w0$, set $z=1+\mathrm{i}\epsilon$ and use L'H\^opital on $\epsilon\to0$.
Because $\partial_\epsilon[L_1(1+\mathrm{i}\epsilon)-L_1(1-\mathrm{i}\epsilon)]=0$ we obtain by (\ref{diffeq0}) modulo $I_{2|w|+2}$
$$
P(G_{2w2})\equiv\frac{(-1)^{|w|+1}}{2}(P^0_{0\widetilde{w}01w}(1)-P^0_{\widetilde{w}01w0}(1)-P^0_{0\widetilde{w}10w}(1)+P^0_{\widetilde{w}10w0}(1)).
$$
With the extension of (\ref{zetaop}) and (\ref{P0w1}) to words with 2s we obtain (\ref{PG}) modulo $I_{2|w|+2}$. Because $P(G_{2w2})$ is homogeneous
of weight $2|w|+3$ (corollary \ref{Pw}) calculating modulo $I_{2|w|+2}$ is equivalent to calculating mod $\sH_{>0}^2$.
\end{proof}
Note that from words $w$ that end in 0 or 2 singular MZVs arise in (\ref{PG}). In a full calculation these singularities are canceled by terms in $I_{2|w|+2}$.
The singular words have to be treated as regularized values by (un-)shuffling to the left with $\zeta_0=0$ \cite{Deligne}. As an example the regularized limit
of $\ln(z)$ at $z=0$ is zero which leads to $0=\zeta_0$. Shuffling with $\zeta_1$ gives $0=\zeta_{01}+\zeta_{10}$ which provides
the regularized limit of the singular word 01: $\zeta_{01}=-\zeta_{10}=\zeta(2)$.

\begin{ex}
For the empty word $w=\emptyset$ we have
$$
P(G_{22})\equiv2(\zeta_{010}-\zeta_{100})\mod\sH_{>0}^2.
$$
Because $\sH_{>0}^2$ is trivial at weight three the above equivalence is an identity. From $0=\zeta_0\zeta_{10}=\zeta_{010}+2\zeta_{100}$ and $\zeta_{100}=-\zeta(3)$ we obtain
$$
P(G_{22})=6\zeta(3).
$$
Upon adding an edge 01 from 0 to 1, $G_{22}$ becomes the complete graph with four vertices which is the wheel with three spokes and the
uncompleted zig-zag graph with three loops (compare (\ref{PWSn}) for $n=3$).

For $w=1$ we get modulo $\sH_{>0}^2$
$$
P(G_{212})\equiv-2(\zeta_{10110}-\zeta_{11010})=2(\zeta(2,1,2)-\zeta(1,2,2))=20\zeta(5)-10\zeta(2)\zeta(3).
$$
Because $P(G_{212})\in\sH^\sv(\ZZ)_5$ by corollary \ref{Pw} we obtain from example \ref{Rsvex}
$$
P(G_{212})=20\zeta(5).
$$
By adding an edge 01 the graph $G_{212}$ becomes the wheel with four spokes which is the uncompleted zig-zag graph with four loops ($n=4$ in (\ref{PWSn})).
\end{ex}

\subsection{$\phi^4$ periods}\label{phi4}
Three families of $\phi^4$ periods are related to sequential periods.
\begin{lem}\label{phi4lemma}
The sequential graph $G_w$ can be made the planar dual of an uncompleted primitive $\phi^4$ graph by (possibly) adding an edge from 0 to 1 in exactly one of the
three cases
\begin{enumerate}
\item $w=2\underbrace{010\ldots}_{n-3}2$ or $w=2\underbrace{101\ldots}_{n-3}2$,
\item $w=2\underbrace{\ldots101}_{m}2\underbrace{010\ldots}_{n}2$ or $w=2\underbrace{\ldots010}_{m}2\underbrace{101\ldots}_{n}2$,
\item $w=2\underbrace{\ldots101}_{m}2\underbrace{101\ldots}_{n}2$ or $w=2\underbrace{\ldots010}_{m}2\underbrace{010\ldots}_{n}2$,
\end{enumerate}
where dotted sequences are alternating in 0 and 1. In the first case the $\phi^4$ graph is the zig-zag graph $Z_n$. In the second and the third
case the $\phi^4$ graphs are $A_{m,n}$ and $B_{m,n}$, respectively.
\end{lem}
\begin{proof}
A sequence of $n$ 0s or 1s in $w$ implies that the graph $G_w$ has a face with $n+3$ edges. The dual of an uncompleted $\phi^4$ graph with $V\geq4$ vertices
has four triangles and $V-4$ squares. Therefore $n\leq1$.

An internal 2 in $w$ gives rise to two triangles. An external 2 gives rise to one triangle.
Because by corollary \ref{Pwelldef} the word $w$ begins and ends in 2 we know that $w$ has at most one internal 2.
If we add an edge 01 in the case of no internal 2 the above three cases give rise to three families of uncompleted $\phi^4$ graphs.
Because in all three cases the sequential period exists the corresponding $\phi^4$ graphs are primitive.
\end{proof}
Figure 8 shows type $A$ and $B$ graphs after completion. We have proved in corollary \ref{Pw} that the $A$, $B$, and zig-zag periods are in $\sH^\sv(\ZZ)$.
For the zig-zags we obtain the following proposition:
\begin{prop}\label{modprop}
Modulo products the zig-zag periods for $n\geq4$ loops are given by
\begin{equation}
P(Z_n)\equiv\left\{\begin{array}{ll}
2\zeta(2^{\{(n-3)/2\}},3,2^{\{(n-3)/2\}})-2\zeta(2^{\{(n-5)/2\}},3,2^{\{(n-1)/2\}}),&\hbox{$n$ odd,}\\
2\zeta(2^{\{(n-4)/2\}},3,2^{\{(n-2)/2\}})-2\zeta(2^{\{(n-2)/2\}},3,2^{\{(n-4)/2\}}),&\hbox{$n$ even.}
\end{array}\right.
\end{equation}
\end{prop}
\begin{proof}
For odd $n$ we use the first word $w$ in (1) of lemma \ref{phi4lemma}.
Because by (\ref{zetaw}) $\zeta_{\widetilde{w}01w0}=-\zeta(2^{\{(n-5)/2\}},3,2^{\{(n-1)/2\}})$ and
$\zeta_{\widetilde{w}10w0}=-\zeta(2^{\{(n-3)/2\}},3,2^{\{(n-3)/2\}})$ the result follows.
For even $n$ we use the second word $w$ in (1) of lemma \ref{phi4lemma}. We use (\ref{Zid}) to reverse the order and simultaneously swap 0 and 1 in
$\widetilde{w}01w0$ and $\widetilde{w}10w0$. This gives an overall minus sign.
\end{proof}
\begin{cor}\label{modprod}
The zig-zag conjecture \ref{zzcon} holds modulo products.
\end{cor}
\begin{proof}
From \cite{Zagier} (see also \cite{Li}) we have an explicit formula for MZVs of 2s with a single 3 in terms of single zetas. Modulo products the result
simplifies to
$$
\zeta(2^{\{a\}},3,2^{\{b\}})\equiv2(-1)^r\left[\binom{2r}{2a+2}-(1-2^{-2r})\binom{2r}{2b+1}\right]\zeta(2r+1),
$$
where $r=a+b+1$. With proposition \ref{modprop} the theorem follows.
\end{proof}
\begin{cor}
The zig-zag conjecture \ref{zzcon} holds for $n\leq13$ loops.
\end{cor}
\begin{proof}
For $n\leq12$ loops the zig-zag period is of weight $\leq21$. A direct calculation of the zig-zag period with {\tt polylog\_procedures} can be reduced
to a single zeta using {\tt zeta\_procedures}. For a weight 23 MZV this is not possible. However, for $n=13$ we can show with {\tt zeta\_procedures}
that the reduced coaction $\Delta' x=\Delta x-1\otimes x-x\otimes 1$ of $P(Z_{13})$ vanishes. Due to corollary \ref{modprod} this is equivalent to proving
the zig-zag conjecture for $n=13$.
\end{proof}
In \cite{ZZ} the zig-zag conjecture is proved in general.
\begin{cor}
The periods of the $A$ and $B$ type $\phi^4$ graphs are in $\sH^\sv(\ZZ)_{2m+2n+5}$ and given modulo $\sH_{>0}^2$ by MZVs of a string of 2s with 1s or 3s in three slots.
\end{cor}
\begin{proof}
Theorem \ref{PGthm} together with lemma \ref{phi4lemma}.
\end{proof}
To illustrate the above corollary we consider the following example.

\begin{ex}
The type $A$ period for even arguments is given modulo products by
\begin{eqnarray}\label{Aexample}
&&\hspace{2cm}P(A_{2m,2n})\equiv\\
&&\hspace{-18pt}\zeta(2^{\{n-1\}},3,2^{\{m\}},3,2^{\{m-1\}},3,2^{\{n\}})+\zeta(2^{\{n-1\}},3,2^{\{m\}},3,2^{\{m-1\}},1,2^{\{n+1\}})\nonumber\\
&&\hspace{-18pt}-\,\zeta(2^{\{n-1\}},3,2^{\{m-1\}},3,2^{\{m\}},3,2^{\{n\}})-\zeta(2^{\{n-1\}},3,2^{\{m-1\}},3,2^{\{m\}},1,2^{\{n+1\}})\nonumber\\
&&\hspace{-18pt}+\,\zeta(2^{\{n\}},1,2^{\{m\}},3,2^{\{m-1\}},3,2^{\{n\}})+\zeta(2^{\{n\}},1,2^{\{m\}},3,2^{\{m-1\}},1,2^{\{n+1\}})\nonumber\\
&&\hspace{-18pt}-\,\zeta(2^{\{n\}},1,2^{\{m-1\}},3,2^{\{m\}},3,2^{\{n\}})-\zeta(2^{\{n\}},1,2^{\{m-1\}},3,2^{\{m\}},1,2^{\{n+1\}}).\nonumber
\end{eqnarray}
\end{ex}

\noindent{\bf Summary.}
Completed primitive $\phi^4$ graphs with vertex connectivity three reduce to products of lower loop order graphs \cite{SchnetzCensus}. For irreducible $\phi^4$ periods we have:

The zig-zag periods exhaust the irreducible primitive $\phi^4$ periods up to 5 loops. Their periods are proved to all orders in \cite{ZZ} as (\ref{PZ}).
The type $A$ and $B$ periods start to differ from the zig-zags at 6 loops where $A_{2,0}=B_{2,0}=P_{6,2}$ in \cite{SchnetzCensus}. The type $A$ and $B$ periods
can be calculated and reduced to a standard MZV basis up to 12 loops with {\tt polylog\_procedures}. The first constructible periods which are neither zig-zags
nor type $A$ or $B$ arise at nine loops. Constructible periods can be calculated and reduced to a standard basis up to eleven loops with the {\tt period} command
in \cite{Polylogproc} or in \cite{Hyperlogproc}. They are in $\sH^\sv(\ZZ)$.

MZV periods (possibly not in $\sH^\sv(\ZZ)$) which are not constructible arise first at six loops as $P_{6,3}$ and $P_{6,4}$ in \cite{SchnetzCensus}.
The period $P_{6,3}$ was calculated by E. Panzer in 2012 by implementing the theory of F. Brown on integration in parametric space \cite{BrFeyn}.
The period $P_{6,4}$ was calculated in \cite{S3}. At seven loops exact numerical methods showed that (at least) eight of the nine periods are MZVs.
Their results are in \cite{BK}, \cite{SchnetzCensus} and \cite{B3}. The missing seven loop period features extensions of MZVs by sixth roots of unity \cite{coact}.
At eight loops exact numerical methods showed that at least 16 out of at most 31 periods are MZVs \cite{SchnetzCensus}. For recent results see \cite{coact}.

Conjectured non-MZV type periods start at eight loops where four periods with higher dimensional geometries were found. They are
$P_{8,37}$, $P_{8,38}$, $P_{8,39}$, $P_{8,41}$ in \cite{SchnetzCensus}. All geometries are modular \cite{modphi4}. The case $P_{8,37}$ was
studied in detail in \cite{K3}. Its period is by standard transcendentality conjectures not of MZV type \cite{BD}.

Periods with geometries which are not modular (of small level) first appear at nine loops. Their periods are conjectured to be not of MZV-type.
At very high loop order most periods are of this type.

\section{Graphical functions in two dimensions}\label{2d}
In this section we cursorily discuss graphical functions in two dimensions.
In $d>2$ dimensions we used massless bosonic propagators $||x-y||^{2-d}$ to define graphical functions. In two dimensions the massless boson propagator
$\ln||x-y||$ has an infrared singularity for $x,y\to\infty$. We resort to using a fermion type propagator in two dimensions.
\subsection{Definition}
In two dimensions we can define graphical functions of several complex variables.
\begin{defn}
Let $G$ be a graph with vertices $V$ labeled by complex variables and two types of directed edges which we call holomorphic and antiholomorphic (see figure 16).
Assign to a holomorphic edge $e$ from $x$ to $y$ the propagator $P_e=(x-y)^{-1}$ and to an antiholomorphic edge $f$ from $x$ to $y$ the propagator
$P_f=(\overline{x}-\overline{y})^{-1}$. The graphical function associated to $G$ is defined as the integral over the internal vertices $V^{\mathrm{int}}\subset V$
of the product of propagators. It depends on the external vertices $\{z_1,\ldots,z_2\}=V^{\mathrm{ext}}=V\backslash V^{\mathrm{int}}$,
\begin{equation}
f^{(0)}_{G}(z_1,\ldots,z_2)=\left(\prod_{v\in V^{\mathrm{int}}} \int_{\CC}\frac{\dd^2x_v}{\pi}\right) \prod_eP_e
\end{equation}
if the integral on the right hand side exists.
\end{defn}
\begin{center}
\fcolorbox{white}{white}{
  \begin{picture}(345,46) (12,-28)
    \SetWidth{0.8}
    \SetColor{Black}
    \Line[arrow,arrowpos=0.5,arrowlength=5,arrowwidth=2,arrowinset=0.2](48,-6)(96,-6)
    \Text(46,-21)[lb]{\normalsize{\Black{$x$}}}
    \Text(94,-23)[lb]{\normalsize{\Black{$y$}}}
    \Text(128,-18)[lb]{\normalsize{\Black{$\displaystyle\frac{1}{x-y}$}}}
    \Vertex(48,-6){2.8}
    \Vertex(96,-6){2.8}
    \Line[arrow,arrowpos=0.5,arrowlength=5,arrowwidth=2,arrowinset=0.2,double,sep=2](208,-6)(256,-6)
    \Vertex(208,-6){2.8}
    \Vertex(256,-6){2.8}
    \Text(206,-21)[lb]{\normalsize{\Black{$x$}}}
    \Text(254,-23)[lb]{\normalsize{\Black{$y$}}}
    \Text(290,-18)[lb]{\normalsize{\Black{$\displaystyle\frac{1}{\overline{x}-\overline{y}}$}}}
  \end{picture}
}
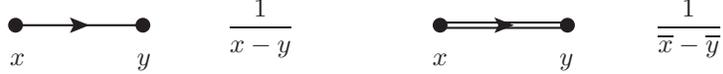
Figure 16: In two dimensions we define holomorphic $\bullet \hspace{-1ex}-\hspace{-1ex}-\hspace{-1ex}\bullet$
and antiholomorphic $\bullet \hspace{-1.3ex}=\hspace{-.5ex}=\hspace{-1.3ex}\bullet$ propagators.
\end{center}
\vskip2ex

By power counting the existence of the two-dimensional graphical function is equivalent to the existence of the four-dimensional analog (see lemma \ref{finitelem}).
Note that in the above definition the integration domain intersects the singular locus of the integrand. The integral may still exist because in the generic
case the singularity is of sufficiently low order.

The case of one internal vertex can be treated explicitly.
\begin{prop}\label{2dprop}
Let $a,b\in\CC$ and $R>\max\{|a|,|b|\}$, then
\begin{eqnarray}\label{2dint}
\frac{1}{\pi}\int_{|x|<R}\frac{1}{x-a}\dd^2x&=&-\overline{a},\nonumber\\
\frac{1}{\pi}\int_{|x|<R}\frac{1}{x-a}\frac{1}{\overline{x}-\overline{b}}\dd^2x&=&\ln\frac{R^2}{|a-b|^2}+O(R^{-2}).
\end{eqnarray}
\end{prop}
\begin{proof}
In the second identity we may assume without restriction that $|a|<|b|$. We fix a small $\epsilon>0$ and calculate the integral over the domain
$D=\{|x|<|a|-\epsilon\}\cup\{|a|+\epsilon<|x|<|b|-\epsilon\}\cup\{|b|+\epsilon<|x|<R\}$. We split the integral according to the three components of $D$,
$$
\frac{1}{\pi}\int_D\frac{1}{x-a}\frac{1}{\overline{x}-\overline{b}}\dd^2x=I_1+I_2+I_3,
$$
Taylor expansion and integrating first over the angle in angular coordinates gives for the three integrals in the limit $\epsilon\to0$
$$
I_1=-\ln\left(1-\frac{\overline{a}}{\overline{b}}\right),\quad I_2=0,\quad I_3=\ln\frac{R^2}{b\overline{b}}-\ln\left(1-\frac{a}{b}\right)+O(R^{-2}).
$$
The first identity is proved by the same method.
\end{proof}
\begin{lem}\label{2dlem}
Let $a_i,b_j\in\CC$ for $i=1,\ldots,m$, $j=1,\ldots,n$, $n\geq2$, $m+n\geq3$, such that none of the $a_i$ is collinear with any of the $b_j$ then
\begin{eqnarray}
\frac{1}{\pi}\int_\CC\frac{1}{\prod_{i=1}^m(x-a_i)}\frac{1}{\prod_{j=1}^n(\overline{x}-\overline{b_j})}\dd^2 x
&\!=&\!\sum_{i=1}^m\left(\prod_{k\neq i}\frac{1}{a_i-a_k}\right)\int_{\overline{a_i}}^\infty\frac{1}{\prod_{j=1}^n(\overline{x}-\overline{b_j})}\dd x\nonumber\\
&\!-&\!\sum_{j=1}^n\left(\prod_{k\neq j}\frac{1}{\overline{b_j}-\overline{b_k}}\right)\int_0^{b_j}\frac{1}{\prod_{i=1}^m(x-a_i)}\dd x,
\end{eqnarray}
where the integral contours are straight lines which originate from 0.
\end{lem}
\begin{proof}
Let $m\geq1$. If we replace the integral to infinity by an integral to $R$ a partial fraction decomposition yields for the right hand side
$$
\sum_{i=1}^m\left(\prod_{k\neq i}\frac{1}{a_i-a_k}\right)\sum_{j=1}^n\left(\prod_{k\neq j}\frac{1}{\overline{b_j}-\overline{b_k}}\right)\left(\ln\frac{R^2}{|a_i-b_j|^2}
+\ln\frac{a_i}{R}\right).
$$
The term $\ln a_i/R$ drops out because $\sum_{j=1}^n\left(\prod_{k\neq j}\frac{1}{\overline{b_j}-\overline{b_k}}\right)=\lim_{\overline{x}\to\infty}\overline{x}
\prod_{j=1}^n\frac{1}{\overline{x}-\overline{b_j}}=0$. The term $\ln R^2/|a_i-b_j|^2$ gives the left hand side due to proposition \ref{2dprop}.

The case $m=0$ reduces by a partial fraction decomposition to the complex conjugate of the first identity in (\ref{2dint}).
\end{proof}

\begin{center}
\fcolorbox{white}{white}{
  \begin{picture}(345,101) (7,11)
    \SetWidth{0.8}
    \SetColor{Black}
    \Vertex(10,61){2.8}
    \Vertex(35,61){2.8}
    \Vertex(35,86){2.8}
    \Vertex(35,36){2.8}
    \Line[arrow,arrowpos=0.5,arrowlength=5,arrowwidth=2,arrowinset=0.2,double,sep=2](10,61)(35,86)
    \Line[arrow,arrowpos=0.5,arrowlength=5,arrowwidth=2,arrowinset=0.2,double,sep=2](10,61)(35,61)
    \Line[arrow,arrowpos=0.5,arrowlength=5,arrowwidth=2,arrowinset=0.2,double,sep=2](10,61)(35,36)
    \Vertex(65,61){2.8}
    \Vertex(90,61){2.8}
    \Vertex(115,86){2.8}
    \Vertex(115,36){2.8}
    \Line[arrow,arrowpos=0.5,arrowlength=5,arrowwidth=2,arrowinset=0.2](90,61)(65,61)
    \Line[arrow,arrowpos=0.5,arrowlength=5,arrowwidth=2,arrowinset=0.2,double,sep=2](90,61)(115,86)
    \Line[arrow,arrowpos=0.5,arrowlength=5,arrowwidth=2,arrowinset=0.2,double,sep=2](90,61)(115,36)
    \Line[arrow,arrowpos=0.5,arrowlength=5,arrowwidth=2,arrowinset=0.2,flip,double,sep=2](145,86)(170,61)
    \Line[arrow,arrowpos=0.5,arrowlength=5,arrowwidth=2,arrowinset=0.2,flip](195,61)(170,61)
    \Line[arrow,arrowpos=0.5,arrowlength=5,arrowwidth=2,arrowinset=0.2,double,sep=2](195,61)(220,86)
    \Line[arrow,arrowpos=0.5,arrowlength=5,arrowwidth=2,arrowinset=0.2,flip](250,86)(275,61)
    \Line[arrow,arrowpos=0.5,arrowlength=5,arrowwidth=2,arrowinset=0.2,flip,double,sep=2](300,61)(275,61)
    \Line[arrow,arrowpos=0.5,arrowlength=5,arrowwidth=2,arrowinset=0.2](300,61)(325,86)
    \Line[arrow,arrowpos=0.5,arrowlength=5,arrowwidth=2,arrowinset=0.2,double,sep=2](195,61)(220,36)
    \Line[arrow,arrowpos=0.5,arrowlength=5,arrowwidth=2,arrowinset=0.2,flip,double,sep=2](145,36)(170,61)
    \Line[arrow,arrowpos=0.5,arrowlength=5,arrowwidth=2,arrowinset=0.2,flip,double,sep=2](250,36)(275,61)
    \Line[arrow,arrowpos=0.5,arrowlength=5,arrowwidth=2,arrowinset=0.2,double,sep=2](300,61)(325,36)
    \Vertex(145,36){2.8}
    \Vertex(145,86){2.8}
    \Vertex(170,61){2.8}
    \Vertex(195,61){2.8}
    \Vertex(220,36){2.8}
    \Vertex(220,86){2.8}
    \Vertex(250,86){2.8}
    \Vertex(250,36){2.8}
    \Vertex(275,61){2.8}
    \Vertex(300,61){2.8}
    \Vertex(325,86){2.8}
    \Vertex(325,36){2.8}
    \Text(33,91)[lb]{\normalsize{\Black{$b_1$}}}
    \Text(33,66)[lb]{\normalsize{\Black{$b_2$}}}
    \Text(33,41)[lb]{\normalsize{\Black{$b_3$}}}
    \Text(62,66)[lb]{\normalsize{\Black{$a_1$}}}
    \Text(112,91)[lb]{\normalsize{\Black{$b_1$}}}
    \Text(113,41)[lb]{\normalsize{\Black{$b_2$}}}
    \Text(142,91)[lb]{\normalsize{\Black{$b_1$}}}
    \Text(140,41)[lb]{\normalsize{\Black{$b_2$}}}
    \Text(217,41)[lb]{\normalsize{\Black{$b_3$}}}
    \Text(218,91)[lb]{\normalsize{\Black{$b_4$}}}
    \Text(246,91)[lb]{\normalsize{\Black{$a_1$}}}
    \Text(245,41)[lb]{\normalsize{\Black{$b_1$}}}
    \Text(323,41)[lb]{\normalsize{\Black{$b_2$}}}
    \Text(322,91)[lb]{\normalsize{\Black{$a_2$}}}
    \Text(15,16)[lb]{\normalsize{\Black{$G_1$}}}
    \Text(85,16)[lb]{\normalsize{\Black{$G_2$}}}
    \Text(177,16)[lb]{\normalsize{\Black{$G_3$}}}
    \Text(284,16)[lb]{\normalsize{\Black{$G_4$}}}
  \end{picture}
}
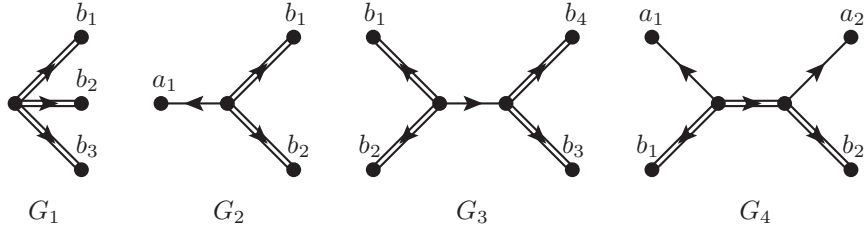
Figure 17: Examples of graphical functions with one or two internal vertices.
\end{center}
\vskip2ex

\begin{ex}
For the graphs $G_1$ and $G_2$ in figure 17 one obtains
\begin{eqnarray}
f_{G_1}&=&\frac{b_1(\overline{b_2}-\overline{b_3})+b_2(\overline{b_3}-\overline{b_1})+b_3(\overline{b_1}-\overline{b_2})}
{(\overline{b_1}-\overline{b_2})(\overline{b_2}-\overline{b_3})(\overline{b_3}-\overline{b_1})},\nonumber\\
f_{G_2}&=&\frac{2}{\overline{b_2}-\overline{b_1}}\ln\left|\frac{a_1-b_1}{a_1-b_2}\right|.
\end{eqnarray}
\end{ex}
Note that unlike the higher dimensional case it is possible to iterate the above lemma to evaluate more complicated graphical functions in two dimensions.
This strategy leads to iterated integrals. In the case of two internal vertices one obtains:
\begin{ex}
For the graphs $G_3$ and $G_4$ in figure 17 one obtains
\begin{eqnarray}
f_{G_3}&=&2\frac{\displaystyle b_1\ln\left|\frac{b_1-b_4}{b_1-b_3}\right|-b_2\ln\left|\frac{b_2-b_4}{b_2-b_3}\right|
 +b_3\ln\left|\frac{b_3-b_1}{b_3-b_2}\right|-b_4\ln\left|\frac{b_4-b_1}{b_4-b_2}\right|}{(\overline{b_1}-\overline{b_2})(\overline{b_3}-\overline{b_4})},\nonumber\\
f_{G_4}&=&\frac{2}{\overline{b_1}-\overline{b_2}}\left(\mathrm{i}D\left(\frac{a_1-b_1}{a_1-a_2}\right)+\mathrm{i}D\left(\frac{a_2-b_2}{a_2-a_1}\right)\right.\nonumber\\
&&\quad\quad+\;\left.\ln\left|\frac{a_1-b_1}{a_1-a_2}\right|\ln\left|\frac{a_2-b_2}{a_2-b_1}\right|
+\ln\left|\frac{a_1-b_1}{a_1-b_2}\right|\ln\left|\frac{a_2-b_2}{a_2-a_1}\right|\right),
\end{eqnarray}
where $D$ is the Bloch-Wigner dilogarithm (\ref{BWdilog}).
\end{ex}
In general, the minimum of the number of holomorphic and the number of antiholomorphic edges is an upper bound for the weight of the graphical function.

\subsection{Completion}
Now we return to the situation where we have three labeled vertices 0, 1, $z$. To formulate a completion theorem we introduce a vertex $\infty$.
We define propagators that emanate from $\infty$ as 1 and propagators that lead into $\infty$ as $-1$.

\begin{center}
\fcolorbox{white}{white}{
  \begin{picture}(345,46) (8,-28)
    \SetWidth{0.8}
    \SetColor{Black}
    \Line[arrow,arrowpos=0.5,arrowlength=5,arrowwidth=2,arrowinset=0.2](8,-6)(46,-6)
    \Text(3,-21)[lb]{\normalsize{\Black{$\infty$}}}
    \Text(44,-21)[lb]{\normalsize{\Black{$x$}}}
    \Text(55,-9)[lb]{\normalsize{\Black{$=\,1,$}}}
    \Vertex(8,-6){2.8}
    \Vertex(46,-6){2.8}
    \Line[arrow,arrowpos=0.5,arrowlength=5,arrowwidth=2,arrowinset=0.2,flip](98,-6)(136,-6)
    \Text(93,-21)[lb]{\normalsize{\Black{$\infty$}}}
    \Text(134,-21)[lb]{\normalsize{\Black{$x$}}}
    \Text(145,-9)[lb]{\normalsize{\Black{$=\,-1,$}}}
    \Vertex(96,-6){2.8}
    \Vertex(136,-6){2.8}
    \Line[arrow,arrowpos=0.5,arrowlength=5,arrowwidth=2,arrowinset=0.2,double,sep=2](198,-6)(236,-6)
    \Vertex(198,-6){2.8}
    \Vertex(236,-6){2.8}
    \Text(193,-21)[lb]{\normalsize{\Black{$\infty$}}}
    \Text(234,-21)[lb]{\normalsize{\Black{$x$}}}
    \Text(245,-9)[lb]{\normalsize{\Black{$=\,1,$}}}
    \Line[arrow,arrowpos=0.5,arrowlength=5,arrowwidth=2,arrowinset=0.2,double,flip,sep=2](288,-6)(326,-6)
    \Vertex(288,-6){2.8}
    \Vertex(326,-6){2.8}
    \Text(283,-21)[lb]{\normalsize{\Black{$\infty$}}}
    \Text(324,-21)[lb]{\normalsize{\Black{$x$}}}
    \Text(335,-9)[lb]{\normalsize{\Black{$=\,-1.$}}}
  \end{picture}
}
Figure 18: Depending on the orientation propagators to $\infty$ are $\pm1$.
\end{center}
\vskip2ex

We define the weight of holomorphic propagators as (1,0) and the weight of antiholomorphic propagators as (0,1).

\begin{defn}
A graph $\Gamma$ is completed in $d=2$ dimensions if it has the labels $0,1,z,\infty$ and weights such that every unlabeled (internal) vertex has weighted valence (2,2)
and every labeled (external) vertex has valence (0,0). The graphical function of $\Gamma$ is
\begin{equation}\label{fGamma2dim}
f_\Gamma^{(0)}(z)=\left(\prod_{v\notin\{0,1,z,\infty\}} \int_{\CC}\frac{\dd^2x_v}{\pi}\right)\prod_eP_e^{\nu_e},
\end{equation}
where the products are over vertices and edges, respectively. The integer $\nu_e$ is the total weight of the edge $e$.
\end{defn}
Analogously to lemma \ref{completionlemma} we may uniquely complete a graph by adding (possibly inverse) propagators $+1$
(i.e.\ adding propagators from 1 to 0 and propagators emanating from $\infty$). The graphical function does not change under completion.

For the $\sS_4$ group of permutation of external vertices we define a transformation $\pi=\phi(\sigma)\circ\sigma$ as a permutation $\sigma\in\sS_4$ followed by a transformation
of the label $z$ as in figure 9. In complete analogy to theorem \ref{completionthm} we have the following theorem:

\begin{thm}\label{completionthm2d}
The completed graphical function is invariant under $\pi=\phi(\sigma)\circ\sigma$,
\begin{equation}
f^{(0)}_\Gamma(z)=f^{(0)}_{\sigma(\Gamma)}(\phi(\sigma)(z)).
\end{equation}
\end{thm}
\begin{proof}
The proof is analogous to the proof of theorem \ref{completionthm} with extra signs due to the orientation of the edges.

The transformation $x\to1-x$ followed by swapping the labels 0 and 1 and changing the label $z$ to $1-z$ contributes with a sign $(-1)^{N_\Gamma-N_\infty}$ where
$N_\Gamma$ counts all propagators (holomorphic and antiholomorphic, normal and inverse) and $N_\infty$ counts all propagators connected to $\infty$. Because
$\infty$ has weight (0,0) the number of normal propagators equals the number of inverse propagators at $\infty$. Hence $N_\infty$ is even. Because every internal vertex has weight (2,2)
counting holomorphic minus inverse holomorphic half-edges gives that the number of internal vertices is the difference of the number of holomorpic and the number of inverse
holomorphic edges,
$$
V^{\mathrm{int}}=N_--N_\sim.
$$
The same identity holds for antiholomorphic edges so that the total number of edges $N_\Gamma$ is even.

The transformation $x\to1/x$ followed by swapping the labels 0 and $\infty$ and changing the label $z$ to $1/z$ changes the sign of every propagator.
Because $N_\Gamma$ is even the overall sign does not change.

The transformation $x\to zx$ followed by changing $z$ to 1 and 1 to $1/z$ does not change the signs of propagators.
\end{proof}

Due to the formula $\partial_z\zz^{-1}=\partial_\zz z^{-1}=\pi\delta^{(2)}(z)$ edges can be appended to two-dimensional graphical functions in much the same
way as in higher dimensions. Moreover one can construct two-dimensional graphical functions from graphical functions in many variables by equating external labels. This leads
to the following conjecture.

\begin{con}\label{2dcon}
For any graph $\Gamma$ the two-dimensional graphical function $f_\Gamma^{(0)}$ is in $\sA^\sv$ (see definition \ref{Adef}).
\end{con}

\subsection{Periods}
Periods of two-dimensional graphical functions arise from completed primitive graphs, which are (2,2)-regular internally 6-connected graphs $\Gamma$. We label the vertices of
$\Gamma$ by $0,1,\infty,v_1,\ldots,v_{V-3}$ and obtain in analogy to (\ref{Pdef}) that the period
\begin{equation}\label{Pdef2d}
P(\Gamma)=\left(\prod_{v\notin\{0,1,\infty\}} \int_{\CC}\frac{\dd^2x_v}{\pi}\right)\prod_eP_e
\end{equation}
is independent of the labeling. We may obtain the period of a graph $\Gamma$ by introducing an extra label $z$ and integrating the graphical function over $z$ with $\pi^{-1}\int_{\CC}\dd^2 z$.
A consequence of theorem \ref{Stabilitythm} and conjecture \ref{2dcon} is the following conjecture.

\begin{con}\label{2dcon2}
The period of any two-dimensional completed primitive graph is in $\sH^\sv$ (see definition \ref{Rsvdef}).
\end{con}

We may ask if the periods of two-dimensional completed primitive graphs are in the ring $\sH^\sv(\ZZ)$.

\subsection{Cell zeta values}
A two-dimensional primitive completed graph $\Gamma$ is (2,2)-regular. Holomorphic (and antiholomorphic) propagators form a union of cycles in $\Gamma$.
If $\Gamma$ has at least six vertices then $\Gamma$ typically has several disjoint cycles. If there exists only a single Hamiltonian holomorphic (or antiholomorphic) cycle
we may assign to $\Gamma$ a cell zeta value by the following definition (see figure 19). In \cite{cellzetas} it is proved that zell zeta values are MZVs.

\begin{defn}
Let $\Gamma$ be a (2,2)-regular graph with a holomorphic Hamiltonian cycle. Label the vertices of $\Gamma$ by $0,1,\infty,x_1,\ldots,x_{V-3}$. If the Hamiltonian cycle is split
at $\infty$ it gives rise to an ordering $\Sigma$ of the labels by $x_{i_1}<\ldots<x_{i_k}<0<x_{i_{k+1}}<\ldots<x_{i_\ell}<1<x_{i_{\ell+1}}<\ldots<x_{i_m}<\infty$.
The sign $\sigma$ of $\Sigma$ is $(-1)^N$ where $N$ is the number of times that an edge points from the smaller value to the bigger value (we assume $\infty>x_i$ for all $i$).
An antiholomorphic edge $e$ from $x_i$ to $x_j$ gives rise to the (real-valued) propagator $P_e=(x_i-x_j)^{-1}$. An edge from $\infty$ to $x_i$ is 1 and an edge from $x_i$ to $\infty$ is $-1$.
The cell zeta value of the graph is
\begin{equation}
P_-(\Gamma)=\sigma\int_\Sigma\;\prod_{e\ \mathrm{antihol}}P_e\prod_{i=1}^{V-3}\dd x_i,
\end{equation}
where the first product is over all antiholomorphic edges $e$.

Interchanging the role of holomorphic and antiholomorphic edges we define $P_=(\Gamma)$ if $\Gamma$ has an antiholomorphic Hamiltonian cycle.
\end{defn}
\begin{center}
\fcolorbox{white}{white}{
  \begin{picture}(345,141) (-40,0)
    \SetWidth{0.8}
    \SetColor{Black}
    \Line[arrow,arrowpos=0.35,arrowlength=5,arrowwidth=2,arrowinset=0.2,double,sep=2](170,91)(260,41)
    \Vertex(85,26){2.8}
    \Vertex(55,116){2.8}
    \Vertex(5,81){2.8}
    \Vertex(25,26){2.8}
    \Vertex(105,81){2.8}
    \Vertex(215,116){2.8}
    \Vertex(215,16){2.8}
    \Vertex(170,91){2.8}
    \Vertex(170,41){2.8}
    \Vertex(260,41){2.8}
    \Vertex(260,91){2.8}
    \Line[arrow,arrowpos=0.5,arrowlength=5,arrowwidth=2,arrowinset=0.2](55,116)(105,81)
    \Line[arrow,arrowpos=0.5,arrowlength=5,arrowwidth=2,arrowinset=0.2](105,81)(85,26)
    \Line[arrow,arrowpos=0.5,arrowlength=5,arrowwidth=2,arrowinset=0.2](85,26)(25,26)
    \Line[arrow,arrowpos=0.5,arrowlength=5,arrowwidth=2,arrowinset=0.2](25,26)(5,81)
    \Line[arrow,arrowpos=0.5,arrowlength=5,arrowwidth=2,arrowinset=0.2](5,81)(55,116)
    \Line[arrow,arrowpos=0.5,arrowlength=5,arrowwidth=2,arrowinset=0.2,double,sep=2](55,116)(25,26)
    \Line[arrow,arrowpos=0.5,arrowlength=5,arrowwidth=2,arrowinset=0.2,double,sep=2](25,26)(105,81)
    \Line[arrow,arrowpos=0.5,arrowlength=5,arrowwidth=2,arrowinset=0.2,double,sep=2](105,81)(5,81)
    \Line[arrow,arrowpos=0.5,arrowlength=5,arrowwidth=2,arrowinset=0.2,double,sep=2](5,81)(85,26)
    \Line[arrow,arrowpos=0.5,arrowlength=5,arrowwidth=2,arrowinset=0.2,double,sep=2](85,26)(55,116)
    \Line[arrow,arrowpos=0.5,arrowlength=5,arrowwidth=2,arrowinset=0.2](215,116)(260,91)
    \Line[arrow,arrowpos=0.5,arrowlength=5,arrowwidth=2,arrowinset=0.2](260,91)(260,41)
    \Line[arrow,arrowpos=0.5,arrowlength=5,arrowwidth=2,arrowinset=0.2](260,41)(215,16)
    \Line[arrow,arrowpos=0.5,arrowlength=5,arrowwidth=2,arrowinset=0.2](215,16)(170,41)
    \Line[arrow,arrowpos=0.5,arrowlength=5,arrowwidth=2,arrowinset=0.2](170,41)(170,91)
    \Line[arrow,arrowpos=0.5,arrowlength=5,arrowwidth=2,arrowinset=0.2](170,91)(215,116)
    \Line[arrow,arrowpos=0.3,arrowlength=5,arrowwidth=2,arrowinset=0.2,double,sep=2](215,116)(170,41)
    \Line[arrow,arrowpos=0.65,arrowlength=5,arrowwidth=2,arrowinset=0.2,double,sep=2](170,41)(260,91)
    \Line[arrow,arrowpos=0.7,arrowlength=5,arrowwidth=2,arrowinset=0.2,double,sep=2](260,91)(215,16)
    \Line[arrow,arrowpos=0.3,arrowlength=5,arrowwidth=2,arrowinset=0.2,double,sep=2](215,16)(170,91)
    \Line[arrow,arrowpos=0.7,arrowlength=5,arrowwidth=2,arrowinset=0.2,double,sep=2](260,41)(215,116)
    \Text(50,123)[lb]{\normalsize{\Black{$\infty$}}}
    \Text(210,123)[lb]{\normalsize{\Black{$\infty$}}}
    \Text(2,90)[lb]{\normalsize{\Black{$0$}}}
    \Text(168,100)[lb]{\normalsize{\Black{$0$}}}
    \Text(102,90)[lb]{\normalsize{\Black{$1$}}}
    \Text(258,100)[lb]{\normalsize{\Black{$1$}}}
    \Text(21,13)[lb]{\normalsize{\Black{$x_1$}}}
    \Text(80,13)[lb]{\normalsize{\Black{$x_2$}}}
    \Text(166,28)[lb]{\normalsize{\Black{$x_1$}}}
    \Text(212,3)[lb]{\normalsize{\Black{$x_2$}}}
    \Text(258,28)[lb]{\normalsize{\Black{$x_3$}}}
    \Text(52,3)[lb]{\normalsize{\Black{$\Gamma_1$}}}
    \Text(190,3)[lb]{\normalsize{\Black{$\Gamma_2$}}}
  \end{picture}
}
Figure 19: Assuming conjectures \ref{2dcon2} and \ref{concellzeta} the completed primitive graphs $\Gamma_1$ and $\Gamma_2$ have periods 0 and $4\zeta(3)$, respectively.
\end{center}
\vskip2ex

It is possible to show that the definition is independent of the chosen labeling.

\begin{ex}\label{2dex}
In the situation of figure 19 we have
\begin{eqnarray*}
P_-(\Gamma_1)&=&(-1)^1\int_{0<x_1<x_2<1}(+1)\frac{1}{x_1-1}\frac{1}{1-0}\frac{1}{0-x_2}(-1)\dd x_1\dd x_2=\zeta(2),\\
P_=(\Gamma_1)&=&(-1)^1\int_{x_2<0,1<x_1}(+1)\frac{1}{1-x_2}\frac{1}{x_2-x_1}\frac{1}{x_1-0}(-1)\dd x_1\dd x_2=-\zeta(2),
\end{eqnarray*}
and
\begin{eqnarray*}
&&\hspace{-10pt}P_-(\Gamma_2)\;=\\
&&\hspace{-10pt}(-1)^1\int_{0<x_1<x_2<x_3<1}(+1)\frac{1}{x_1-1}\frac{1}{1-x_2}\frac{1}{x_2-0}\frac{1}{0-x_3}(-1)\dd x_1\dd x_2\dd x_3=2\zeta(3),\\
&&\hspace{-10pt}P_=(\Gamma_2)\;=\\
&&\hspace{-10pt}(-1)^1\int_{x_3<0<x_2<1<x_1}(+1)\frac{1}{1-x_3}\frac{1}{x_3-x_2}\frac{1}{x_2-x_1}\frac{1}{x_1-0}(-1)\dd x_1\dd x_2\dd x_3=2\zeta(3).
\end{eqnarray*}
\end{ex}
Iterating lemma \ref{2dlem} leads to the following conjecture.

\begin{con}\label{concellzeta}
Let $\Gamma$ be a two-dimensional completed primitive graph with $V$ vertices. The maximum weight piece of the period $P^{(0)}_{\mathrm{max}}(\Gamma)$ is equivalent
modulo products to the sum of the holomorphic and the antiholomorphic cell zeta values,
\begin{equation}
P^{(0)}_{\mathrm{max}}(\Gamma)\equiv P_-(\Gamma)+P_=(\Gamma)\mod \sH_{>0}^2,
\end{equation}
where the (anti-)holomorphic cell zeta value is 0 if there exists no (anti-)holomorphic Hamiltonian cycle.
\end{con}
\begin{ex}
Because $\sH_{>0}^2$ is trivial up to weight three we obtain for the two graphs in figure 19:
$$
P(\Gamma_1)=0,\quad P(\Gamma_2)=4\zeta(3).
$$
The vanishing period of $\Gamma_1$ is consistent with conjecture \ref{2dcon2} because there exist no even weight generators in $\sH^\sv$.
The period of $\Gamma_2$ is in $\sH^\sv(\ZZ)$ (see example \ref{Rsvex}).
\end{ex}

\begin{appendix}
\section{Proof of proposition \ref{prop2}}
\subsection{Identity (C1)}
We need an estimate for the product of two Gegenbauer polynomials. We have
$$C^{(\lambda)}_k(x)C^{(\lambda)}_\ell(x)=\sum_{m=|k-\ell|}^{k+l}f_{k,\ell}^mC^{(\lambda)}_m(x),$$
where only those terms contribute to the sum where $n=(k+\ell+m)/2$ is an integer.
Dougall's linearization formula \cite{Askey} (p.\ 39) gives
$$f_{k,\ell}^m=\frac{(m+\lambda)m!(2\lambda)_n(\lambda)_{n-k}(\lambda)_{n-\ell}(\lambda)_{n-m}}
{(2\lambda)_m(\lambda)_{n+1}(n-k)!(n-\ell)!(n-m)!},$$
where $(x)_n=\Gamma(x+n)/\Gamma(x)$ is Pochhammer's symbol.
We only need that linearization coefficients are non-negative because in this case we can specialize the argument to $x=1$ and get the estimate
$$f_{k,\ell}^m\leq \frac{C^{(\lambda)}_k(1)C^{(\lambda)}_\ell(1)}{C^{(\lambda)}_m(1)}.$$
From
$$C^{(\lambda)}_n(1)=\binom{n+2\lambda-1}{n}$$
we obtain
$$f_{k,\ell}^m\leq(k+2\lambda-1)^{2\lambda-1}(\ell+2\lambda-1)^{2\lambda-1}\leq[2\lambda(k+\ell)]^{4\lambda-2}.$$
Now, assume there exist $A,\alpha,A',\alpha'>0$, $L_1,L_2\in\NN$ such then $f\in\sC^{(\lambda)}_{p_1,q_1}$ has the expansion
$$
f(z)=\sum_{\ell_1=0}^{L_1}\sum_{m_1=0}^\infty\sum_{n_1=0}^{m_1}a_{\ell_1,m_1,n_1}(\ln||z||)^{\ell_1}||z||^{m_1-p_1}C_{n_1}^{(\lambda)}(\cos\angle(z,e_1))
$$
for $||z||<1$ and $g\in\sC^{(\lambda)}_{p_2,q_2}$ has the expansion
$$
g(z)=\sum_{\ell_2=0}^{L_2}\sum_{m_2=0}^\infty\sum_{n_2=0}^{m_2}a'_{\ell_2,m_2,n_2}(\ln||z||)^{\ell_2}||z||^{m_2-p_2}C_{n_2}^{(\lambda)}(\cos\angle(z,e_1))
$$
for $||z||<1$ with
$$
|a_{\ell_1,m_1,n_1}|\leq Am_1^\alpha\quad\hbox{and}\quad |a'_{\ell_2,m_2,n_2}|\leq A'm_2^{\alpha'}.
$$
Then $fg$ has the expansion
\begin{equation}\label{apeq1}
f(z)g(z)=\sum_{\ell=0}^{L_1+L_2}\sum_{m=0}^\infty\sum_{n=0}^\infty a''_{\ell,m,n}(\ln||z||)^{\ell}||z||^{m-p_1-p_2}C_n^{(\lambda)}(\cos\angle(z,e_1))
\end{equation}
with
$$
a''_{\ell,m,n}=\sum_{\ell_1=0}^\ell\sum_{m_1=0}^m\sum_{n_1=0}^{m_1}\sum_{n_2=0}^{m-m_1} a_{\ell_1,m_1,n_1}a'_{\ell-\ell_1,m-m_1,n_2}f^n_{n_1,n_2}.
$$
We have $n_1+n_2\leq m$ and because $f^n_{n_1,n_2}$ is zero if $n>n_1+n_2$ the sum in $n$ in (\ref{apeq1}) terminates at $m$. Hence (\ref{apeq1}) has the
desired shape and we have only to check that $a''$ is sufficiently bounded. By the previous considerations we obtain
\begin{eqnarray*}
|a''_{\ell,m,n}|&\leq&\sum_{\ell_1=0}^\ell\sum_{m_1=0}^m\sum_{n_1=0}^{m_1}\sum_{n_2=0}^{m-m_1} A m_1^\alpha A' (m-m_1)^{\alpha'}[2\lambda(n_1+n_2)]^{4\lambda-2}\\
&\leq&\sum_{\ell_1=0}^\ell\sum_{m_1=0}^m\sum_{n_1=0}^{m_1}\sum_{n_2=0}^{m-m_1} AA'(2\lambda)^{4\lambda-2} m^{\alpha+\alpha'+4\lambda-2}\\
&\leq&\ell AA'(2\lambda)^{4\lambda-2} m^{\alpha+\alpha'+4\lambda+1}.
\end{eqnarray*}
This proves the boundedness in the case $||z||<1$. The case $||z||>1$ is analogous.

\subsection{Identity (C2)}
Assume there exist $A,\alpha,B,\beta>0$, $L\in\NN$ such that $f\in\sC^{(\lambda)}_{p,q}$ has the expansions
$$f(z)=\sum_{\ell=0}^{L}\sum_{m=0}^\infty\sum_{n=0}^{m}a_{\ell,m,n}(\ln||z||)^{\ell}||z||^{m-p}C_{n}^{(\lambda)}(\cos\angle(z,e_1))$$
for $||z||<1$ and
$$f(z)=\sum_{\ell=0}^{L}\sum_{m=0}^\infty\sum_{n=0}^{m}b_{\ell,m,n}(\ln||z||)^{\ell}||z||^{-m-q}C_{n}^{(\lambda)}(\cos\angle(z,e_1))$$
for $||z||>1$ with
$$
|a_{\ell,m,n}|\leq Am^\alpha\quad\hbox{and}\quad |b_{\ell,m,n}|\leq Bm^{\beta}.
$$
Then
$$
g(z)=\frac{1}{\pi^{d/2}}\int_{\RR^d}\frac{f(x)}{||x||^{2\alpha}||x-z||^{2\lambda}}\dd^dx
$$
can be calculated using expansion (\ref{propexpand}) and the orthogonality (\ref{ortho}) of the Gegenbauer polynomials. A straight forward but tedious calculation gives
for $g(z)$ if $2-q<2\alpha<2\lambda+2-p$ in the case $||z||<1$:
\begin{eqnarray}\label{Gegenbauerint1}
&&\sum_{\ell=0}^L\sum_{m=0}^\infty\sum_{n=0}^m\frac{2}{\Gamma(\lambda)(n+\lambda)}C^{(\lambda)}_n(\cos\angle(z,e_1))\nonumber\\
&&\cdot\,\Big[a_{\ell,m,n}\sum_{k=0}^\ell(-1)^{\ell-k}(\ln||z||)^k\,||z||^{m-p-2\alpha+2}\nonumber\\
&&\cdot\,\left(\frac{1}{(m+n-p+2\lambda-2\alpha+2)^{\ell-k+1}}-\frac{1-\delta_{m-n,p+2\alpha-2}}{(m-n-p-2\alpha+2)^{\ell-k+1}}\right)\nonumber\\
&&+\,\Big(a_{\ell,m,n}(1-\delta_{m-n,p+2\alpha-2})\frac{(-1)^\ell}{(m-n-p-2\alpha+2)^{\ell+1}}\nonumber\\
&&\quad-\,a_{\ell,m,n}\delta_{m-n,p+2\alpha-2}\frac{(\ln||z||)^{\ell+1}}{\ell+1}+b_{\ell,m,n}\frac{1}{(m+n+q+2\alpha-2)^{\ell+1}}\Big)||z||^n\Big],
\end{eqnarray}
where $\delta_{m,n}$ is the Konecker $\delta$. In the case $||z||>1$ we obtain
\begin{eqnarray}\label{Gegenbauerint2}
&&\sum_{\ell=0}^L\sum_{m=0}^\infty\sum_{n=0}^m\frac{2}{\Gamma(\lambda)(n+\lambda)}C^{(\lambda)}_n(\cos\angle(z,e_1))\nonumber\\
&&\cdot\,\Big[b_{\ell,m,n}\sum_{k=0}^\ell(\ln||z||)^k\,||z||^{-m-q-2\alpha+2}\nonumber\\
&&\cdot\,\left(\frac{1}{(m+n+q+2\alpha-2)^{\ell-k+1}}-\frac{1-\delta_{m-n,-q+2\lambda-2\alpha+2}}{(m-n+q-2\lambda+2\alpha-2)^{\ell-k+1}}\right)\nonumber\\
&&+\,\Big(b_{\ell,m,n}(1-\delta_{m-n,-q+2\lambda-2\alpha+2})\frac{1}{(m-n+q-2\lambda+2\alpha-2)^{\ell+1}}\nonumber\\
&&\quad+\,b_{\ell,m,n}\delta_{m-n,-q+2\lambda-2\alpha+2}\frac{(\ln||z||)^{\ell+1}}{\ell+1}\nonumber\\
&&\quad+\,a_{\ell,m,n}\frac{(-1)^\ell}{(m+n-p+2\lambda-2\alpha+2)^{\ell+1}}\Big)||z||^{-n-2\lambda}\Big].
\end{eqnarray}

Every term in the square brackets is in $\sC^{(\lambda)}_{\max\{0,p+2\alpha-2\},\min\{2\lambda,q+2\alpha-2\}}$ and hence
$g\in\sC^{(\lambda)}_{\max\{0,p+2\alpha-2\},\min\{2\lambda,q+2\alpha-2\}}$.
\end{appendix}

\bibliographystyle{plain}
\renewcommand\refname{References}

\end{document}